\documentclass[a4paper,12pt]{amsart}
\title[Tropical curves with parallel rays]{Tropical curves with parallel rays}

\author{Song JuAe}
\address{Faculty of Mathematics, Kyushu University, 744 Motooka, Nishi-ku, Fukuoka, 819-0395, Japan.}
\email{songjuae@math.kyushu-u.ac.jp}


\usepackage{xcolor}
\definecolor{cadmiumgreen}{rgb}{0.0, 0.42, 0.24}
\usepackage{amsmath,amssymb,amscd}
\usepackage{amsthm}
\usepackage{subfigure}
\usepackage{mathrsfs}
\usepackage{appendix}
\usepackage{tikz}
\usepackage{geometry}
\usepackage[
    colorlinks=true,
    linkcolor=blue,
    urlcolor=blue,
    citecolor=cadmiumgreen,
]{hyperref}

\usetikzlibrary{arrows.meta}

\newtheorem{dfn}{Definition}[section]
\newtheorem{thm}[dfn]{Theorem}
\newtheorem{prop}[dfn]{Proposition}
\newtheorem{cor}[dfn]{Corollary}

\newtheorem{rem}[dfn]{Remark}
\newtheorem{ex}[dfn]{Example}

\newtheorem{prom}[dfn]{Promise}

\def\Gamma{\varGamma}
\def\t{\,{}^t\!}

\begin{document}

\begin{abstract}
In the previous works, the rational function semifields of abstract tropical curves were characterized.
In this paper, we give a contravariant categorical equivalence between the category of abstract tropical curves with morphisms and the category of semifields over the tropical semifield $\boldsymbol{T}$ characterized above with $\boldsymbol{T}$-algebra homomorphisms.

The characterization tells us that the traditional definition of abstract tropical curves has a fatal flaw such that we are never able to deal with parallel rays, unlike the traditional tropical curves, which generally admit them.
To address this flaw, we introduce a new notion of abstract tropical curves with parallel rays.
Then we define the rational function semifields of these curves and give a characterization of them, and a variant of the categorical equivalence between their categories with a suitable notion of morphisms between these curves.

Under the categorical equivalences, we translate several geometric notions for traditional or abstract tropical curves (with parallel rays) into algebraic ones, including weights on edges and the balancing condition.
\end{abstract}

\maketitle

\tableofcontents

\section{Introduction}
	\label{section1}

\subsection*{Background}

The definition of abstract tropical curves we give in Subsection~\ref{subsection2.4} is very classical in tropical geometry.
These are (equivalence classes of) graphs endowed with a metric.
This definition traces its origins to \cite{Kirchhoff}, with a version more closely resembling our own appearing in \cite{Roth}.
Our definition of abstract tropical curves was also introduced in \cite{Kuchment} under the name of metric graphs within the framework of quantum graphs.
The name ``abstract tropical curve" was given in \cite{Mikhalkin=Zharkov}, and in \textit{loc.~cit.} and \cite{Gathmann=Kerber}, using the Riemann--Roch theorem for finite graphs, which had been established earlier in \cite{Baker=Norine}, the Riemann--Roch theorem for abstract tropical curves was proven independently.
This made abstract tropical curves attracted particular interest as objects of study in tropical geometry.

By Kapranov’s theorem (or its generalization, the fundamental theorem of tropical algebraic geometry, see \cite[Subsections~3.1 and 3.2]{Maclagan=Sturmfels}), tropicalizations of algebraic curves or tropical plane curves, that is, objects that had traditionally been called tropical curves, can be viewed as abstract tropical curves by lattice length.
This justifies the name ``abstract tropical curves".

Another justification can be found in situations where tropicalization is considered via Berkovich spaces, as in \cite{Amini=Baker=Brugalle=Rabinoff}.
In such cases as well, it is natural to call the support of a connected one-dimensional polyhedral complex equipped with a metric an abstract tropical curve.

On the other hand, it should be noted that in this abstraction from classical tropical curves via lattice length, information such as edge weights and the balancing condition inherent to classical tropical curves is discarded.
Furthermore, phenomena that can never occur in classical tropical curves due to the balancing condition—such as the existence of leaves of finite length or the absence of leaves—are permitted within the framework of abstract tropical curves.
It seems that the motivation for introducing this broader framework, compared with the abstraction of classical tropical curves using lattice lengths, largely stems from the proof of the Riemann--Roch theorem mentioned above.

At this point, at least the following issue naturally arises:

(i) How can weights and the balancing condition be incorporated into the framework of abstract tropical curves?

There are broadly two possible approaches: an intrinsic one and an extrinsic one.
The former modifies the very definition of abstract tropical curves in order to build a framework that incorporates weights and the balancing condition, while the latter keeps the existing framework intact and describes weights and the balancing condition by imposing additional external conditions.

Research close to the former approach includes the study of tropical ideals introduced in \cite{Maclagan=Rincon}.
In this theory, tropical ideals are defined by abstracting the properties possessed by tropicalizations of ideals; that is, by imposing conditions on ideals to restrict the class under consideration, \cite{Maclagan=Rincon2} shows that the varieties they define have structure of weighted polyhedral complexes satisfying the balancing condition, and then studies further properties.
In this paper, however, we pursue the latter approach.
Moreover, in view of the fact that tropical geometry is a form of algebraic geometry, this approach is realized within an algebraic framework:

(ii) How do we translate several geometric notions (containing weights and the balancing condition) for abstract tropical curves into algebraic ones?

To this end, we must first recall that for an abstract tropical curve $\Gamma$, its rational function semifield $\operatorname{Rat}(\Gamma)$ is characterized as follows (for notations, see Section~\ref{section2}):

\begin{thm}[{\cite[Corollary~3.19]{JuAe5}, cf.~\cite[Section~3]{JuAe=Nakajima}}]
	\label{thm1-1}
Let $S$ be a $\boldsymbol{T}$-algebra.
Then $S$ is isomorphic to $\operatorname{Rat}(\Gamma)$ as a $\boldsymbol{T}$-algebra with some abstract tropical curve $\Gamma$ if and only if some (and any) surjective $\boldsymbol{T}$-algebra homomorphism $\psi$ from a tropical rational function semifield to $S$ satisfies the following five conditions:

$(1)$ $\operatorname{Ker}(\psi)$ is finitely generated as a congruence,

$(2)$ $\operatorname{Ker}(\psi) = \boldsymbol{E}(\boldsymbol{V}(\operatorname{Ker}(\psi)))$ holds,

$(3)$ $\boldsymbol{V}(\operatorname{Ker}(\psi))$ is connected,

$(4)$ $\boldsymbol{V}(\operatorname{Ker}(\psi))$ is of dimension zero or one, and

$(5)$ if $L_1, L_2$ are rays of $\boldsymbol{V}(\operatorname{Ker}(\psi))$ with the same direction vector toward infinity, then $L_1 \subset L_2$ or $L_1 \supset L_2$ holds.
\end{thm}

By the classical fact that over an algebraically closed field $k$, the category of smooth projective curves with dominant morphisms is contravariantly equivalent to the category of one-dimensional function fields with $k$-homomorphisms, together with Theorem~\ref{thm1-1}, we expect a contravariant equivalence between the category of abstract tropical curves and the category of semifields over $\boldsymbol{T}$ in Theorem~\ref{thm1-1}.
If such an equivalence can be established, various notions associated with abstract tropical curves can be translated into algebraic notions of their rational function semifields, allowing abstract tropical curves to be studied within an algebro-geometric framework.

This categorical equivalence has already been partially resolved in \cite[Corollary~3.20]{JuAe5} and \cite[Section~3]{JuAe=Nakajima}, and in this paper we aim to provide a complete resolution in Corollary~\ref{cor3-1-1}.
There are three points that deserve special attention here.

The first concern is the definition of morphisms between abstract tropical curves.
This notion was introduced by Chan in \cite{Chan}, where only (finite) harmonic morphisms were considered, rather than general morphisms.
This is a reasonable assumption since finite morphisms between algebraic curves are sent to harmonic morphisms between abstract tropical curves under tropicalization as explained in \cite{Amini=Baker=Brugalle=Rabinoff}.
On the other hand, no justification had previously been given for why the definition of morphisms adopted in \cite{Chan} should be used; the justification provided in Corollary~\ref{cor3-1-1} is the first of its kind.
Namely, morphisms between abstract tropical curves are geometric realizations of $\boldsymbol{T}$-algebra homomorphisms between their rational function semifields.

The second point to note concerns condition $(5)$ of Theorem~\ref{thm1-1}.
From the proof of this theorem, the following statement can be derived:

\begin{prop}[{\cite[Proposition~3.18]{JuAe5}, \cite[Proposition~3.11]{JuAe4}}]
    \label{prop1-1}
Let $\Gamma$ be an abstract tropical curve.
For $f_1, \ldots, f_n \in \operatorname{Rat}(\Gamma) \setminus \{ -\infty \}$, let $\theta : \Gamma \setminus \Gamma_{\infty} \to \boldsymbol{R}^n; x \mapsto (f_1(x), \ldots, f_n(x))$.
Then the following are equivalent:

$(1)$ $\operatorname{Rat}(\Gamma)$ is isomorphic to $\overline{\boldsymbol{T}(\boldsymbol{X})}_n / \boldsymbol{E}(\operatorname{Im}(\theta))$ as a $\boldsymbol{T}$-algebra, and

$(2)$ $\theta$ is an injective local isometry and for any rays $L_1$ and $L_2$ of $\Gamma$, the images $\theta(L_1 \setminus \Gamma_{\infty})$ and $\theta(L_2 \setminus \Gamma_{\infty})$ do not contain rays with the same direction vector toward infinity.
\end{prop}

Considering that, under the above definition of morphisms, the notion of isomorphism between abstract tropical curves coincides with that of isometries, this statement is rather surprising.
Indeed, if an abstract tropical curve $\Gamma$ is ``realized" (other than points at infinity) as the image of a map $\theta$ consisting of a finite number of rational functions on $\Gamma$ other than $-\infty$ as in Proposition~\ref{prop1-1}, their ``rational function semifields" $\operatorname{Rat}(\Gamma)$ and $\overline{\boldsymbol{T}(\boldsymbol{X})}_n / \boldsymbol{E}(\operatorname{Im}(\theta))$ should coincide, but this is not the case.
This discrepancy arises from a deficiency in the definition of abstract tropical curves, which extracts only the information of topology and lattice length in the abstraction of classical tropical curves, and the definition of morphisms between them that respect only their topologies and metrics.

In fact, for example, tropical plane curves can clearly have parallel rays that do not contain each other, but information about such parallelism—ultimately originating from lattice data—cannot be captured in the above category.
Similar issues concerning parallel rays also arise in higher-dimensional cases; see \cite{Amini=Kawaguchi=Song} for details.

The third concern is finitely generated subsemifields over $\boldsymbol{T}$ of rational function semifields of abstract tropical curves.
This arises from condition $(5)$ of Theorem~\ref{thm1-1} again.
The author explained in \cite[Section~$1$]{JuAe4} that $\overline{\boldsymbol{T}(\boldsymbol{X})}_n / E$ plays the role of coordinate rings for any congruence $E$ on $\overline{\boldsymbol{T}(\boldsymbol{X})}_n$.
From such a perspective, we naturally expect that for an abstract tropical curve $\Gamma$, a finitely generated subsemifield over $\boldsymbol{T}$ of $\operatorname{Rat}(\Gamma)$ is again a rational function semifield of some abstract tropical curve.
However condition $(5)$ of Theorem~\ref{thm1-1} clearly denies this.

(iii) How do we solve these problems arising from condition $(5)$ of Theorem~\ref{thm1-1}?

\subsection*{Main results}

In this paper, to solve them, we introduce an extension of the notion of abstract tropical curves that allows the presence of parallel rays, namely, abstract tropical curves with parallel rays.
We then define the rational function semifields of such curves and provide a characterization in Theorem~\ref{thm4-1-1} analogous to Theorem~\ref{thm1-1}, thereby establishing the aforementioned categorical equivalence for this extended category as well (Corollary~\ref{cor4-2-1}).
Corollary~\ref{cor4-1-2} and Remark~\ref{rem4-1-2} show that a finitely generated subsemifield over $\boldsymbol{T}$ of the rational function semifield of an abstract tropical curve with parallel rays is again the rational function semifield of some abstract tropical curve with parallel rays.
We interpret weights as lattice dilations and define them via bijective morphisms between abstract tropical curves with parallel rays.
The balancing condition is understood through rational functions that are harmonic away from points at infinity.

Since harmonicity of rational functions at a point is a local property, we introduce localization at points in order to translate this notion into algebraic terms.

Taken together, these constructions yield a framework that treats lattice information in a more refined manner.
An advantage of this new framework is that, since everything is described algebraically, it provides a clear path toward higher-dimensional generalizations.
Although we do not pursue higher-dimensional cases in this paper, such generalizations appear feasible in light of the results of \cite{Amini=Kawaguchi=Song}.

\subsection*{Organization}

The remainder of this paper is organized as follows.
In Section~\ref{section2}, we briefly review the notions necessary for the discussion.
In Subsection~\ref{subsection2.8}, we further investigate the pseudodirect product of semifields over $\boldsymbol{T}$, introduced in \cite{JuAe=Nakajima} in tropical geometry, and introduce the corresponding geometric notion, i.e., disconnected abstract tropical curves.
In Sections~\ref{section3} and \ref{section4}, we frequently refer to results from \cite{JuAe3}, \cite{JuAe}, \cite{JuAe2}, \cite{JuAe5}, \cite{JuAe4}, and \cite{JuAe=Nakajima}, but due to their volume, they cannot be reproduced in Section~\ref{section2}; the reader is therefore encouraged to consult them as needed.

In Section~\ref{section3}, we first give the categorical equivalence mentioned above (Corollary~\ref{cor3-1-1}) in Subsection~\ref{subsection3.1}.
Next, we provide an algebraic characterization of subgraphs of abstract tropical curves in Subsection~\ref{subsection3.2}.
We then use this characterization to describe algebraically the operation of gluing abstract tropical curves along specified subgraphs in Subsection~\ref{subsection3.3}.

In Subsection~\ref{subsection4.1}, we introduce the notion of abstract tropical curves with parallel rays and define their rational function semifields.
After characterizing these semifields in Theorem~\ref{thm4-1-1}, we define morphisms between abstract tropical curves with parallel rays in Subsection~\ref{subsection4.2}, and Subsection~\ref{subsection4.3} is devoted to introduce disconnected abstract tropical curves with parallel rays.
We again provide an algebraic characterization of subgraphs in Subsection~\ref{subsection4.4}.
Subsection~\ref{subsection4.5} is devoted to an algebraic description of weights, and in Subsection~\ref{subsection4.6} localization is discussed using subgraphs.
In this process, we consider a semifield over $\boldsymbol{T}$ obtained by endowing the set of all tropical Laurent monomials together with $-\infty$ with natural operations; this structure is examined in detail in Appendix~\ref{appendixA}.
In Subsection~\ref{subsection4.7}, we define a notion of degree of finitely generated $\boldsymbol{T}$-modules in rational function semifields of abstract tropical curves with parallel rays.
In Subsection~\ref{subsection4.8}, by relating the balancing condition to the harmonicity of rational functions, we characterize weighted one-dimensional polyhedral complexes in $\boldsymbol{R}^n$ satisfying the balancing condition as images of maps given by collections of rational functions that are harmonic away from points at infinity on abstract tropical curves with parallel rays.
With this viewpoint, as explained in Subsection~\ref{subsection4.9}, the intersection number at an intersection point of two tropical plane curves can be understood purely in terms of rational functions on abstract tropical curves with parallel rays.

\section*{Acknowledgements}
The author thanks Yuki Tsutsui for his helpful comments, which improved a draft of this paper.
This work was supported by JSPS KAKENHI Grant Number 25K17230.

\section{Preliminaries}
	\label{section2}

In this section, we recall several definitions which we need later.
We refer to \cite{Golan} (resp. \cite{Maclagan=Sturmfels}) for an introduction to the theory of semirings (resp. tropical geometry) and employ definitions in \cite{Jun} (resp. \cite{JuAe3}) related to semirings (resp. abstract tropical curves).
The definition of morphisms between abstract tropical curves we employ in Subsection \ref{subsection2.7} is given in \cite{Chan}.
Today, it is usual for us to assume that a morphism between abstract tropical curves is (finite) harmonic (cf. \cite{Chan}, \cite{JuAe3}).
However, since with this definition, we obtain a contravariant categorical equivalence between the categories of abstract tropical curves and thier rational function semifields (see Corollary~\ref{cor3-1-1}), in our setting, it is natural to employ Chan's definition of morphisms between abstract tropical curves in \cite{Chan}.

\subsection{Semirings and algebras}
	\label{subsection2.1}

In this paper, a \textit{semiring} $S$ is a commutative semiring with the absorbing identity $0_S$ for addition $+$ and the identity $1_S$ for multiplication $\cdot$.
If every nonzero element of a semiring $S$ is multiplicatively invertible and $0_S \not= 1_S$, then $S$ is called a \textit{semifield}.

A map between semirings $\varphi : S_1 \to S_2$ is a \textit{semiring homomorphism} if for any $x, y \in S_1$,
\begin{align*}
\varphi(x + y) = \varphi(x) + \varphi(y), \	\varphi(x \cdot y) = \varphi(x) \cdot \varphi(y), \	\varphi(0_{S_1}) = 0_{S_2}, \   \text{and}\	\varphi(1_{S_1}) = 1_{S_2}.
\end{align*}

Given a semiring homomorphism $\varphi : S_1 \to S_2$, we call the pair $(S_2, \varphi)$ (for short, $S_2$) a \textit{$S_1$-algebra}.
For a semiring $S_1$, a map $\psi : (S_2, \varphi) \to (S_2^{\prime}, \varphi^{\prime})$ between $S_1$-algebras is a \textit{$S_1$-algebra homomorphism} if $\psi$ is a semiring homomorphism and $\varphi^{\prime} = \psi \circ \varphi$.
When there is no confusion, we write $\psi : S_2 \to S_2^{\prime}$ simply.

If a semiring homomorphism $\varphi : S_1 \to S_2$ is injective and both $S_1$ and $S_2$ are semifield, then $S_2$ is called a \textit{semifield over $S_1$}.
Then for $x_1, \ldots, x_n \in S_2 \setminus \{ 0_{S_2} \}$, we write the smallest subsemifield over $S_1$ of $S_2$ that contains $x_1, \ldots, x_n$ as $S_1(x_1, \ldots, x_n)$ and call it the semifield over $S_1$ \textit{generated by $x_1, \ldots, x_n$} (we do not use in this paper, but we can also define the same thing for any subset of $S_2 \setminus \{ 0_{S_2} \}$).
If there exist finitely many $x_1, \ldots, x_n \in S_2 \setminus \{ 0_{S_2} \}$ such that $S_1(x_1, \ldots, x_n) = S_2$, then $S_2$ is said to be \textit{finitely generated as a semifield over $S_1$}.

Let $\boldsymbol{R}$ be the set of real numbers.
The set $\boldsymbol{T} := \boldsymbol{R} \cup \{ -\infty \}$ with two tropical operations:
\begin{align*}
a \oplus b := \operatorname{max}\{ a, b \} \quad	\text{and} \quad a \odot b := a + b,
\end{align*}
where $a, b \in \boldsymbol{T}$, becomes a semifield.
Here, for any $a \in \boldsymbol{T}$, we handle $-\infty$ as follows:
\begin{align*}
a \oplus (-\infty) = (-\infty) \oplus a = a \quad \text{and} \quad a \odot (-\infty) = (-\infty) \odot a = -\infty.
\end{align*}
Then $\boldsymbol{T} = (\boldsymbol{T}, \oplus, \odot)$ is called the \textit{tropical semifield}.
In addition, $\boldsymbol{B} := (\{ 0, -\infty \}, \operatorname{max}, +)$ is a subsemifield of $\boldsymbol{T}$ called the \textit{Boolean semifield}.

The \textit{tropical polynomials} (resp.~\textit{tropical Laurent polynomials}) are defined as polynomials (resp.~Laurent polynomials) with respect to tropical operations $\oplus$ and $\odot$, where the coefficients are taken in $\boldsymbol{T}$, and the set of all tropical polynomials (resp.~tropical Laurent polynomials) in $n$-variables is denoted by $\boldsymbol{T}[X_1, \ldots, X_n]$ (resp.~$\boldsymbol{T}[X_1^{\pm}, \ldots, X_n^{\pm}]$).
It becomes a semiring with two tropical operations $\oplus$ and $\odot$ and is called the \textit{tropical polynomial semiring} (resp.~\textit{tropical Laurent polynomial semiring}).
We often write $\boldsymbol{T}[\boldsymbol{X}^{\pm}]_n$ instead of $\boldsymbol{T}[X_1^{\pm}, \ldots, X_n^{\pm}]$.
By the same way, we obtain the subsemiring $\boldsymbol{B}[\boldsymbol{X}^{\pm}]_n$ of $\boldsymbol{T}[\boldsymbol{X}^{\pm}]_n$ consisting of tropical Laurent polynomials whose coefficients are in $\boldsymbol{B}$ together with $-\infty$.
A \textit{tropical Laurent monomial} means a tropical Laurent polynomial of the form $a \odot \boldsymbol{X}^{\odot \boldsymbol{i}} := a \odot X_1^{\odot i_1} \odot \cdots \odot X_n^{\odot i_n}$ with $a \in \boldsymbol{R}$ and an integer vector $\boldsymbol{i} = (i_1, \ldots, i_n) \in \boldsymbol{Z}^n$.
Note that $-\infty$ is not a tropical Laurent monomial.
When $\boldsymbol{i} \in \boldsymbol{Z}_{\ge 0}^n$, we call $a \odot \boldsymbol{X}^{\odot \boldsymbol{i}}$ a \textit{tropical monomial}, and its \textit{degree} is the sum $i_1 + \cdots + i_n$. 
A tropical Laurent polynomial $F$ other than $-\infty$ is written by the sum of finitely many tropical Laurent monomials, each of which is called a \textit{term} of $F$.
When $F$ is a tropical polynomial, its \textit{degree} is defined by the maximum of the degrees of terms of $F$.
For $F = -\infty$, we define that it has degree $-\infty$.

\subsection{Congruences}
	\label{subsection2.2}

A \textit{congruence} $E$ on a semiring $S$ is a subset of $S^2 = S \times S$ satisfying

$(1)$ for any $x \in S$, $(x, x) \in E$,

$(2)$ if $(x, y) \in E$, then $(y, x) \in E$,

$(3)$ if $(x, y) \in E$ and $(y, z) \in E$, then $(x, z) \in E$,

$(4)$ if $(x, y) \in E$ and $(z, w) \in E$, then $(x + z, y + w) \in E$, and

$(5)$ if $(x, y) \in E$ and $(z, w) \in E$, then $(x \cdot z, y \cdot w) \in E$.

Applying the operations of $S$ coordinatewise, the set $S^2$ becomes a semiring and is a congruence on $S$.
This is called the \textit{improper} congruence on $S$.
Congruences other than the improper congruence are said to be \textit{proper}.
For a congruence $E$ on $S$, by $(1), (2)$ and $(3)$ above, we have the quotient set $S / E$ of $S$ by $E$.
This forms a semiring by $(4)$ and $(5)$ above with the following two well-defined operations: $[x] + [y] := [x + y]$ and $[x] \cdot [y] := [x \cdot y]$ for $x, y \in S$, where $[z]$ denotes the equivalence class of $z \in S$ under $E$.
Then the natural surjection $\pi_E : S \twoheadrightarrow S / E; x \mapsto [x]$ is a semiring homomorphism.

The intersection of (possibly infinitely many) congruences is again a congruence.
For a subset $T$ of $S^2$, let $\langle T \rangle$ be the smallest congruence on $S$ containing $T$, i.e., the intersection of all congruences on $S$ containing $T$.
A congruence $E$ on $S$ is \textit{finitely generated} if there exist a finite number of elements $(a_1, b_1), \ldots, (a_n, b_n) \in E$ such that $\langle \{ (a_1, b_1), \ldots, (a_n, b_n) \} \rangle = E$.

For a semiring homomorphism $\psi : S_1 \to S_2$, the \textit{kernel congruence} $\operatorname{Ker}(\psi)$ of $\psi$ is the congruence $\{ (x, y) \in S_1^2 \,|\, \psi(x) = \psi(y) \}$.
For semirings and congruences on them, the fundamental homomorphism theorem holds (\cite[Proposition 2.4.4]{Giansiracusa=Giansiracusa}).
Then, for the above $\pi_E$, we have $\operatorname{Ker}(\pi_E) = E$.

A congruence $P$ on a semiring $S$ is \textit{prime} if it is proper and $(x_1 \cdot y_1 + x_2 \cdot y_2, x_1 \cdot y_2 + x_2 \cdot y_1) \in P$ implies that $(x_1, y_1) \in P$ or $(x_2, y_2) \in P$ holds.
The \textit{Krull dimension} of $S$ is defined by the maximum length of strict inclusions of prime congruences on $S$.
If $S$ has no prime congruence, then we define that $S$ has Krull dimension $-\infty$.

Let $\overline{E} := \{ (f, g) \in \boldsymbol{T}[\boldsymbol{X}^{\pm}]_n^2 \,|\, \forall x \in \boldsymbol{R}^n, f(x) = g(x) \}$.
Then $E$ is a congruence on $\boldsymbol{T}[\boldsymbol{X}^{\pm}]_n$ and we write the quotient semiring $\boldsymbol{T}[\boldsymbol{X}^{\pm}]_n / \overline{E}$ as $\overline{\boldsymbol{T}[\boldsymbol{X}^{\pm}]}_n$.
It has the \textit{tropical rational function semifield} $\overline{\boldsymbol{T}(\boldsymbol{X})}_n$ in $n$-variables as its semifield of fractions.
An element of $\overline{\boldsymbol{T}(\boldsymbol{X})}_n \setminus \{ -\infty \}$ is a \textit{rational function} on $\boldsymbol{R}^n$, which is a real-valued function on $\boldsymbol{R}^n$ defined by the fraction of two tropical (Laurent) polynomials in $n$-variables.
By abuse of notation, we write the image of $X_i \in \boldsymbol{T}[\boldsymbol{X}^{\pm}]_n$ in $\overline{\boldsymbol{T}(\boldsymbol{X})}_n$ again $X_i$.

For a subset $V \subset \boldsymbol{R}^n$, let $\boldsymbol{E}(V)$ be the subset $\{ (f, g) \in \overline{\boldsymbol{T}(\boldsymbol{X})}_n^2 \,|\, \forall x \in V, f(x) = g(x) \}$, which is a congruence on $\overline{\boldsymbol{T}(\boldsymbol{X})}_n$.
Similarly, for a subset $T \subset \overline{\boldsymbol{T}(\boldsymbol{X})}_n^2$, let $\boldsymbol{V}(T) := \{ x \in \boldsymbol{R}^n \,|\, \forall (f, g) \in T, f(x) = g(x) \}$.
Since $\boldsymbol{V}(T) = \boldsymbol{V}(\langle T \rangle)$ holds by definition, we call $\boldsymbol{V}(T)$ the \textit{congruence variety associated with $T$}.

\subsection{Polyhedral sets}
	\label{subsection2.3}

A \textit{polyhedral set} in $\boldsymbol{R}^n$ is the solution set of a system of a finite number of linear inequalities (in the usual sense).
A \textit{polyhedral complex} $X$ is a complex consisting of a finite number of polyhedral sets.
Each polyhedral set in $X$ is a \textit{cell}.
Its \textit{support} $|X|$ is the union of its polyhedral sets.
When $|X|$ is connected as a subset of $\boldsymbol{R}^n$, we say that $X$ is \textit{connected}.
A finite union of polyhedral sets in $\boldsymbol{R}^n$ is the support of a polyhedral complex in $\boldsymbol{R}^n$, and vice versa (cf.~\cite[Proposition~4.1.1(a)]{Mikhalkin=Rau}).

For a nonempty polyhedral set $P$ in $\boldsymbol{R}^n$, its \textit{dimension} $\operatorname{dim}P$ is defined by the dimension of the smallest affine subspace of $\boldsymbol{R}^n$ containing $P$.
If $P$ is empty, we define that it has dimension $-1$.
The \textit{dimension} $\operatorname{dim}X$ of a polyhedral complex $X$ consisting of polyhedral sets $P_1, \ldots, P_m$ is defined as $\operatorname{max}\{ \operatorname{dim}P_i \,|\, i = 1, \ldots, m\}$.
We say that $|X|$ has dimension $\operatorname{dim}X$, i.e., $\operatorname{dim}(|X|) := \operatorname{dim}X$.
Since $X$ consists of only a finite number of polyhedral sets, $\operatorname{dim}(|X|)$ is independent of the choice of its \textit{polyhedral structures}, i.e., polyhedral complexes with $|X|$ as their supports.

A polyhedral set $P$ in $\boldsymbol{R}^n$ is \textit{$\boldsymbol{R}$-rational} if $P$ is of the form $\{ \boldsymbol{x} \in \boldsymbol{R}^n \,|\, A\boldsymbol{x} \ge \boldsymbol{b} \}$ with some $l \times n$-matrix $A$ with rational entries and some vector $\boldsymbol{b} \in \boldsymbol{R}^l$.
A polyhedral complex is \textit{$\boldsymbol{R}$-rational} if it consists only of $\boldsymbol{R}$-rational polyhedral sets.

For a congruence on $\overline{\boldsymbol{T}(\boldsymbol{X})}_n$, if it is finitely generated as a congruence, then the associated congruence variety is a finite union of $\boldsymbol{R}$-rational polyhedral sets in $\boldsymbol{R}^n$ by \cite[Corollary~3.5]{JuAe5}.
For a subset $V$ of $\boldsymbol{R}^n$, the congruence $\boldsymbol{E}(V)$ on $\overline{\boldsymbol{T}(\boldsymbol{X})}_n$ is finitely generated if and only if the closure of $V$ for the Euclidean topology of $\boldsymbol{R}^n$ is a finite union of $\boldsymbol{R}$-rational polyhedral sets by \cite[Theorem~1.1]{JuAe5}.
By \cite[Corollary~3.30]{JuAe=Nakajima}, for a nonempty finite union $V$ of $\boldsymbol{R}$-rational polyhedral sets in $\boldsymbol{R}^n$, the Krull dimension of $\overline{\boldsymbol{T}(\boldsymbol{X})}_n / \boldsymbol{E}(V)$ is equal to the dimension of $V$ (as the support of a polyhedral complex) plus one.

For a tropical Laurent polynomial $F \in \boldsymbol{T}[\boldsymbol{X}^{\pm}]_n \setminus \{ -\infty \}$, we define the \textit{tropical hypersurface} $\boldsymbol{V}(F)$ defined by $F$ as
\begin{align*}
\{ x \in \boldsymbol{R}^n \,|\, F \text{ has at least two terms that take the value } F(x) \}.
\end{align*}
For $F = -\infty \in \boldsymbol{T}[\boldsymbol{X}^{\pm}]_n$, we set $\boldsymbol{V}(-\infty) = \boldsymbol{R}^n$.
Note that if $F$ is a tropical Laurent monomial, i.e., it has just one term, then $\boldsymbol{V}(F) = \varnothing$, and vice versa.
It is well-known that $\boldsymbol{V}(F)$ has a connected weighted one-dimensional $\boldsymbol{R}$-rational polyhedral complex structure satisfying the balancing condition when $n = 2$ and $F$ is not a tropical Laurent monomial or $-\infty$ (see \cite[Proposition~3.1.6 and Theorem~3.3.5]{Maclagan=Sturmfels} for more details and more general cases).
Here \textit{weighted} means that its maximal cells (hence, in this case, one-dimensional cells) have positive integers called their \textit{weights}.
In particular, the weights are naturally defined by $F$.
Also satisfying the \textit{balancing condition} means that for any zero-dimensional cell $x$, if $P_1, \ldots, P_k$ are the one-dimensional cells incident to $x$, then the sum of primitive vectors from $x$ to $P_1, \ldots, P_k$ times their weights, respectively, is the zero vector.
A vector in $\boldsymbol{R}^n$ is \textit{primitive} if it is in $\boldsymbol{Z}^n$ and the greatest common divisor of the components is one.
For a primitive vector $\boldsymbol{v} = (v_1, \ldots, v_n) \in \boldsymbol{Z}^n$ and a nonnegative number $\lambda \ge 0$, the vector $\lambda \boldsymbol{v} = (\lambda v_1, \ldots, \lambda v_n)$ has $\lambda$ as its \textit{lattice length}.

\subsection{Tropical curves}
    \label{subsection2.4}

In this paper, a \textit{graph} is an unweighted, undirected, finite, connected nonempty multigraph that may have loops.
For a graph $G$, the set of vertices is denoted by $V(G)$ and the set of edges by $E(G)$.
A vertex $v$ of $G$ is a \textit{leaf end} if $v$ is incident to only one edge and this edge is not a loop.
A \textit{leaf edge} is an edge of $G$ incident to a leaf end.

An \textit{abstract tropical curve} is the underlying topological space of the pair $(G, l)$ of a graph $G$ and a function $l: E(G) \to {\boldsymbol{R}}_{>0} \cup \{\infty\}$, where $l$ can take the value $\infty$ only on leaf edges, together with an identification of each edge $e$ of $G$ with the closed interval $[0, l(e)]$.
The interval $[0, \infty]$ is the one-point compactification of the interval $[0, \infty)$.
We regard $[0, \infty]$ not just as a topological space but as an extended metric space.
The distance between $\infty$ and any other point is infinite.
When $l(e)=\infty$, the leaf end of $e$ must be identified with $\infty$.
If $E(G) = \{ e \}$ and $l(e)=\infty$, then we can identify either leaf ends of $e$ with $\infty$.
In what follows, we say abstract tropical curves just tropical curves.
When a tropical curve $\Gamma$ is obtained from $(G, l)$, the pair $(G, l)$ is called a \textit{model} for $\Gamma$.
There are many possible models for $\Gamma$.
An \textit{edge} of $\Gamma$ means an edge of $G$ for some model $(G, l)$ for $\Gamma$.
We frequently identify a vertex (resp. an edge) of $G$ with the corresponding point (resp. the corresponding closed subset) of $\Gamma$.
A model $(G, l)$ is \textit{loopless} if $G$ is loopless.
For a point $x$ of $\Gamma$, if $x$ is identified with $\infty$, then $x$ is called a \textit{point at infinity}, otherwise, $x$ is called a \textit{finite point}.
Let $\Gamma_{\infty}$ denote the set of all points at infinity of $\Gamma$.
If $x$ is a finite point, then the \textit{valence} $\operatorname{val}(x)$ is the number of connected components of $U \setminus \{ x \}$ with any sufficiently small connected neighborhood $U$ of $x$; if $x$ is a point at infinity, then $\operatorname{val}(x) := 1$.
We construct a model $(G_{\circ}, l_{\circ})$ called the {\it canonical model} for $\Gamma$ as follows.
Generally, we define $V(G_{\circ}) := \{ x \in \Gamma \,|\, \operatorname{val}(x) \not= 2 \}$ except for the following two cases.
When $\Gamma$ is homeomorphic to a circle $S^1$, we define $V(G_{\circ})$ as the set consisting of one arbitrary point of $\Gamma$.
When $\Gamma$ has the pair $(T, l)$ as its model, where $T$ is a tree consisting of three vertices and two edges and $l(E(T)) = \{ \infty \}$, we define $V(G_{\circ})$ as the set of two points at infinity and any finite point of $\Gamma$.
A \textit{ray} of $\Gamma$ is an edge of $G$ of length infinity for some model $(G, l)$ for $\Gamma$, so it contains a point at infinity.
A \textit{subgraph} of $\Gamma$ is a closed subset of $\Gamma$ with a finite number of connected components.

\subsection{Rational functions and chip firing moves}
    \label{subsection2.5}

Let $\Gamma$ be a tropical curve.
A continuous map $f : \Gamma \to \boldsymbol{R} \cup \{ \pm \infty \}$ is a \textit{rational function} on $\Gamma$ if $f$ is identically $-\infty$ or a piecewise affine function with integer slopes, with a finite number of pieces and that can take the values $\pm \infty$ at only points at infinity.
Let $\operatorname{Rat}(\Gamma)$ denote the set of all rational functions on $\Gamma$.
For rational functions $f, g \in \operatorname{Rat}(\Gamma)$ and a point $x \in \Gamma \setminus \Gamma_{\infty}$, we define
\begin{align*}
(f \oplus g) (x) := \operatorname{max}\{f(x), g(x)\} \quad \text{and} \quad (f \odot g) (x) := f(x) + g(x).
\end{align*}
We extend $f \oplus g$ and $f \odot g$ to points at infinity to be continuous on the whole of $\Gamma$.
Then both are rational functions on $\Gamma$.
Note that for any $f \in \operatorname{Rat}(\Gamma)$, we have
\begin{align*}
f \oplus (-\infty) = (-\infty) \oplus f = f \quad \text{and} \quad f \odot (-\infty) = (-\infty) \odot f = -\infty.
\end{align*}
Then $\operatorname{Rat}(\Gamma)$ becomes a semifield with these two operations.
Also, $\operatorname{Rat}(\Gamma)$ becomes a $\boldsymbol{T}$-algebra with the natural inclusion $\boldsymbol{T} \hookrightarrow \operatorname{Rat}(\Gamma)$.
Note that for $f, g \in \operatorname{Rat}(\Gamma)$, the equality $f = g$ means that $f(x) = g(x)$ for any $x \in \Gamma$.

Let $\Gamma^{\prime}$ be a subgraph of $\Gamma$ which has no connected components consisting of only a point at infinity and $l$ a positive number or infinity.
The \textit{chip firing move} by $\Gamma^{\prime}$ and $l$ is defined as the rational function $\operatorname{CF}(\Gamma^{\prime}, l)(x) := - \operatorname{min}\{ \operatorname{dist}(\Gamma^{\prime}, x), l \}$ with $x \in \Gamma$, where $\operatorname{dist}(\Gamma^{\prime}, x)$ denotes the distance between $\Gamma^{\prime}$ and $x$ in $\Gamma$.

\subsection{Divisors and \texorpdfstring{$\boldsymbol{T}$}{T}-modules}
    \label{subsection2.6}

Let $\Gamma$ be a tropical curve.
A \textit{divisor} on $\Gamma$ is an element of the free abelian group $\operatorname{Div}(\Gamma)$ generated by the points of $\Gamma$.
For a divisor $D$ on $\Gamma$ and a point $x$ of $\Gamma$, we write the coefficient of $D$ at $x$ as $D(x)$.
The sum of coefficients of $D$ at all points is denoted by $\operatorname{deg}(D)$, and called the \textit{degree} of $D$.
If all coefficients of $D$ are nonnegative, then $D$ is said to be \textit{effective} and written by $D \ge 0$.

For a finite point $x$ of $\Gamma$, a rational function $f \in \operatorname{Rat}(\Gamma) \setminus \{ -\infty \}$ and one outgoing direction at $x$, the \textit{outgoing slope} of $f$ at $x$ is the ratio $\frac{f(y) - f(x)}{\operatorname{dist}(x,y)}$ (in the usual sense) with a finite point $y$ in a sufficiently small neighborhood of $x$ in the direction.
If $x$ is a point at infinity, then we regard the \textit{outgoing slope} of $f$ at $x$ as the slope of $f$ from $y$ to $x$ times minus one, where $y$ is a finite point on the leaf edge incident to $x$ such that $f$ has a constant slope on the interval $(y, x)$.
In both cases, the definition of outgoing slope of $f$ at $x$ in the direction is independent of the choice of $y$.
The point $x$ is a \textit{zero} (resp.~\textit{pole}) of $f$ if the sign of the sum of outgoing slopes of $f$ at $x$ is positive (resp.~negative).
The absolute value of the sum is its \textit{degree}.
By the definition of rational functions on tropical curves, $f$ has at most finitely many zeros and poles.
For the zeros $x_1, \ldots, x_n \in \Gamma$ with degrees $d_1, \ldots, d_n$, respectively, and the poles $y_1, \ldots, y_m \in \Gamma$ with degrees $s_1, \ldots, s_m$, respectively, of $f$, let $\operatorname{div}(f)$ be the divisor $\sum_{i = 1}^n d_i x_i - \sum_{j = 1}^m s_j y_j \in \operatorname{Div}(\Gamma)$, which is called the \textit{principal divisor} defined by $f$.
When $f \in \boldsymbol{R}$, we consider $\operatorname{div}(f) = 0 \in \operatorname{Div}(\Gamma)$, and for $f = -\infty$, we do not define its principal divisor.
For a divisor $D \in \operatorname{Div}(\Gamma)$, let $R(D) := \{ f \in \operatorname{Rat}(\Gamma) \setminus \{ -\infty \} \,|\, D + \operatorname{div}(f) \ge 0 \} \cup \{ -\infty \}$.
It is known that $R(D)$ is a finitely generated $\boldsymbol{T}$-module by \cite[Theorem~6]{Haase=Musiker=Yu} (in the case of excluding points at infinity, \cite[Theorem~3.15]{JuAe6} in the case of including such points).
Here a subset $R$ of $\operatorname{Rat}(\Gamma)$ is a \textit{$\boldsymbol{T}$-module} (in $\operatorname{Rat}(\Gamma)$) means that for any $f, g \in R$ and $t \in \boldsymbol{T}$, the sum $f \oplus g$ and the tropical scalar multiplication $t \odot f$ are in $R$.
Also a subset $R$ of $\operatorname{Rat}(\Gamma)$ is \textit{finitely generated as a $\boldsymbol{T}$-module} means that there exist finitely many $f_1, \ldots, f_n$ in $R$ such that for any $f \in R$, there exist $t_1, \ldots, t_n \in \boldsymbol{T}$ satisfying $f = t_1 \odot f_1 \oplus \cdots \oplus t_n \odot f_n$.
Let $R$ be a $\boldsymbol{T}$-module in $\operatorname{Rat}(\Gamma)$.
An element $f \in R \setminus \{ -\infty \}$ is an \textit{extremal} if for any $g, h \in R$, the equality $f = g \oplus h$ implies $f = g$ or $f = h$.
By \cite[Proposition~8]{Haase=Musiker=Yu}, any finitely generated $\boldsymbol{T}$-module $R \not= \{ -\infty \}$ in $\operatorname{Rat}(\Gamma)$ is generated by extremals in $R$ and this generating set is minimal and unique up to tropical scalar multiplication other than $-\infty$.

One might expect that for a tropical curve $\Gamma$, any finitely generated $\boldsymbol{T}$-module in $\operatorname{Rat}(\Gamma)$ is of the form $R(D)$ for some divisor $D$.
However, this is false:

\begin{ex}
    \label{ex2-6-1}
\upshape{
Let $\Gamma := [-\infty, \infty]$, which is regarded as a tropical curve.

For the rational function $f(x) := 2x$ for $x \in \Gamma$, the $\boldsymbol{T}$-module $R$ generated by $f$ is not equal to $R(D)$ for any $D \in \operatorname{Div}(\Gamma)$.
In fact, the map $\Gamma \setminus \Gamma_{\infty} = (-\infty, \infty) \to \boldsymbol{R}; x \mapsto f(x)$ is injective, and if $R = R(D)$ holds for some divisor $D$, then $f$ generates $\operatorname{Rat}(\Gamma)$ as a semifield over $\boldsymbol{T}$ by \cite[Corollary~3.21]{JuAe5}.
But this is a contradiction, clearly $f$ does not generate $\operatorname{Rat}(\Gamma)$.

By the same reason, for $g(x) := -x$ with $x \le 0$ and $g(x) := x$ with $x > 0$ and $h(x) := 0$ with $x \le 0$ and $h(x) := x$ with $0 < x \le 1$ and $h(x) := 1$ with $x > 1$, the $\boldsymbol{T}$-module generated by $g$ and $h$ is never equal to $R(D)$ for any divisor $D$ on $\Gamma$.
}
\end{ex}

\subsection{Morphisms between tropical curves}
    \label{subsection2.7}

Let $\varphi : \Gamma \to \Gamma^{\prime}$ be a continuous map between tropical curves.
This $\varphi$ is a \textit{morphism} if there exist loopless models $(G, l)$ and $(G^{\prime}, l^{\prime})$ for $\Gamma$ and $\Gamma^{\prime}$, respectively, such that $\varphi$ can be regarded as a map $V(G) \cup E(G) \to V(G^{\prime}) \cup E(G^{\prime})$ satisfying $\varphi(V(G)) \subset V(G^{\prime})$ and for $e \in \varphi(E(G))$, there exists a nonnegative integer $\operatorname{deg}_e(\varphi)$ such that for any finite points $x$ and $y$ of $e$, the equality $\operatorname{dist}_{\varphi(e)}(\varphi (x), \varphi (y)) = \operatorname{deg}_e(\varphi) \cdot \operatorname{dist}_e(x, y)$ holds, where $\operatorname{dist}_{\varphi(e)}(\varphi(x), \varphi(y))$ (resp.~$\operatorname{dist}_{e}(x, y)$) denotes the distance between $\varphi(x)$ and $\varphi(y)$ in $\varphi(e)$ (resp.~$x$ and $y$ in $e$).
This integer $\operatorname{deg}_e(\varphi)$ is called the \textit{degree} of $\varphi$ on $e$.

For a morphism $\varphi : \Gamma \to \Gamma^{\prime}$ and $f^{\prime} \in \operatorname{Rat}(\Gamma^{\prime})$, it is easy to check that the composition $f^{\prime} \circ \varphi$ is a rational function on $\Gamma$.
The induced map $\varphi^{\ast} : \operatorname{Rat}(\Gamma^{\prime}) \to \operatorname{Rat}(\Gamma); f^{\prime} \mapsto f^{\prime} \circ \varphi$ is called the \textit{pull-back map} or the \textit{pull-back $\boldsymbol{T}$-algebra homomorphism} of $\varphi$, which is a $\boldsymbol{T}$-algebra homomorphism.

\subsection{Pseudodirect products and disconnected tropical curves}
    \label{subsection2.8}

Let $S_i$ be a semifield over $\boldsymbol{T}$.
The \textit{pseudodirect product} $S_1 \bowtie S_2 := (S_1 \bowtie S_2, +, \cdot)$ of $S_1$ and $S_2$ is defined by
\begin{align*}
S_1 \bowtie S_2 &:= \{ (s_1, s_2) \in (S_1 \setminus \{ 0_{S_1} \}) \times (S_2 \setminus \{ 0_{S_2} \}) \} \cup \{ (0_{S_1}, 0_{S_2}) \},\\
(s_1, s_2) + (t_1, t_2) &:= (s_1 + t_1, s_2 + t_2), \text{ and}\\
(s_1, s_2) \cdot (t_1, t_2) &:= (s_1 \cdot t_1, s_2 \cdot t_2).
\end{align*}
This is a semifield over $\boldsymbol{T}$ with the diagonal semiring homomorphism $\boldsymbol{T} \hookrightarrow S_1 \bowtie S_2; t \mapsto (t, t)$ by \cite[Lemma~3.32]{JuAe=Nakajima} and has the universal mapping property below.
This $S_1 \bowtie S_2$ is a subsemifield over $\boldsymbol{T}$ of the $\boldsymbol{T}$-algebra $S_1 \times S_2$ equipped with the operations of $S_1$ and $S_2$ coordinatewise, respectively.
Note that $S_1 \bowtie S_2$ has the natural surjective $\boldsymbol{T}$-algebra homomorphism $\pi_i : S_1 \bowtie S_2 \twoheadrightarrow S_i; (s_1, s_2) \not= (0_{S_1}, 0_{S_2}) \mapsto s_i, (0_{S_1}, 0_{S_2}) \mapsto 0_{S_i}$ by \cite[Lemma~3.33]{JuAe=Nakajima}.
We also write $S_1 \bowtie S_2$ as $\bowtie_{i = 1}^2 S_i$.

The pseudodirect product was originally defined for more general semirings, see \cite[Chapter~2]{Golan}.

\begin{prop}[Universal mapping property]
    \label{prop2-8-1}
In the above setting, for a semifield $T$ over $\boldsymbol{T}$ and a $\boldsymbol{T}$-algebra homomorphism $\psi_i : T \to S_i$ with $i = 1, 2$, there exists a unique $\boldsymbol{T}$-algebra homomorphism $\psi : T \to S_1 \bowtie S_2$ satisfying $\psi_i = \pi_i \circ \psi$ with $i = 1, 2$.
\end{prop}

\begin{proof}
For $a \in T$, if $\psi_1(a) = 0_{S_1}$ or $\psi_2(a) = 0_{S_2}$, then $a$ must be $0_T$.
In fact, since both $T$ and $S_i$ are semifields over $\boldsymbol{T}$ and $\psi_i$ is a $\boldsymbol{T}$-algebra homomorphism, the image $\psi_i(T)$ is a subsemifield over $\boldsymbol{T}$ of $S_i$, and hence $a$ must be $0_T$ when $\psi_i(a) = 0_{S_i}$ by \cite[Lemma~2.2]{JuAe4}.
Thus for any $a \in T \setminus \{ 0_T \}$, the pair $(\psi_1(a), \psi_2(a))$ defines an element of $S_1 \bowtie S_2$. 
Defining $\psi$ by the correspondence $a \not= 0_T \mapsto (\psi_1(a), \psi_2(a)), 0_T \mapsto (0_{S_1}, 0_{S_2})$, it becomes a $\boldsymbol{T}$-algebra homomorphism $T \to S_1 \bowtie S_2$ as $t \in \boldsymbol{T} \subset T$ is mapped to $(t, t) \in \boldsymbol{T} \subset S_1 \bowtie S_2$ by $\psi$.
This $\psi$ clearly satisfies $\psi_i = \pi_i \circ \psi$ by definition.
Let $\psi^{\prime} : T \to S_1 \bowtie S_2$ be a $\boldsymbol{T}$-algebra homomorpshim satisfying $\psi_i = \pi_i \circ \psi^{\prime}$.
Then $\psi = \psi^{\prime}$ holds since for any $a \in T \setminus \{ 0_T \}$, when $\psi^{\prime}(a) = (b_1, b_2)$, each $b_i$ coincides with $\psi_i(a)$ by $\psi_i = \pi_i \circ \psi^{\prime}$ and so $\psi^{\prime}(a) = (b_1, b_2) = (\psi_1(a), \psi_2(a)) = \psi(a)$ and $\psi^{\prime}(0_T) = (0_{S_1}, 0_{S_2}) = \psi(0_T)$.
\end{proof}

The following is clear:

\begin{prop}
    \label{prop2-8-2}
Let $S_i$ be a semifield over $\boldsymbol{T}$.
Then the map $(S_1 \bowtie S_2) \bowtie S_3 \to S_1 \bowtie (S_2 \bowtie S_3); ((s_1, s_2), s_3) \not= ((0_{S_1}, 0_{S_2}), 0_{S_3}) \mapsto (s_1, (s_2, s_3)), ((0_{S_1}, 0_{S_2}), 0_{S_3}) \mapsto (0_{S_1}, (0_{S_2}, 0_{S_3}))$ is a $\boldsymbol{T}$-algebra isomorpshim.
\end{prop}

By Proposition~\ref{prop2-8-2}, we can write $(S_1 \bowtie S_2) \bowtie S_3$ (or $S_1 \bowtie (S_2 \bowtie S_3)$) just $S_1 \bowtie S_2 \bowtie S_3$.

For two tropical curves $\Gamma_1$ and $\Gamma_2$, by \cite[Proposition~3.11]{JuAe4} and \cite[Lemma~3.41]{JuAe=Nakajima}, we can regard $\operatorname{Rat}(\Gamma_1) \bowtie \operatorname{Rat}(\Gamma_2)$ as the rational function semifield of the direct sum (as a topological space) $\Gamma_1 \sqcup \Gamma_2$ of $\Gamma_1$ and $\Gamma_2$, i.e.,
\begin{align*}
\operatorname{Rat}(\Gamma_1 \sqcup \Gamma_2) := \operatorname{Rat}(\Gamma_1) \bowtie \operatorname{Rat}(\Gamma_2).
\end{align*}
Here $\Gamma_1 \sqcup \Gamma_2$ has the same metric as $\Gamma_i$ as in the image of the natural inclusion $\Gamma_i \hookrightarrow \Gamma_1 \sqcup \Gamma_2$ for each $i = 1, 2$.
We do not define the distance between the points of the image of $\Gamma_1$ and those of the image of $\Gamma_2$.
Thus, for a subgraph $\Gamma^{\prime}_i$ of $\Gamma_i$ that has no connected components consisting of a point at infinity and a positive number or infinity $l$, the chip firing move $\operatorname{CF}(\Gamma^{\prime}_1, l) \in \operatorname{Rat}(\Gamma_1 \sqcup \Gamma_2)$ (resp.~$\operatorname{CF}(\Gamma^{\prime}_1 \cup \Gamma^{\prime}_2, l) \in \operatorname{Rat}(\Gamma_1 \sqcup \Gamma_2)$) is the rational function $(\operatorname{CF}(\Gamma^{\prime}_1, l), 0) \in \operatorname{Rat}(\Gamma_1 \sqcup \Gamma_2)$ (resp.~$(\operatorname{CF}(\Gamma^{\prime}_1, l), \operatorname{CF}(\Gamma^{\prime}_2, l)) \in \operatorname{Rat}(\Gamma_1 \sqcup \Gamma_2)$).
We write the addition and multiplication in $\operatorname{Rat}(\Gamma_1 \sqcup \Gamma_2)$ as $\oplus$ and $\odot$, respectively.
In $\operatorname{Rat}(\Gamma_1 \sqcup \Gamma_2)$, the identity $0_{\operatorname{Rat}(\Gamma_1 \sqcup \Gamma_2)} = (-\infty, -\infty)$ with respect to $\oplus$ is written by $-\infty$.

Note that $\operatorname{Rat}(\Gamma_1 \sqcup \Gamma_2)$ is finitely generated as a semifield over $\boldsymbol{T}$ by \cite[Theorem~1.1]{JuAe3} and \cite[Lemma~3.34]{JuAe=Nakajima}.
For a surjective $\boldsymbol{T}$-algebra homomorphism $\psi : \overline{\boldsymbol{T}(\boldsymbol{X})}_n \twoheadrightarrow \operatorname{Rat}(\Gamma_1 \sqcup \Gamma_2)$, the congruence variety $\boldsymbol{V}(\operatorname{Ker}(\psi))$ is a finite union of $\boldsymbol{R}$-rational polyhedral sets in $\boldsymbol{R}^n$ with two connected components each of which is isometric to $\Gamma_1 \setminus \Gamma_{1, \infty}$ and $\Gamma_2 \setminus \Gamma_{2, \infty}$ with the lattice length, respectively, and $\boldsymbol{V}(\operatorname{Ker}(\psi))$ has nontrivial rays that have the same direction vector toward infinity by \cite[Corollary~3.19]{JuAe5} and \cite[Corollary~3.43]{JuAe=Nakajima}.

\begin{rem}
    \label{rem2-8-1}
\upshape{
In the above setting, we call the direct sum $\Gamma_1 \sqcup \Gamma_2$ a \textit{disconnected tropical curve} (with two connected components).
We can also define disconnected tropical curves with a finite number of connected components in the same way.
Then we can characterize the rational function semifields of such disconnected tropical curves as in \cite[Corollary~3.19]{JuAe5} or Theorem~\ref{thm1-1} without condition $(3)$ by the same proofs in \cite[Section~3]{JuAe5}.
But since we do not use this fact in this paper, to avoid redundancy, we do not more discuss about these ones.

Note here that the definition of rational function semifields of disconnected tropical curves is provisional.
As an analogue of the classical algebraic geometry, it might be more natural employing the product $\boldsymbol{T}$-algebra $\operatorname{Rat}(\Gamma_1) \times \operatorname{Rat}(\Gamma_2)$ as its definition.
However, in this paper, we use the pseudodirect product since with this notion, our discussion will be completed solely within the context of finitely generated semifields over $\boldsymbol{T}$, and we have many benefits explained in Subsections~\ref{subsection3.2}, \ref{subsection3.3}, \ref{subsection4.3} and \ref{subsection4.4}.
}
\end{rem}

For a disconnected tropical curve $\Gamma$, a \textit{ray} of $\Gamma$ is a ray of a connected component of $\Gamma$.

\section{Tropical curves}
	\label{section3}

In this section, we study (abstract) tropical curves and their rational function semifields beyond the results we gave in Section~\ref{section2} and \cite{JuAe3}, \cite{JuAe}, \cite{JuAe2}, \cite{JuAe5}, \cite{JuAe4}, \cite{JuAe=Nakajima}.
This section corresponds to question~(ii) in Section~\ref{section1}.

\subsection{Categorical equivalence}
    \label{subsection3.1}

It has not been pointed out until now, but \cite[the proof of Corollary~A1]{JuAe4} provides the following theorem:

\begin{thm}
    \label{thm3-1-1}
Let $\Gamma$ and $\Gamma^{\prime}$ be tropical curves, respectively.
For a $\boldsymbol{T}$-algebra homomorphism $\psi : \operatorname{Rat}(\Gamma^{\prime}) \to \operatorname{Rat}(\Gamma)$, there exists a unique morphism $\varphi : \Gamma \to \Gamma^{\prime}$ such that $\psi = \varphi^{\ast}$ holds.
\end{thm}

\begin{proof}
Replacing $\Gamma_2$ and $\Gamma_1$ in \cite[Corollary~A1 and its proof]{JuAe4} with $\Gamma$ and $\Gamma^{\prime}$, respectively, the same proof works except for the parts corresponding to the surjectivity of $\psi$ and the injectivity of $\varphi$ by \cite[Theorem~3.14]{JuAe4}.
\end{proof}

So far, $\psi$ in Theorem~\ref{thm3-1-1} is always assumed to be injective (and then $\varphi$ is surjective), but in fact this assumption is not necessary.

Theorem~\ref{thm3-1-1} clearly gives the following:

\begin{cor}[cf.~{\cite[Corollary~3.20]{JuAe5}}]
    \label{cor3-1-1}
The following categories $\mathscr{C}$ and $\mathscr{D}$ are equivalent via the following (contravariant) functors $F$ and $G$.

$(1)$ The class $\operatorname{Ob}(\mathscr{C})$ of objects of $\mathscr{C}$ is the tropical curves.

For $\Gamma_1, \Gamma_2 \in \operatorname{Ob}(\mathscr{C})$, the set $\operatorname{Hom}_{\mathscr{C}}(\Gamma_1, \Gamma_2)$ of morphisms from $\Gamma_1$ to $\Gamma_2$ consists of the morphisms $\Gamma_1 \to \Gamma_2$ (in the sense of Subsection~\ref{subsection2.7}).

$(2)$ The class $\operatorname{Ob}(\mathscr{D})$ of objects of $\mathscr{D}$ is the finitely generated semifields over $\boldsymbol{T}$ satisfying conditions $(1), \ldots, (5)$ in Theorem~\ref{thm1-1}.

For $S_1, S_2 \in \operatorname{Ob}(\mathscr{D})$, the set $\operatorname{Hom}_{\mathscr{D}}(S_1, S_2)$ of morphisms from $S_1$ to $S_2$ consists of the $\boldsymbol{T}$-algebra homomorphisms $S_1 \to S_2$.

The functor $F \colon \mathscr{C} \to \mathscr{D}$ maps $\Gamma \in \operatorname{Ob}(\mathscr{C})$ to $\operatorname{Rat}(\Gamma) \in \operatorname{Ob}(\mathscr{D})$ and for $\Gamma_1, \Gamma_2 \in \operatorname{Ob}(\mathscr{C})$, maps $\varphi \in \operatorname{Hom}_{\mathscr{C}}(\Gamma_1, \Gamma_2)$ to $\varphi^{\ast} \in \operatorname{Hom}_{\mathscr{D}}(\operatorname{Rat}(\Gamma_2), \operatorname{Rat}(\Gamma_1))$.

The functor $G \colon \mathscr{D} \to \mathscr{C}$ maps $S \in \operatorname{Ob}(\mathscr{D})$ to $\overline{\boldsymbol{V}(\operatorname{Ker}(\psi_S))} \in \operatorname{Ob}(\mathscr{C})$ and for $S_1, S_2 \in \operatorname{Ob}(\mathscr{D})$, maps $\widetilde{\psi} \in \operatorname{Hom}_{\mathscr{D}}(S_1, S_2)$ to $\varphi_{\widetilde{\psi}} \in \operatorname{Hom}_{\mathscr{C}}(\overline{\boldsymbol{V}(\operatorname{Ker}(\psi_{S_2}))}, \overline{\boldsymbol{V}(\operatorname{Ker}(\psi_{S_1}))})$, where $\psi_S$ is a fixed surjective $\boldsymbol{T}$-algebra homomorphism from a tropical rational function semifield to $S$ and $\overline{\boldsymbol{V}(\operatorname{Ker}(\psi_S))}$ is the natural compactification of $\boldsymbol{V}(\operatorname{Ker}(\psi_S))$ as a tropical curve for any $S \in \operatorname{Ob}(\mathscr{D})$ and $\varphi_{\widetilde{\psi}}$ is the unique morphism $\overline{\boldsymbol{V}(\operatorname{Ker}(\psi_{S_2}))} \to \overline{\boldsymbol{V}(\operatorname{Ker}(\psi_{S_1}))}$ such that $\widetilde{\psi} = \left(\varphi_{\widetilde{\psi}}\right)^{\ast}$ given in Theorem~\ref{thm3-1-1}.
\end{cor}

In Corollary~\ref{cor3-1-1}, the \textit{natural compactification} of $\boldsymbol{V}(\operatorname{Ker}(\psi_S))$ as a tropical curve means doing one-point compactifications on all points at infinity and inheriting topology and metric from those of $\boldsymbol{V}(\operatorname{Ker}(\psi_S))$ when we regard $\boldsymbol{V}(\operatorname{Ker}(\psi_S))$ as a metric space by the lattice length.

As explained in Sections~\ref{section1} and~\ref{section2}, the definition of morphisms between tropical curves was given by Chan in \cite{Chan}, but she used only (finite) harmonic morphisms and did not expalin why she called it a morphism.
Corollary~\ref{cor3-1-1} justifies this definition by $\boldsymbol{T}$-algebra homomorphisms, which are an algebraically definite notion.

\subsection{Subgraphs}
    \label{subsection3.2}

In this subsection, we investigate the relation between subgraphs (other than points at infinity) of tropical curves and surjective $\boldsymbol{T}$-algebra homomorphisms between their rational function semifields.

\begin{prop}
    \label{prop3-2-1}
Let $\Gamma$ be a tropical curve and $\Gamma^{\prime}$ a connected subgraph of $\Gamma$ that is not contained in $\Gamma_{\infty}$.
Then, regarding $\Gamma^{\prime}$ itself as a tropical curve, the restriction $f|_{\Gamma^{\prime}}$ on $\Gamma^{\prime}$ for any $f \in \operatorname{Rat}(\Gamma)$ is in $\operatorname{Rat}(\Gamma^{\prime})$ and the restriction map $\psi : \operatorname{Rat}(\Gamma) \to \operatorname{Rat}(\Gamma^{\prime}) : f \mapsto f|_{\Gamma^{\prime}}$ is a surjective $\boldsymbol{T}$-algebra homomorphism.
\end{prop}

\begin{proof}
Since $\Gamma^{\prime} \not\subset \Gamma_{\infty}$, the restriction $f|_{\Gamma^{\prime}}$ for any $f \in \operatorname{Rat}(\Gamma)$ is in $\operatorname{Rat}(\Gamma^{\prime})$, that is, $f|_{\Gamma^{\prime}}$ is never a constant function of $\infty$.

Clearly $\psi$ is a $\boldsymbol{T}$-algebra homomorphism and if $\Gamma^{\prime} = \Gamma$, then $\psi$ is surjective.

Assume that $\Gamma^{\prime} \subset \Gamma \setminus \Gamma_{\infty}$.
Let $f^{\prime} \in \operatorname{Rat}(\Gamma^{\prime}) \setminus \{ -\infty \}$.
Since $\Gamma^{\prime}$ is a subgraph of $\Gamma$ contained in $\Gamma \setminus \Gamma_{\infty}$, the rational function $f^{\prime}$ takes some real minimum value $a \in \boldsymbol{R}$ on $\Gamma^{\prime}$.
If there exists $f \in \operatorname{Rat}(\Gamma)$ such that $\psi(f) = a^{\odot (-1)} \odot f^{\prime}$, then $\psi(a \odot f) = f^{\prime}$ holds as $\psi$ is a $\boldsymbol{T}$-algebra homomorphism.
Hence it is enough to show the case $a = 0$, and assume this.
For a negative integer $s$ whose absolute value is sufficiently large, let $g$ be the function of $\Gamma$ such that $(1)$ $g$ coincides with $f^{\prime}$ on $\Gamma^{\prime}$ and $(2)$ $g$ has $s$ as its slope in the outgoing direction at each boundary point $x$ of $\Gamma^{\prime}$ in $\Gamma$ until it takes the value zero and $(3)$ $g$ is constant zero the outside of the union of $\Gamma^{\prime}$ and all such segments incident to each $x$. 
Then $g \in \operatorname{Rat}(\Gamma)$ and $\psi(g) = f^{\prime}$ hold by definition, and thus $\psi$ is surjective in this case.

Assume that $\Gamma^{\prime} \not= \Gamma$ and $\Gamma^{\prime} \cap \Gamma_{\infty} \not= \varnothing$.
Let $L^{\prime} \subset \Gamma^{\prime}$ be a ray of $\Gamma$.
Then $\psi(\operatorname{CF}(\overline{\Gamma \setminus L^{\prime}}, \infty)) \in \operatorname{Rat}(\Gamma^{\prime})$ holds, where $\overline{\Gamma \setminus L^{\prime}}$ denotes the closure of $\Gamma \setminus L^{\prime}$ in $\Gamma$.
By the previous case and \cite[the proof of Lemma~1.4]{JuAe3}, the $\boldsymbol{T}$-algebra homomorphism $\psi$ is surjective.
\end{proof}

The converse of Proposition~\ref{prop3-2-1} holds as follows.

\begin{prop}
    \label{prop3-2-2}
Let $\Gamma$ and $\Gamma^{\prime}$ be tropical curves and $\psi : \operatorname{Rat}(\Gamma) \twoheadrightarrow \operatorname{Rat}(\Gamma^{\prime})$ a surjective $\boldsymbol{T}$-algebra homomorphism.
Then the corresponding injective morphism $\iota : \Gamma^{\prime} \hookrightarrow \Gamma$ such that $\psi = \iota^{\ast}$ is a local isometry and its image $\operatorname{Im}(\iota)$ is not contained in $\Gamma_{\infty}$.
\end{prop}

\begin{proof}
The existence of the injective morphism $\iota$ that satisfies $\psi(f) = f \circ \iota$ for any $f \in \operatorname{Rat}(\Gamma)$ and $\operatorname{Im}(\iota) \not\subset \Gamma_{\infty}$ is guaranteed by \cite[Proposition~3.11 and Theorem~3.14]{JuAe4}.
For any $x^{\prime} \in \Gamma^{\prime} \setminus \Gamma^{\prime}_{\infty}$, since $\psi$ is surjective, there exists $f \in \operatorname{Rat}(\Gamma)$ such that $\operatorname{CF}(\{ x^{\prime} \}, \infty) = \psi(f) = f \circ \iota$.
For an edge $e^{\prime}$ of $\Gamma^{\prime}$ containing $x^{\prime}$ and any point $y^{\prime} \in e^{\prime}$ close enough to $x^{\prime}$, since $0 = \operatorname{CF}(\{ x^{\prime} \}, \infty)(x^{\prime}) = (f \circ \iota)(x^{\prime}) = f(\iota(x^{\prime}))$ and $- \operatorname{dist}(x^{\prime}, y^{\prime}) = \operatorname{CF}(\{ x^{\prime} \}, \infty)(y^{\prime}) = (f \circ \iota)(y^{\prime}) = f(\iota(y^{\prime}))$, the slope of $f$ from $\iota(x^{\prime})$ to $\iota(y^{\prime})$ is $- \frac{\operatorname{dist}(x^{\prime}, y^{\prime})}{\operatorname{dist}(\iota(x^{\prime}), \iota(y^{\prime}))} \in \boldsymbol{Z}$.
As $\varphi$ is an injective morphism, the inequality $\operatorname{dist}(\iota(x^{\prime}), \iota(y^{\prime})) \ge \operatorname{dist}(x^{\prime}, y^{\prime})$ holds, and so $\operatorname{dist}(\iota(x^{\prime}), \iota(y^{\prime}))$ must be $\operatorname{dist}(x^{\prime}, y^{\prime})$, that is, $\iota$ is a local isometry.
\end{proof}

By these two propositions, we can identify a connected subgraph $\Gamma^{\prime}$ of a tropical curve $\Gamma$ not contained in $\Gamma_{\infty}$ with a surjective $\boldsymbol{T}$-algebra homomorphism $\operatorname{Rat}(\Gamma) \twoheadrightarrow \operatorname{Rat}(\Gamma^{\prime})$.
More generally, we can identify a (possibly disconnected) subgraph $\Gamma^{\prime}$ of $\Gamma$ whose every connected component $\Gamma_i^{\prime}$ is not contained in $\Gamma_{\infty}$ with a surjective $\boldsymbol{T}$-algebra homomorphism $\psi : \operatorname{Rat}(\Gamma) \twoheadrightarrow \bowtie_{i = 1}^l \operatorname{Rat}(\Gamma_i^{\prime})$ as follows.

\begin{prop}
    \label{prop3-2-3}
Let $\Gamma$ be a tropical curve and $\Gamma^{\prime}$ a subgraph of $\Gamma$ with connected components $\Gamma^{\prime}_1, \ldots, \Gamma^{\prime}_l$ such that $\Gamma^{\prime}_j \not\subset \Gamma_{\infty}$ for each $j = 1, \ldots, l$.
Then the $l$-tuple $(f|_{\Gamma^{\prime}_1}, \ldots, f|_{\Gamma^{\prime}_l})$ for any $f \in \operatorname{Rat}(\Gamma)$ is in $\bowtie_{i = 1}^l \operatorname{Rat}(\Gamma_i^{\prime})$ and the restriction map $\psi : \operatorname{Rat}(\Gamma) \to \bowtie_{i = 1}^l \operatorname{Rat}(\Gamma_i^{\prime}); f \mapsto (f|_{\Gamma^{\prime}_1}, \ldots, f|_{\Gamma^{\prime}_l})$ is a surjective $\boldsymbol{T}$-algebra homomorphism.
Moreover for each natural surjective $\boldsymbol{T}$-algebra homomorphism $\pi_j : \bowtie_{i = 1}^l \operatorname{Rat}(\Gamma_i^{\prime}) \twoheadrightarrow \operatorname{Rat}(\Gamma_j^{\prime})$, the composition $\pi_j \circ \psi$ coincides with the restriction map $\operatorname{Rat}(\Gamma) \twoheadrightarrow \operatorname{Rat}(\Gamma_j^{\prime}); f \mapsto f|_{\Gamma_j^{\prime}}$.
\end{prop}

\begin{proof}
The existence and definition of $\psi$ is clear by Propositions~\ref{prop3-2-1} and \ref{prop2-8-1}.
Since $\Gamma_i^{\prime} \cap \Gamma^{\prime}_j = \varnothing$ for $i \not= j$, for any $f_i \in \operatorname{Rat}(\Gamma^{\prime}_i) \setminus \{ -\infty \}$, there exists $g_i \in \operatorname{Rat}(\Gamma)$ such that $g_i|_{\Gamma^{\prime}_i} = f_i$ and $g|_{\Gamma^{\prime}_j} = 0$ for $j \not= i$.
Then $\psi(g_i) = (0, \ldots, 0, f_i, 0, \ldots, 0)$, where $f_i$ is the $i$th component, and thus $\psi(g_1 \odot \cdots \odot g_l) = \psi(g_1) \odot \cdots \odot \psi(g_l) = (f_1, 0, \ldots, 0) \odot \cdots \odot (0, \ldots, 0, f_l) = (f_1, \ldots, f_l)$.
This means that $\psi$ is surjective.
\end{proof}

\begin{prop}
    \label{prop3-2-4}
Let $\Gamma$ and $\Gamma_i$ be tropical curves and $\psi : \operatorname{Rat}(\Gamma) \twoheadrightarrow \bowtie_{i = 1}^l \operatorname{Rat}(\Gamma_i)$ a surjective $\boldsymbol{T}$-algebra homomorphism.
Then the corresponding injective map $\iota : \bigsqcup_{i = 1}^l \Gamma_i^{\prime} \hookrightarrow \Gamma$ satisfying $\psi(f) = f \circ \iota$ for any $f \in \operatorname{Rat}(\Gamma)$ is a local isometry and each connected component of its image $\operatorname{Im}(\iota)$ is not contained in $\Gamma_{\infty}$.
\end{prop}

\begin{proof}
It is clear by \cite[Lemma~3.41]{JuAe=Nakajima}, \cite[Proposition~3.11 and Theorem~3.14]{JuAe4} and Proposition~\ref{prop3-2-2}.
\end{proof}

\begin{rem}
    \label{rem3-2-1}
\upshape{
From Propositions~\ref{prop3-2-3} and \ref{prop3-2-4}, we improve the definition of subgraphs of tropical curves by adding the condition that they have no connected components consisting of only one point at infinity.
In what follows, we always use the word ``subgraphs of tropical curves" in this manner.
And note that these mean that we can regard such subgraphs as disconnected tropical curves embedded in a tropical curve.
Also we can generalize these propositions by replacing $\Gamma_1$ and $\Gamma_2$ with disconnected tropical curves respectively, but in this paper we do not use it, and so we omit it.
}
\end{rem}

\subsection{Gluing and decomposition}
    \label{subsection3.3}

In this subsection, we give the notion of gluing tropical curves along subgraphs, and translate this into algebraic language.

\begin{dfn}
    \label{dfn3-3-1}
\upshape{
Let $\Gamma_i$ and $\Gamma^{\prime}$ be tropical curves, respectively, and $\psi_i : \operatorname{Rat}(\Gamma_i) \twoheadrightarrow \operatorname{Rat}(\Gamma^{\prime})$ a surjective $\boldsymbol{T}$-algebra homomorphism.
For the natural surjective $\boldsymbol{T}$-algebra homomorphism $\pi_i : \operatorname{Rat}(\Gamma_1) \bowtie \operatorname{Rat}(\Gamma_2) \twoheadrightarrow \operatorname{Rat}(\Gamma_i); (f_1, f_2) \mapsto f_i$, we define
\begin{align*}
&~\operatorname{Rat}(\Gamma_1) \bowtie_{\operatorname{Rat}(\Gamma^{\prime})}\operatorname{Rat}(\Gamma_2)\\
:=&~\{ h \in \operatorname{Rat}(\Gamma_1) \bowtie \operatorname{Rat}(\Gamma_2) \,|\, \psi_1(\pi_1(h)) = \psi_2(\pi_2(h))\}.
\end{align*}
}
\end{dfn}

As in the setting of Definition~\ref{dfn3-3-1}, let $\iota_i : \Gamma^{\prime} \hookrightarrow \Gamma_i$ be the injective morphism that is a local isometry such that $\psi_i = \iota_i^{\ast}$ and $\Gamma_1 \sqcup_{\Gamma^{\prime}} \Gamma_2$ the tropical curve obtained by gluing $\Gamma_1$ and $\Gamma_2$ via $\iota_1$ and $\iota_2$.
More explicitly, $\Gamma_1 \sqcup_{\Gamma^{\prime}} \Gamma_2$ is the quotient set of the disconnected tropical curve $\Gamma_1 \sqcup \Gamma_2$ by the equivalence relation $\sim$ generated by
\begin{align*}
x_1 \sim x_2 \quad \Longleftrightarrow \quad x_1 \in \iota_1(\Gamma^{\prime}) \text{ and } x_2 = \iota_2(\iota_1^{-1}(x_1))
\end{align*}
\noindent
for $x_i \in \Gamma_i$.
Since $\iota_i$ is an injective local isometry, $\Gamma_1 \sqcup_{\Gamma^{\prime}} \Gamma_2$ has the natural topology and metric structure that inherit those of $\Gamma_i$.

\begin{prop}
    \label{prop3-3-1}
In the above setting, the set $\operatorname{Rat}(\Gamma_1) \bowtie_{\operatorname{Rat}(\Gamma^{\prime})} \operatorname{Rat}(\Gamma_2)$ is a subsemifiled over $\boldsymbol{T}$ of $\operatorname{Rat}(\Gamma_1) \bowtie \operatorname{Rat}(\Gamma_2)$ isomorphic to $\operatorname{Rat}(\Gamma_1 \sqcup_{\Gamma^{\prime}} \Gamma_2)$ as a $\boldsymbol{T}$-algebra.
\end{prop}

\begin{proof}
Since $\operatorname{Rat}(\Gamma_1) \bowtie \operatorname{Rat}(\Gamma_2)$ is a semifield over $\boldsymbol{T}$ via the diagonal map $\boldsymbol{T} \hookrightarrow \operatorname{Rat}(\Gamma_1) \bowtie \operatorname{Rat}(\Gamma_2); t \mapsto (t, t)$ and the composition $\psi_i \circ \pi_i$ is a $\boldsymbol{T}$-algebra homomorphism, $\operatorname{Rat}(\Gamma_1) \bowtie_{\operatorname{Rat}(\Gamma^{\prime})}\operatorname{Rat}(\Gamma_2)$ contains $\boldsymbol{T}$ naturally and is closed by the operations taking tropical sum $\oplus$, tropical multiplication $\odot$, and tropical division $^{\odot (-1)}$, and thus is a subsemifiled over $\boldsymbol{T}$ of $\operatorname{Rat}(\Gamma_1) \bowtie \operatorname{Rat}(\Gamma_2)$.

For $f \in \operatorname{Rat}(\Gamma_1) \bowtie_{\operatorname{Rat}(\Gamma^{\prime})} \operatorname{Rat}(\Gamma_2)$, let $\psi(f)$ be the function on $\Gamma_1 \sqcup_{\Gamma^{\prime}} \Gamma_2$ defined by
\begin{align*}
\psi(f) = \begin{cases}
\psi_1(\pi_1(f)) \quad \text{on} \quad (\Gamma_1 \setminus \iota_1(\Gamma^{\prime})) \subset \Gamma_1 \sqcup_{\Gamma^{\prime}} \Gamma_2,\\
\psi_2(\pi_2(f)) \quad \text{on} \quad (\Gamma_2 \setminus \iota_2(\Gamma^{\prime})) \subset \Gamma_1 \sqcup_{\Gamma^{\prime}} \Gamma_2,\\
\psi_1(\pi_1(f)) \quad \text{on} \quad \iota_1(\Gamma^{\prime}) = \iota_2(\Gamma^{\prime})  \subset \Gamma_1 \sqcup_{\Gamma^{\prime}} \Gamma_2.
\end{cases}
\end{align*}
This is well-defined and is a rational function on $\Gamma_1 \sqcup_{\Gamma^{\prime}} \Gamma_2$ by definition.
Clearly this map $\psi : \operatorname{Rat}(\Gamma_1) \bowtie_{\operatorname{Rat}(\Gamma^{\prime})} \operatorname{Rat}(\Gamma_2) \to \operatorname{Rat}(\Gamma_1 \sqcup_{\Gamma^{\prime}} \Gamma_2)$ is a $\boldsymbol{T}$-algebra isomorphism.
\end{proof}

By this proposition, we know that $\operatorname{Rat}(\Gamma_1) \bowtie_{\operatorname{Rat}(\Gamma^{\prime})} \operatorname{Rat}(\Gamma_2)$ is identified with the glued tropical curve $\Gamma_1 \sqcup_{\Gamma^{\prime}} \Gamma_2$ through Corollary~\ref{cor3-1-1}, and the natural inclusion $\operatorname{Rat}(\Gamma_1) \bowtie_{\operatorname{Rat}(\Gamma^{\prime})} \operatorname{Rat}(\Gamma_2) \hookrightarrow \operatorname{Rat}(\Gamma_1) \bowtie \operatorname{Rat}(\Gamma_2)$ is regarded as the \textit{decomposition} of the tropical curve $\Gamma_1 \sqcup_{\Gamma^{\prime}} \Gamma_2$ into $\Gamma_1$ and $\Gamma_2$ when $\Gamma^{\prime}$ is a point.
A variant of the notion of decomposition is explained in \cite[Example~2.7]{Maclagan=Rincon}, where it is written by tropical ideals.

\begin{rem}
    \label{rem3-3-1}
\upshape{
We do not give the proof, but the same thing is true by replacing $\Gamma_1, \Gamma_2$ and $\Gamma^{\prime}$ with disconnected tropical curves, respectively, in Proposition~\ref{prop3-3-1}.
}
\end{rem}

\section{(Abstract) tropical curves with parallel rays}
    \label{section4}

In this section, we extend the several definitions and results gathered and discussed in Sections~\ref{section2} and \ref{section3}.
Furthermore, we consider the notions of weights and the balancing condition in the algebraic language.
To do so, we also prepare the notion of localization.
This section corresponds to questions~(i), (ii) and (iii) in Section~\ref{section1}.

\subsection{(Abstract) tropical curves with parallel rays and rational function semifields}
    \label{subsection4.1}

We give a characterization of rational function semifields of tropical curves with parallel rays in this subsection.
To do so, we first define these notions.

\begin{dfn}
    \label{dfn4-1-1}
\upshape{
An \textit{(abstract) tropical curve with parallel rays} $\Gamma_{\operatorname{par}}$ is an (abstract) tropical curve $\Gamma$ whose rays have an equivalence relation $\sim$ such that for rays $L_1$ and $L_2$ of $\Gamma$, if $L_1 \subset L_2$, then $L_1 \sim L_2$ holds.
Then $\Gamma$ is called its \textit{underlying tropical curve} and a ray of $\Gamma$ a \textit{ray} of $\Gamma_{\operatorname{par}}$.
Two rays of $\Gamma_{\operatorname{par}}$ are \textit{parallel} if these are equivalent under $\sim$, otherwise, these are \textit{nonparallel}.

For a tropical curve $\Gamma$, a \textit{parallelization} of $\Gamma$ is to give an equivalence relation $\sim$ on rays of $\Gamma$ satisfying the condition above. 
}
\end{dfn}

In what follows, we just say tropical curves with parallel rays without ``abstract".
Clearly one tropical curve possibly has several parallelizations.

\begin{rem}
    \label{rem4-1-1}
\upshape{
Clearly, for a tropical curve $\Gamma$, for rays $L_1, L_2 \subset \Gamma$, defining $L_1 \sim_{\operatorname{triv}} L_2$ if $L_1 \subset L_2$ or $L_1 \supset L_2$, the relation $\sim_{\operatorname{triv}}$ is an equivalence relation on rays of $\Gamma$ and makes the pair $(\Gamma, \sim_{\operatorname{triv}})$ a tropical curve with parallel rays.
Hence Definition~\ref{dfn4-1-1} gives an extended notion of tropical curves.
}
\end{rem}

\begin{dfn}
    \label{dfn4-1-2}
\upshape{
Let $\Gamma_{\operatorname{par}}$ be a tropical curve with parallel rays and $\Gamma$ the underlying tropical curve of $\Gamma_{\operatorname{par}}$.
A function $f : \Gamma \to \boldsymbol{R} \cup \{ \pm \infty \}$ is a \textit{rational function} on $\Gamma_{\operatorname{par}}$ if, $f \in \operatorname{Rat}(\Gamma)$ and $f$ has the same slope at infinity along any pair of parallel rays of $\Gamma_{\operatorname{par}}$, or $f \equiv - \infty$.
Let $\operatorname{Rat}(\Gamma_{\operatorname{par}})$ denote the set of all rational functions on $\Gamma_{\operatorname{par}}$, which is a subset of $\operatorname{Rat}(\Gamma)$.
We call $\operatorname{Rat}(\Gamma_{\operatorname{par}})$ the \textit{rational function semifield} of $\Gamma_{\operatorname{par}}$.
}
\end{dfn}

\begin{prop}[{cf.~\cite[Theorem~1.1]{JuAe3}}]
    \label{prop4-1-1}
For a tropical curve with parallel rays $\Gamma_{\operatorname{par}}$ and its underlying tropical curve $\Gamma$, the set $\operatorname{Rat}(\Gamma_{\operatorname{par}})$ is a finitely generated subsemifield over $\boldsymbol{T}$ of $\operatorname{Rat}(\Gamma)$.
\end{prop}

\begin{proof}
Clearly $\operatorname{Rat}(\Gamma_{\operatorname{par}})$ is a subsemifield over $\boldsymbol{T}$ of $\operatorname{Rat}(\Gamma)$.

Let $\Gamma^{\prime} \subset \Gamma$ be a connected subgraph which has no rays and the closure of each connected component of $\Gamma \setminus \Gamma^{\prime}$ is a ray of $\Gamma$.
Let $L_{11}, \ldots, L_{1m_1}, L_{21}, \ldots, L_{km_k}$ be all distinct rays of $\Gamma_{\operatorname{par}}$, each of which is the closure of a connected component of $\Gamma \setminus \Gamma^{\prime}$ and each two of $L_{i1}, \ldots, L_{im_i}$ are parallel and $L_{i1}$ and $L_{j1}$ are nonparallel for any $i \not= j$.
By definition, the chip firing moves $g_i:= \operatorname{CF} \left(\overline{\Gamma \setminus \bigcup_{j = 1}^{m_i} L_{ij}}, \infty \right)$ and $h_{ij} := \operatorname{CF} \left( L_{ij}, \varepsilon \right)$ for $\varepsilon > 0$ are in $\operatorname{Rat}(\Gamma_{\operatorname{par}})$.
Let $x_{ij}$ be the unique boundary point of $L_{ij}$ in $\Gamma$.
By \cite[the proof of Lemma~1.4]{JuAe3}, for $f \in \operatorname{Rat}(\Gamma_{\operatorname{par}}) \setminus \{ -\infty \}$, the subsemifield over $\boldsymbol{T}$ of $\operatorname{Rat}(\Gamma_{\operatorname{par}})$ generated by $g_i$ contains a rational function $G_{ij}$ which coincides with $f(x_{ij})^{\odot (-1)} \odot f$ on $L_{ij}$ and is constant zero on $\overline{\Gamma \setminus \bigcup_{j = 1}^{m_i} L_{ij}}$.
Since $f$ has the same slope at infinity along $L_{i1}, \ldots, L_{im_i}$ and has only a finite number of pieces, there exists a positive integer $l_{ij} \in \boldsymbol{Z}_{>0}$ such that
\begin{align*}
G_{ij} \odot h_{ij}^{\odot l_{ij}} \le G_{i1}, \ldots, G_{im_i} \quad \text{on} \quad \Gamma \setminus L_{ij}.
\end{align*}
For the injective $\boldsymbol{T}$-algebra homomorphism $\psi : \operatorname{Rat}(\Gamma^{\prime}) \hookrightarrow \operatorname{Rat}(\Gamma)$ and the surjective homomorphism $\phi : \operatorname{Rat}(\Gamma) \twoheadrightarrow \operatorname{Rat}(\Gamma^{\prime})$ such that $\phi \circ \psi$ is the identity map $\operatorname{id}_{\operatorname{Rat}(\Gamma^{\prime})}$ of $\operatorname{Rat}(\Gamma^{\prime})$, the image $\psi(\phi(f))$ is equal to $f$ on $\overline{\Gamma \setminus \bigcup_{j = 1}^{m_i} L_{ij}}$ and is constant $f(x_{ij})$ on $L_{ij}$.
Note that these are the pull-back maps of the surjective morphism $\varphi : \Gamma \twoheadrightarrow \Gamma^{\prime}$ and the injective morphism $\iota : \Gamma^{\prime} \hookrightarrow \Gamma$ such that $\varphi \circ \iota = \operatorname{id}_{\Gamma^{\prime}}$, which exist by definition.
This implies that
\begin{equation}
\label{prop4-1-1:eq1}
\psi(\phi(f)) \odot \bigodot_{i = 1}^k  \left( \bigoplus_{j = 1}^{m_i} G_{ij} \odot h_{ij}^{\odot l_{ij}} \right) \odot \psi \left( \phi \left( \bigodot_{i = 1}^k  \left( \bigoplus_{j = 1}^{m_i} G_{ij} \odot h_{ij}^{\odot l_{ij}} \right) \right) \right)^{\odot (-1)} = f.
\end{equation}
Hence $\operatorname{Rat}(\Gamma_{\operatorname{par}})$ is generated by all $g_i, h_{ij}$ and the images of a finite generating set of $\operatorname{Rat}(\Gamma^{\prime})$ by $\psi$, which exists by \cite[Proposition~3.6]{JuAe3}, in particular, $\operatorname{Rat}(\Gamma_{\operatorname{par}})$ is finitely generated as a semifield over $\boldsymbol{T}$.
\end{proof}

\begin{prom}
    \label{prom4-1-1}
\upshape{
Usually, two rays $L_1$ and $L_2$ in $\boldsymbol{R}^n$ are said to be parallel if $\boldsymbol{d}_1 = \boldsymbol{d}_2$ or $\boldsymbol{d}_1 = - \boldsymbol{d}_2$ holds for their unit direction vectors (toward infinity) $\boldsymbol{d}_1, \boldsymbol{d}_2 \in \boldsymbol{R}^n$.
But in this paper, it is more convenient to say that $L_1$ and $L_2$ are parallel just when $\boldsymbol{d}_1 = \boldsymbol{d}_2$.
In what follows, when we say that two rays in $\boldsymbol{R}^n$ are \textit{parallel}, these have the same unit direction vector.
}
\end{prom}

By Proposition~\ref{prop4-1-1}, for a tropical curve with parallel rays $\Gamma_{\operatorname{par}}$, there exist generators $f_1, \ldots, f_n \in \operatorname{Rat}(\Gamma_{\operatorname{par}}) \setminus \{ -\infty \}$ of $\operatorname{Rat}(\Gamma_{\operatorname{per}})$ as a semifield over $\boldsymbol{T}$.
Then we have the surjective $\boldsymbol{T}$-algebra homomorphism $\psi : \overline{\boldsymbol{T}(\boldsymbol{X})}_n \twoheadrightarrow \operatorname{Rat}(\Gamma_{\operatorname{par}})$ induced by the correspondence $X_i \mapsto f_i$ as in \cite[the proof of Lemma~3.10]{JuAe4}.
Hence the quotient semifield $\overline{\boldsymbol{T}(\boldsymbol{X})}_n / \operatorname{Ker}(\psi)$ is isomorphic to $\operatorname{Rat}(\Gamma_{\operatorname{par}})$ by \cite[Proposition~2.4.4]{Giansiracusa=Giansiracusa}.

\begin{prop}[cf.~{\cite[Proposition~3.11]{JuAe4}}]
    \label{prop4-1-2}
In the above setting,

$(1)$ $\operatorname{Ker}(\psi) = \boldsymbol{E}(\boldsymbol{V}(\operatorname{Ker}(\psi)))$ holds,

$(2)$ $\operatorname{Im}(\theta) = \boldsymbol{V}(\operatorname{Ker}(\psi))$ holds,

$(3)$ $\theta$ is an injective local isometry for the lattice length, and

$(4)$ $\theta$ preserves the equivalence relation for parallelity of rays, i.e., for rays $L_1$ and $L_2$ of $\Gamma_{\operatorname{par}}$, if these are parallel, then both $\theta(L_1 \setminus \Gamma_{\infty})$ and $\theta(L_2 \setminus \Gamma_{\infty})$ are connected sets consisting of finitely many line segments (possibly degenerating to a point) and parallel rays in $\boldsymbol{R}^n$, otherwise, these are connected sets consisting of finitely many line segments (possibly degenerating to a point) and nonparallel rays in $\boldsymbol{R}^n$, where $\Gamma$ denotes the underlying tropical curve of $\Gamma_{\operatorname{par}}$.
\end{prop}

\begin{proof}
By \cite[the proof of Proposition~3.11]{JuAe4}, three assertions $(1), (2)$ and $(3)$ hold.
Assertion $(4)$ immediately follows from the definitions of rational functions on tropical curves with parallel rays and $\theta$.
\end{proof}

The following is clear by \cite[the proof of Proposition~3.18]{JuAe5}:

\begin{prop}
    \label{prop4-1-3}
Any rational function $f \in \overline{\boldsymbol{T}(\boldsymbol{X})}_n \setminus \{ -\infty \}$ has the same slope at infinity along any pair of parallel rays in $\boldsymbol{R}^n$.
\end{prop}

\begin{cor}
    \label{cor4-1-1}
For $f_1, \ldots, f_m \in \overline{\boldsymbol{T}(\boldsymbol{X})}_n \setminus \{ -\infty \}$, let $\theta : \boldsymbol{R}^n \to \boldsymbol{R}^m; x \mapsto (f_1(x), \ldots, f_m(x))$.
Then $\theta$ maps parallel rays $L_1$ and $L_2$ in $\boldsymbol{R}^n$ that are $\boldsymbol{R}$-rational as polyhedral sets to

$(1)$ both connected sets consisting of finitely many line segments (possibly degenerating to a point) that are $\boldsymbol{R}$-rational as polyhedral sets, or

$(2)$ both connected sets consisting of finitely many line segments (possibly degenerating to a point) and parallel rays that are $\boldsymbol{R}$-rational as polyhedral sets and the expansion factors of $\theta$ on $L_1$ and $L_2$ at infinity coincide.
\end{cor}

In the setting of Corollary~\ref{cor4-1-1}, the \textit{expansion factor} of $\theta$ on a segment $\overline{xy}$ that is an $\boldsymbol{R}$-rational polyhedral set in $\boldsymbol{R}^n$ means the ratio of the lattice length of $\overline{\theta(x)\theta(y)}$ to that of $\overline{xy}$.
Note that this corollary is already implicitly used to give Theorem~\ref{thm1-1}, but it plays an essential role in this paper, so we put the statement here explicitly.

By Proposition~\ref{prop4-1-2}, Corollary~\ref{cor4-1-1} and \cite[Theorem~3.14]{JuAe4}, we can characterize rational function semifields of tropical curves with parallel rays as Theorem~\ref{thm1-1}:

\begin{thm}
    \label{thm4-1-1}
Let $S$ be a $\boldsymbol{T}$-algebra.
There exists a tropical curve with parallel rays $\Gamma_{\operatorname{par}}$ whose rational function semifield $\operatorname{Rat}(\Gamma_{\operatorname{par}})$ is isomorphic to $S$ as a $\boldsymbol{T}$-algebra if and only if there exists a surjective $\boldsymbol{T}$-algebra homomorphism $\psi$ from a tropical rational function semifield to $S$ satisfying the following four conditions:

$(1)$ $\operatorname{Ker}(\psi)$ is finitely generated as a congruence,

$(2)$ $\operatorname{Ker}(\psi) = \boldsymbol{E}(\boldsymbol{V}(\operatorname{Ker}(\psi)))$ holds,

$(3)$ $\boldsymbol{V}(\operatorname{Ker}(\psi))$ is connected, and

$(4)$ $\boldsymbol{V}(\operatorname{Ker}(\psi))$ is of dimension zero or one.

These are also equivalent to that any surjective $\boldsymbol{T}$-algebra homomorphism $\psi$ from a tropical rational function semifield to $S$ satisfies four conditions $(1), \ldots, (4)$ above.
\end{thm}

Note that by \cite[Theorem~1.1]{JuAe5} and \cite[Corollary~3.30]{JuAe=Nakajima}, or as explained in Subsection~\ref{subsection2.3}, in Theorem~\ref{thm4-1-1}, under conditions $(1)$ and $(2)$, condition $(4)$ is equivalent to that the Krull dimension of $\overline{\boldsymbol{T}(\boldsymbol{X})}_n / \operatorname{Ker}(\psi)$ is one or two.

One might be dissatisfied with condition $(3)$ in Theorem~\ref{thm4-1-1} written in geometric language and wonder if we can replace condition $(3)$ with some condition using the pseudodirect product of semifields over $\boldsymbol{T}$ as in \cite[Section~3]{JuAe=Nakajima}.
However, this idea does not work in our setting now, see Subsection~\ref{subsection4.3}.
To solve this problem here, we give the following proposition:

\begin{prop}
    \label{prop4-1-4}
Let $S$ be a $\boldsymbol{T}$-algebra and assume that there exists a surjective $\boldsymbol{T}$-algebra homomorphism $\psi : \overline{\boldsymbol{T}(\boldsymbol{X})}_n \twoheadrightarrow S$.
When $\operatorname{Ker}(\psi)$ is finitely generated as a congruence and $\operatorname{Ker}(\psi) = \boldsymbol{E}(\boldsymbol{V}(\operatorname{Ker}(\psi)))$ holds, the following are equivalent:

$(1)$ $\boldsymbol{V}(\operatorname{Ker}(\psi))$ is nonempty and disconnected, and

$(2)$ there exist $s \in S \setminus \{ 0_S \}$ and $a_i \in \boldsymbol{R}$ such that $a_3 < a_2 < a_1$ and $a_1 - a_2 = a_2 - a_3$ and $s \not= s \oplus a_3, \left( s^{\odot (-1)} \oplus a_1^{\odot (-1)} \right)^{\odot (-1)}$ and $(s \oplus a_1) \odot \left(s^{\odot (-1)} \oplus a_2^{\odot (-1)} \right)^{\odot (-1)} = \left(a_1 \odot a_2^{\odot (-1)} \right) \odot (s \oplus a_2) \odot \left(s^{\odot (-1)} \oplus a_3^{\odot (-1)} \right)^{\odot (-1)}$.
\end{prop}

To see this proposition, we prepare a clear statement:

\begin{prop}
    \label{prop4-1-5}
Let $S$ be a $\boldsymbol{T}$-algebra and $\psi : \overline{\boldsymbol{T}(\boldsymbol{X})}_n \twoheadrightarrow S$ a surjective $\boldsymbol{T}$-algebra homomorphism.
When $\operatorname{Ker}(\psi) = \boldsymbol{E}(\boldsymbol{V}(\operatorname{Ker}(\psi)))$ holds, the following are equivalent:

$(1)$ $\boldsymbol{V}(\operatorname{Ker}(\psi))$ is empty, and

$(2)$ $\operatorname{Ker}(\psi)$ is improper.
\end{prop}

\begin{proof}[Proof of Proposition~\ref{prop4-1-4}]
Assume that $(1)$ holds.
By \cite[Corollary~3.5]{JuAe5} and $(1)$, there exist nonempty finite unions $V_1$ and $V_2$ of $\boldsymbol{R}$-rational polyhedral sets such that $\boldsymbol{V}(\operatorname{Ker}(\psi)) = V_1 \cup V_2$ and $V_1 \cap V_2 = \varnothing$ and $V_1$ is connected.
Then, by \cite[Lemmas~3.9 and 3.10]{JuAe5}, there exists $f \in \overline{\boldsymbol{T}(\boldsymbol{X})}_n$ such that $\boldsymbol{V}((f,0)) = V_1$ and $f(x) \ge \operatorname{dist}(V_1, x)$ holds for any $x \in \boldsymbol{R}^n$.
Thus for a positive number $\varepsilon$ such that $0 < \varepsilon < \frac{1}{3} \operatorname{dist}(V_1, V_2)$, taking as $a_1 := 3\varepsilon, a_2 := 2\varepsilon, a_3 := \varepsilon$, the equalities $a_1 - a_2 = a_2 - a_3 = \varepsilon$ and
\begin{align*}
&~([f] \oplus a_1) \odot \left([f]^{\odot (-1)} \oplus a_2^{\odot (-1)} \right)^{\odot (-1)}\\
=&~ \left( a_1 \odot a_2^{\odot (-1)} \right) \odot ([f] \oplus a_2) \odot \left( [f]^{\odot (-1)} \oplus a_3^{\odot (-1)} \right)^{\odot (-1)}
\end{align*}
hold, where $[f]$ denotes the equivalence class of $f$ in $\overline{\boldsymbol{T}(\boldsymbol{X})}_n / \operatorname{Ker}(\psi)$.
For the latter equality, since $\operatorname{Ker}(\psi) = \boldsymbol{E}(\boldsymbol{V}(\operatorname{Ker}(\psi)))$, it is enough to see that at any $x \in \boldsymbol{V}(\operatorname{Ker}(\psi)) = V_1 \sqcup V_2$, the values of both sides coincide, and it is easily checked.
The inequalities $[f] \not= [f] \oplus a_3, \left( [f]^{\odot (-1)} \oplus a_1^{\odot (-1)} \right)^{\odot (-1)}$ are also clear.
By \cite[Proposition~2.4.4]{Giansiracusa=Giansiracusa}, the image $\psi(f)$ is a desired element $s \in S \setminus \{ 0_S \}$, that is, $(2)$ holds.

Assume $(2)$ holds.
Again by \cite[Proposition~2.4.4]{Giansiracusa=Giansiracusa}, there exists $f \in \overline{\boldsymbol{T}(\boldsymbol{X})}_n \setminus \{ -\infty \}$ such that $\psi(f) = s$.
As $\operatorname{Ker}(\psi) = \boldsymbol{E}(\boldsymbol{V}(\operatorname{Ker}(\psi)))$ and $s \not= s \oplus a_3$, by Proposition~\ref{prop4-1-5} and \cite[Proposition~2.4.4]{Giansiracusa=Giansiracusa}, the congruence variety $\boldsymbol{V}(\operatorname{Ker}(\psi))$ is nonempty.
Assume that $\boldsymbol{V}(\operatorname{Ker}(\psi))$ is connected.
If it consists of only a point $x \in \boldsymbol{R}^n$, then, by the inequality $s \not= s \oplus a_3$, the value $[f](x) = f(x)$ is less than $a_3$, but this contradicts $s \not= \left( s^{\odot (-1)} \oplus a_1^{\odot (-1)} \right)^{\odot (-1)}$ as $[f](x) = \left( [f](x)^{\odot (-1)} \oplus a_1^{\odot (-1)} \right)^{\odot (-1)}$.
Hence $\boldsymbol{V}(\operatorname{Ker}(\psi))$ contains at least two points.
Because $\boldsymbol{V}(\operatorname{Ker}(\psi))$ is a finite union of $\boldsymbol{R}$-rational polyhedral sets by \cite[Corollary~3.5]{JuAe5} and connected, it must consist of infinitely many points.
By the inequalities in $(2)$, there are $y, z \in \boldsymbol{V}(\operatorname{Ker}(\psi))$ such that $[f](y) < a_3$ and $[f](z) > a_1$.
As $[f]$ is continuous on $\boldsymbol{V}(\operatorname{Ker}(\psi))$ and $\boldsymbol{V}(\operatorname{Ker}(\psi))$ is connected, there exists $x \in \boldsymbol{V}(\operatorname{Ker}(\psi))$ such that $a_3 < [f](x) < a_2$.
Then
\begin{align*}
&~([f](x) \oplus a_1) \odot \left( [f](x)^{\odot (-1)} \oplus a_2^{\odot (-1)} \right)^{\odot (-1)}\\
=&~ a_1 \odot [f](x)\\
\not=&~a_1 \odot a_3\\
=&~a_1 \odot a_2^{\odot (-1)} \odot a_2 \odot a_3\\
=&~\left(a_1 \odot a_2^{\odot (-1)} \right) \odot ([f](x) \oplus a_2) \odot \left( [f](x)^{\odot (-1)} \oplus a_3^{\odot (-1)} \right)^{\odot (-1)},
\end{align*}
which is a contradiction.
In conclusion, $(1)$ holds.
\end{proof}

Condition $(4)$ in Theorem~\ref{thm4-1-1} guaranties that $\boldsymbol{V}(\operatorname{Ker}(\psi))$ is nonempty.
Therefore, in Theorem~\ref{thm4-1-1}, condition $(3)$ is equivalent to that there exist no $s \in S \setminus \{ 0_S \}$ and $a_i \in \boldsymbol{R}$ satisfying the conditions in $(2)$ in Propsition~\ref{prop4-1-4}.

\begin{cor}
    \label{cor4-1-2}
Let $\Gamma_{\operatorname{par}}$ be a tropical curve with parallel rays.
For $f_1, \ldots, f_n \in \operatorname{Rat}(\Gamma_{\operatorname{par}}) \setminus \{ -\infty \}$ and the $\boldsymbol{T}$-algebra homomorphism $\psi : \overline{\boldsymbol{T}(\boldsymbol{X})}_n \to \operatorname{Rat}(\Gamma_{\operatorname{par}})$ induced by the correspondence $X_i \mapsto f_i$, the subsemifield $\boldsymbol{T}(f_1, \ldots, f_n)$ over $\boldsymbol{T}$ of $\operatorname{Rat}(\Gamma_{\operatorname{par}})$ generated by $f_1, \ldots, f_n$ is isomorphic to $\overline{\boldsymbol{T}(\boldsymbol{X})}_n / \operatorname{Ker}(\psi)$ as a $\boldsymbol{T}$-algebra.
Moreover, $\overline{\boldsymbol{T}(\boldsymbol{X})}_n / \operatorname{Ker}(\psi)$ is a semifield over $\boldsymbol{T}$ in Theorem~\ref{thm4-1-1} and $\boldsymbol{V}(\operatorname{Ker}(\psi))$ coincides with the image $\operatorname{Im}(\theta)$ of the map $\theta : \Gamma_{\operatorname{par}} \setminus \Gamma_{\operatorname{par}, \infty} \to \boldsymbol{R}^n; x \mapsto (f_1(x), \ldots, f_n(x))$.
\end{cor}

\begin{proof}
The first statement is clear by \cite[Proposition~2.4.4]{Giansiracusa=Giansiracusa}.
By Theorem~\ref{thm4-1-1} and \cite[Corollary~3.15]{JuAe4}, the kernel congruence $\operatorname{Ker}(\psi)$ coincides with $\boldsymbol{E}(\boldsymbol{V}(\operatorname{Ker}(\psi)))$.
By the definition of $\theta$, the image $\operatorname{Im}(\theta)$ is closed in $\boldsymbol{R}^n$.
Thus $\operatorname{Im}(\theta) = \boldsymbol{V}(\operatorname{Ker}(\psi))$ holds by \cite[Theorem~3.14]{JuAe4}, and hence it is a (connected) finite union of $\boldsymbol{R}$-rational polyhedral sets of dimension zero or one by the definition of $\theta$ and rational functions on tropical curves with parallel rays.
Therefore $\operatorname{Ker}(\psi)$ is finitely generated as a congruence by \cite[Theorem~1.1]{JuAe5}.
\end{proof}

\begin{rem}
    \label{rem4-1-2}
\upshape{
Proposition~\ref{prop4-1-1} and Corollary~\ref{cor4-1-2} verify that all rational function semifields of tropical curves with parallel rays are finitely generated subsemifields over $\boldsymbol{T}$ of rational function semifileds of tropical curves and vise versa.
This is also a characterization of them.
In addition, this also means that a finitely generated subsemifield over $\boldsymbol{T}$ of rational function semifileds of tropical curves with parallel rays is the rational function semifield of some tropical curve with parallel rays.
This is an analogue of the classical fact that any finitely generated subalgebra of an affine coordinate ring is itself an affine coordinate ring, which is a (partial) answer to question~(iii) in Section~\ref{section1}.
}
\end{rem}

For a tropical curve with parallel rays $\Gamma_{\operatorname{par}}$, in what follows, we say a \textit{point} of $\Gamma_{\operatorname{par}}$ to indicate the corresponding point of the underlying tropical curve of $\Gamma_{\operatorname{par}}$.
The same is true for other words or notations other than when we specify.

\subsection{Morphisms and categorical equivalence}
    \label{subsection4.2}

The properties given in Corollary~\ref{cor4-1-1} define morphisms between tropical curves with parallel rays as follows:

\begin{dfn}
    \label{dfn4-2-1}
\upshape{
Let $\Gamma_{\operatorname{par}}$ and $\Gamma^{\prime}_{\operatorname{par}}$ be two tropical curves with parallel rays.

A map $\varphi : \Gamma_{\operatorname{par}} \to \Gamma^{\prime}_{\operatorname{par}}$ is a \textit{morphism} if

$(1)$ $\varphi$ is a morphism between the underying tropical curves, and

$(2)$ for any parallel rays $L_1$ and $L_2$ of $\Gamma_{\operatorname{par}}$, if one of the images $\varphi(L_1)$ and $\varphi(L_2)$ has no ray, then so do the other, otherwise, these are parallel rays (possibly coincide) of $\Gamma^{\prime}_{\operatorname{par}}$ and the degrees of $\varphi$ on $L_1$ and $L_2$ at infinity coincide.

A morphism $\varphi : \Gamma_{\operatorname{par}} \to \Gamma^{\prime}_{\operatorname{par}}$ is an \textit{isomorphism} if it is a bijective morphism and the inverse map $\varphi^{-1} : \Gamma^{\prime}_{\operatorname{par}} \to \Gamma_{\operatorname{par}}$ is also a morphism between tropical curves with parallel rays.
}
\end{dfn}

Note that by definition, a morphism between tropical curves with parallel rays $\varphi : \Gamma_{\operatorname{par}} \to \Gamma^{\prime}_{\operatorname{par}}$ is an isomorphism if and only if $\varphi$ is continuous on the whole of $\Gamma_{\operatorname{par}}$ and induces an isometry $\Gamma_{\operatorname{par}} \setminus \Gamma_{\operatorname{par}, \infty} \twoheadrightarrow \Gamma^{\prime}_{\operatorname{par}} \setminus \Gamma^{\prime}_{\operatorname{par}, \infty}$ (and hence, $\varphi$ is bijective) and for rays $L_1$ and $L_2$ of $\Gamma_{\operatorname{par}}$, these are parallel if and only if so are the images $\varphi(L_1)$ and $\varphi(L_2)$ by $\varphi$.

The reader might wonder whether the condition on the degrees along parallel rays can be omitted for morphisms between tropical curves with parallel rays.
However, we saw in Corollary~\ref{cor4-1-1} that this condition is natural for maps defined by a finite number of rational functions, and we also know in the proof of the following proposition that without this condition we cannot define the pull-back map.

\begin{prop}
    \label{prop4-2-1}
Let $\Gamma_{\operatorname{par}}$ and $\Gamma^{\prime}_{\operatorname{par}}$ be tropical curves with parallel rays and $\Gamma$ and $\Gamma^{\prime}$ their underlying tropical curves, respectively.
For a morphism $\varphi : \Gamma \to \Gamma^{\prime}$ between tropical curves, if it is a morphism between tropical curves with parallel rays from $\Gamma_{\operatorname{par}}$ to $\Gamma^{\prime}_{\operatorname{par}}$, i.e., $\varphi$ satisfies condition $(2)$ in Definition~\ref{dfn4-2-1}, then the restriction of $\varphi^{\ast}$ on $\operatorname{Rat}(\Gamma_{\operatorname{par}})$ gives a $\boldsymbol{T}$-algebra homomorphism $\operatorname{Rat}(\Gamma^{\prime}_{\operatorname{par}}) \to \operatorname{Rat}(\Gamma_{\operatorname{par}})$.
\end{prop}

\begin{proof}
Let $L_1$ and $L_2$ be parallel rays of $\Gamma_{\operatorname{par}}$.
For $f^{\prime} \in \operatorname{Rat}(\Gamma^{\prime}_{\operatorname{par}})$, when $\varphi(L_1)$ and $\varphi(L_2)$ are parallel rays, since $L_1$ and $L_2$ are parallel and the degrees of $\varphi$ on $L_1$ and $L_2$ at infinity coincide, the slopes of $\psi(f^{\prime})$ at infinity along $\varphi(L_1)$ and $\varphi(L_2)$ coincide, that is, $\varphi^{\ast}(f^{\prime}) \in \operatorname{Rat}(\Gamma_{\operatorname{par}})$.
\end{proof}

As in Proposition~\ref{prop4-3-2}, for a morphism between tropical curves with parallel rays $\varphi : \Gamma_{\operatorname{par}} \to \Gamma^{\prime}_{\operatorname{par}}$, we write the $\boldsymbol{T}$-algebra homomorphism $\operatorname{Rat}(\Gamma^{\prime}_{\operatorname{par}}) \to \operatorname{Rat}(\Gamma_{\operatorname{par}}); f^{\prime} \mapsto f^{\prime} \circ \varphi$ as $\varphi^{\ast}$, and call it the \textit{pull-back map} or the \textit{pull-back $\boldsymbol{T}$-algebra homomorphism} of $\varphi$ again. 

The converse of Proposition~\ref{prop4-2-1} holds by the almost same proof of Theorem~\ref{thm3-1-1}:

\begin{thm}
    \label{thm4-2-1}
Let $\Gamma_{\operatorname{par}}$ and $\Gamma^{\prime}_{\operatorname{par}}$ be tropical curves with parallel rays, respectively.
For a $\boldsymbol{T}$-algebra homomorphism $\psi : \operatorname{Rat}(\Gamma^{\prime}_{\operatorname{par}}) \to \operatorname{Rat}(\Gamma_{\operatorname{par}})$, there exists a unique morphism $\varphi : \Gamma_{\operatorname{par}} \to \Gamma^{\prime}_{\operatorname{par}}$ such that $\psi = \varphi^{\ast}$.
\end{thm}

\begin{proof}
Replacing $\Gamma_2$ and $\Gamma_1$ in \cite[Corollary~A1 and its proof]{JuAe4} with $\Gamma_{\operatorname{par}}$ and $\Gamma^{\prime}_{\operatorname{par}}$, respectively, the same proof works except for the parts corresponding to the surjectivity of $\psi$ and the injectivity of $\varphi$ by Proposition~\ref{prop4-1-2} and \cite[Theorem~3.14]{JuAe4}.
Corollary~\ref{cor4-1-2} guaranties that the obtained map $\varphi$ is a morphism.
\end{proof}

The composition of morphisms between tropical curves with parallel rays is again a morphism.
Hence, a variant of Corollary~\ref{cor3-1-1} holds:

\begin{cor}
    \label{cor4-2-1}
The following categories $\mathscr{C}$ and $\mathscr{D}$ are equivalent via the following (contravariant) functors $F$ and $G$.

$(1)$ The class $\operatorname{Ob}(\mathscr{C})$ of objects of $\mathscr{C}$ is the tropical curves with parallel rays.

For $\Gamma_{\operatorname{par}, 1}, \Gamma_{\operatorname{par}, 2} \in \operatorname{Ob}(\mathscr{C})$, the set $\operatorname{Hom}_{\mathscr{C}}(\Gamma_{\operatorname{par}, 1}, \Gamma_{\operatorname{par}, 2})$ of morphisms from $\Gamma_{\operatorname{par}, 1}$ to $\Gamma_{\operatorname{par}, 2}$ consists of the morphisms $\Gamma_{\operatorname{par}, 1} \to \Gamma_{\operatorname{par}, 2}$ (in the sense of Definition~\ref{dfn4-2-1}).

$(2)$ The class $\operatorname{Ob}(\mathscr{D})$ of objects of $\mathscr{D}$ is the finitely generated semifields over $\boldsymbol{T}$ satisfying conditions $(1), \ldots, (4)$ in Theorem~\ref{thm4-1-1}.

For $S_1, S_2 \in \operatorname{Ob}(\mathscr{D})$, the set $\operatorname{Hom}_{\mathscr{D}}(S_1, S_2)$ of morphisms from $S_1$ to $S_2$ consists of the $\boldsymbol{T}$-algebra homomorphisms $S_1 \to S_2$.

The functor $F \colon \mathscr{C} \to \mathscr{D}$ maps $\Gamma_{\operatorname{par}} \in \operatorname{Ob}(\mathscr{C})$ to $\operatorname{Rat}(\Gamma_{\operatorname{par}}) \in \operatorname{Ob}(\mathscr{D})$ and for $\Gamma_{\operatorname{par}, 1}, \Gamma_{\operatorname{par}, 2} \in \operatorname{Ob}(\mathscr{C})$, maps $\varphi \in \operatorname{Hom}_{\mathscr{C}}(\Gamma_{\operatorname{par}, 1}, \Gamma_{\operatorname{par}, 2})$ to $\varphi^{\ast} \in \operatorname{Hom}_{\mathscr{D}}(\operatorname{Rat}(\Gamma_{\operatorname{par}, 2}), \operatorname{Rat}(\Gamma_{\operatorname{par}, 1}))$.

The functor $G \colon \mathscr{D} \to \mathscr{C}$ maps $S \in \operatorname{Ob}(\mathscr{D})$ to $\overline{\boldsymbol{V}(\operatorname{Ker}(\psi_S))} \in \operatorname{Ob}(\mathscr{C})$ and for $S_1, S_2 \in \operatorname{Ob}(\mathscr{D})$, maps $\widetilde{\psi} \in \operatorname{Hom}_{\mathscr{D}}(S_1, S_2)$ to $\varphi_{\widetilde{\psi}} \in \operatorname{Hom}_{\mathscr{C}}(\overline{\boldsymbol{V}(\operatorname{Ker}(\psi_{S_2}))}, \overline{\boldsymbol{V}(\operatorname{Ker}(\psi_{S_1}))})$, where $\psi_S$ is a fixed surjective $\boldsymbol{T}$-algebra homomorphism from a tropical rational function semifield to $S$ and $\overline{\boldsymbol{V}(\operatorname{Ker}(\psi_S))}$ is the natural compactification of $\boldsymbol{V}(\operatorname{Ker}(\psi_S))$ as a tropical curve with parallel rays for any $S \in \operatorname{Ob}(\mathscr{D})$ and $\varphi_{\widetilde{\psi}}$ is the unique morphism $\overline{\boldsymbol{V}(\operatorname{Ker}(\psi_{S_2}))} \to \overline{\boldsymbol{V}(\operatorname{Ker}(\psi_{S_1}))}$ such that $\widetilde{\psi} = \left(\varphi_{\widetilde{\psi}}\right)^{\ast}$ given in Theorem~\ref{thm4-2-1}.
\end{cor}

In Corollary~\ref{cor4-2-1}, the \textit{natural compactification} of $\boldsymbol{V}(\operatorname{Ker}(\psi_S))$ as a tropical curve with parallel rays means doing one-point compactifications on all points at infinity and inheriting topology, metric and parallelity of rays from those of $\boldsymbol{V}(\operatorname{Ker}(\psi_S))$ when we regard $\boldsymbol{V}(\operatorname{Ker}(\psi_S))$ as a metric space by the lattice length.

Corresponding to Proposition~\ref{prop1-1}, we can prove the following:

\begin{prop}
    \label{prop4-2-2}
Let $\Gamma_{\operatorname{par}}$ be a tropical curve with parallel rays.
For $f_1, \ldots, f_n \in \operatorname{Rat}(\Gamma_{\operatorname{par}}) \setminus \{ -\infty \}$ and the map $\theta : \Gamma_{\operatorname{par}} \setminus \Gamma_{\operatorname{par}, \infty} \to \boldsymbol{R}^n; x \mapsto (f_1(x), \ldots, f_n(x))$, the following are equivalent:

$(1)$ $f_1, \ldots, f_n$ generate $\operatorname{Rat}(\Gamma_{\operatorname{par}})$ as a semifield over $\boldsymbol{T}$,

$(2)$ $\overline{\boldsymbol{T}(\boldsymbol{X})}_n / \boldsymbol{E}(\operatorname{Im}(\theta))$ is isomorphic to $\operatorname{Rat}(\Gamma_{\operatorname{par}})$ as a $\boldsymbol{T}$-algebra, and

$(3)$ $\theta$ is an isometry on the image with the lattice length and for rays $L_1, L_2 \subset \Gamma_{\operatorname{par}}$, these are parallel if and only if $\theta(L_1 \setminus \Gamma_{\operatorname{par}, \infty})$ and $\theta(L_2 \setminus \Gamma_{\operatorname{par}, \infty})$ are connected sets consisting finitely many line segments (possibly degenerating to a point) and parallel rays in $\boldsymbol{R}^n$.
\end{prop}

\begin{proof}
Proposition~\ref{prop4-1-2} and Corollary~\ref{cor4-1-2} give $(1) \Longleftrightarrow (2) \Longrightarrow (3)$.

Assume $(3)$ holds.
When $n = 1$, it is clearly the case of original tropical curves, and so $(2)$ holds by Proposition~\ref{prop1-1}.
Assume $n \ge 2$.
For a ray $L \subset \Gamma_{\operatorname{par}}$, \cite[the proof of Lemma~3.15]{JuAe5} ensures the existence of a rational function $f \in \overline{\boldsymbol{T}(\boldsymbol{X})}_n \setminus \{ -\infty \}$ which has the same slope at infinity along any ray of $\operatorname{Im}(\theta)$ parallel to $\theta(L \setminus \Gamma_{\operatorname{par}, \infty})$.
This fact and the proof of Proposition~\ref{prop4-1-1} and \cite[Lemmas~3.16 and 3.17]{JuAe5} show that the $\boldsymbol{T}$-algebra homomorphism $\overline{\boldsymbol{T}(\boldsymbol{X})}_n \to \operatorname{Rat}(\Gamma_{\operatorname{par}})$ defined by the correspondence $X_i \mapsto f_i$ is surjective.
By Corollary~\ref{cor4-1-2}, it implies $(2)$.
\end{proof}

Note that in the setting of Proposition~\ref{prop4-2-2}, one (and hence, all) of $(1), (2), (3)$ implies that the image $\operatorname{Im}(\theta)$ coincides with the congruence variety $\boldsymbol{V}(\operatorname{Ker}(\psi))$ for the surjective $\boldsymbol{T}$-algebra homomorpshim $\psi : \overline{\boldsymbol{T}(\boldsymbol{X})}_n \twoheadrightarrow \operatorname{Rat}(\Gamma_{\operatorname{par}})$ induced by the correspondence $X_i \mapsto f_i$ by Proposition~\ref{prop4-1-2}.

\begin{rem}
    \label{rem4-2-1}
\upshape{
Propositions~\ref{prop1-1} and \ref{prop4-2-2} say that when we realize tropical curves with parallel rays as congruence varieties in Euclidean spaces, we have to respect not only their topologies and metrics, but parallelity of rays.
In these propositions, the map $\theta$ can be regarded as an ``isomorphism" of tropical curves with parallel rays by the natural compactification of $\operatorname{Im}(\theta)$ as a tropical curve with parallel rays.
To do so, in Proposition~\ref{prop1-1}, the tropical curve $\Gamma$ must be regarded as a tropical curve with paralllel rays $(\Gamma, \sim_{\operatorname{triv}})$ as in Remark~\ref{rem4-1-1}.
}
\end{rem}

From such a point of view, the following definition is quite natural:

\begin{dfn}
    \label{dfn4-2-2}
\upshape{
Let $\Gamma_{\operatorname{par}}$ be a tropical curve with parallel rays.
For $f_1, \ldots, f_n \in \operatorname{Rat}(\Gamma_{\operatorname{par}}) \setminus \{ -\infty \}$, the image of the map $\theta : \Gamma_{\operatorname{par}} \setminus \Gamma_{\operatorname{par}, \infty} \to \boldsymbol{R}^n; x \mapsto (f_1(x), \ldots, f_n(x))$ is called the \textit{realization} of $\Gamma_{\operatorname{par}}$ if $f_1, \ldots, f_n$ generate $\operatorname{Rat}(\Gamma_{\operatorname{par}})$ as a semifield over $\boldsymbol{T}$.
}
\end{dfn}

For a tropical curve $\Gamma$, we also call picking a finite number of elements $f_1, \ldots, f_n \in \operatorname{Rat}(\Gamma) \setminus \{ -\infty \}$ such that the map $\theta : \Gamma \setminus \Gamma_{\infty} \to \boldsymbol{R}^n; x \mapsto (f_1(x), \ldots, f_n(x))$ is an isometry on its image with the lattice length a \textit{parallelization} of $\Gamma$.
In this case, the inclusion $\boldsymbol{T}(f_1, \ldots, f_n) \hookrightarrow \operatorname{Rat}(\Gamma)$ is again called a \textit{parallelization} of $\Gamma$.

\begin{rem}
    \label{rem4-2-2}
\upshape{
Corollary~\ref{cor4-1-2}, Propositions~\ref{prop4-2-1}, \ref{prop4-2-2} and Theorem~\ref{thm4-2-1} tell us the following:
for a tropical curve $\Gamma$, choosing a finite number of elements $f_1, \ldots, f_n \in \operatorname{Rat}(\Gamma) \setminus \{ -\infty \}$ gives the injective composition $\operatorname{Rat}(\Gamma^{\prime}_{\operatorname{par}}) \cong \overline{\boldsymbol{T}(\boldsymbol{X})}_n / \boldsymbol{E}(\operatorname{Im}(\theta)) \cong \boldsymbol{T}(f_1, \ldots, f_n) \hookrightarrow \operatorname{Rat}(\Gamma^{\prime}) \to \operatorname{Rat}(\Gamma)$.
Here $\theta : \Gamma \setminus \Gamma_{\infty} \to \boldsymbol{R}^n; x \mapsto (f_1(x), \ldots, f_n(x))$ and $\Gamma^{\prime}_{\operatorname{par}}$ is the tropical curve with parallel rays obtained by the natural compactification of $\operatorname{Im}(\theta)$ and $\Gamma^{\prime}$ is the underlying tropical curve of $\Gamma^{\prime}_{\operatorname{par}}$.
The injective $\boldsymbol{T}$-algebra homomorphism $\operatorname{Rat}(\Gamma^{\prime}_{\operatorname{par}}) \hookrightarrow \operatorname{Rat}(\Gamma^{\prime})$ is the parallelization of $\Gamma^{\prime}$ and the images of $f_1, \ldots, f_n$ in $\operatorname{Rat}(\Gamma^{\prime}_{\operatorname{par}})$ give a realization of $\Gamma^{\prime}_{\operatorname{par}}$.
The $\boldsymbol{T}$-algebra homomorphism $\operatorname{Rat}(\Gamma^{\prime}) \to \operatorname{Rat}(\Gamma)$ is induced by $\theta$.
Also $\theta$ induces a natural equivalence relation $\sim_{\theta}$ on rays of $\Gamma$ such that for rays $L_1, L_2 \subset \Gamma$, $L_1 \sim_{\theta} L_2$ if and only if, both $\theta(L_1 \setminus \Gamma_{\infty})$ and $\theta(L_2 \setminus \Gamma_{\infty})$ do not include rays or both include rays that are parallel.
Then, for the tropical curve with parallel rays $\Gamma_{\operatorname{par}}$ defined by the pair $(\Gamma, \sim_{\theta})$, the parallelization $\operatorname{Rat}(\Gamma_{\operatorname{par}}) \hookrightarrow \operatorname{Rat}(\Gamma)$ of $\Gamma$ commutes with the $\boldsymbol{T}$-algebra homomorphisms above:
\begin{align*}
    \begin{tikzpicture}[auto]
\node (a) at (0, 1.5) {$\operatorname{Rat}(\Gamma^{\prime}_{\operatorname{par}})$};
\node (x) at (4, 1.5) {$\operatorname{Rat}(\Gamma_{\operatorname{par}})$};
\node (b) at (0, 0) {$\operatorname{Rat}(\Gamma^{\prime})$};
\node (y) at (4, 0) {$\operatorname{Rat}(\Gamma)$.};
\draw[->] (a) -- (x);
\draw[{Hooks[right]}->] (a) -- (b);
\draw[{Hooks[right]}->] (x) -- (y);
\draw[->] (b) -- (y);
\node at (2,0.75) {$\circlearrowleft$};
\end{tikzpicture}
\end{align*}
}
\end{rem}

\subsection{Disconnected tropical curves with parallel rays and direct sum}
    \label{subsection4.3}

As explained in Subsection~\ref{subsection2.8}, the notion of disconnected tropical curves is just defined by the directed sums of tropical curves.
This is so simple since we need not respect the parallelity of rays.
However, in respect of this property, the situation is more complicated, i.e., the direct sum of tropical curves with parallel rays can be regarded as a disconnected tropical curve with parallel rays, but the converse does not hold.
We explain this in this subsection.

\begin{dfn}
    \label{dfn4-3-1}
\upshape{
A \textit{disconnected tropical curve with parallel rays} $\Gamma_{\operatorname{par}}$ is a disconnected tropical curve $\Gamma = \Gamma_1 \sqcup \cdots \sqcup \Gamma_n$ possessing an equivalence relation $\sim$ on rays of $\Gamma$ such that for rays $L_1$ and $L_2$ of $\Gamma$, if $L_1 \subset L_2$, then $L_1 \sim L_2$ holds.
The disconnected tropical curve $\Gamma$ is called its \textit{underlying disconnected tropical curve}.
We call a ray of $\Gamma$ a \textit{ray} of $\Gamma_{\operatorname{par}}$.
Two rays of $\Gamma_{\operatorname{par}}$ are \textit{parallel} if these are equivalent by $\sim$, otherwise, these are \textit{nonparallel}.
}
\end{dfn}

This definition is just an extension of Definition~\ref{dfn4-1-1} and it is quite natural from a geometric point of view.
Hence the following definition is also natural:

\begin{dfn}
    \label{dfn4-3-2}
\upshape{
Let $\Gamma_{\operatorname{par}}$ be a disconnected tropical curve with parallel rays and $\Gamma$ the underlying disconnected tropical curve of $\Gamma_{\operatorname{par}}$.
A function $f : \Gamma \to \boldsymbol{R} \cup \{ \pm \infty \}$ is a \textit{rational function} on $\Gamma_{\operatorname{par}}$ if, $f \in \operatorname{Rat}(\Gamma)$ and $f$ has the same slope at infinity along any two parallel rays of $\Gamma_{\operatorname{par}}$, or $f \equiv - \infty$.
Let $\operatorname{Rat}(\Gamma_{\operatorname{par}})$ denote the set of all rational functions on $\Gamma_{\operatorname{par}}$, which is a subset of $\operatorname{Rat}(\Gamma)$.
We call $\operatorname{Rat}(\Gamma_{\operatorname{par}})$ the \textit{rational function semifield} of $\Gamma_{\operatorname{par}}$.
}
\end{dfn}

In the setting of Definition~\ref{dfn4-3-2}, clearly $\operatorname{Rat}(\Gamma_{\operatorname{par}})$ is a subsemifield over $\boldsymbol{T}$ of $\operatorname{Rat}(\Gamma)$.

For a disconnected tropical curve with parallel rays $\Gamma_{\operatorname{par}}$ with connected components $\Gamma_{\operatorname{par}, 1}, \ldots, \Gamma_{\operatorname{par}, n}$, where each $\Gamma_{\operatorname{par}, i}$ is considered as a tropical curve with parallel rays inheriting the equivalence relation from that of $\Gamma_{\operatorname{par}}$, one might expect that $\operatorname{Rat}(\Gamma_{\operatorname{par}}) \cong \bowtie_{i = 1}^n \operatorname{Rat}(\Gamma_{\operatorname{par}, i})$ holds.
However, this is not the case.
See the following example.

\begin{ex}
    \label{ex4-3-1}
\upshape{
Let $\Gamma_i$ be a copy of the interval $[-\infty, \infty]$.
For the disconnected tropical curve $\Gamma := \Gamma_1 \sqcup \Gamma_2$ and rays $L_1, L_2 \subset \Gamma$, we define
\begin{align*}
    L_1 \sim L_2 \quad \Longleftrightarrow \quad  -\infty \in L_1, L_2 \text{ or } \infty \in L_1, L_2,
\end{align*}
which is an equivalence relation on rays of $\Gamma$.
Put $\Gamma_{\operatorname{par}}$ the disconnected tropical curve with parallel rays defined by the pair $(\Gamma, \sim)$ and say $\Gamma_{\operatorname{par}, i}$ the tropical curve with parallel rays $(\Gamma_i, \sim_{\Gamma_i})$, where $\sim_{\Gamma_i}$ denotes the restriction of $\sim$ on rays of $\Gamma_i$ and it is an equivalence relation.
Note that $\operatorname{Rat}(\Gamma_{\operatorname{par}, i})$ is just $\operatorname{Rat}(\Gamma_i)$ in this case.
Then for the rational function $f \in \operatorname{Rat}(\Gamma_{\operatorname{par}, 1})$ defined by
\begin{align*}
f(t) := \begin{cases}
    0  &\text{if} \quad t < 0;\\
    t  &\text{if} \quad t \ge 0,
\end{cases}
\end{align*}
the pair $(f, 0)$ is in $\operatorname{Rat}(\Gamma_{\operatorname{par}, 1}) \bowtie \operatorname{Rat}(\Gamma_{\operatorname{par}, 2})$ but not in $\operatorname{Rat}(\Gamma_{\operatorname{par}})$.
}
\end{ex}

The reason why such a thing happens is that in the above setting, the pseudodirect product $\bowtie_{i = 1}^n \operatorname{Rat}(\Gamma_{\operatorname{par}, i})$ is isomorphic to the rational function semifield of the direct sum $\Gamma_{\operatorname{par}, 1} \sqcup \cdots \sqcup \Gamma_{\operatorname{par}, n}$ of $\Gamma_{\operatorname{par}, 1}, \ldots, \Gamma_{\operatorname{par}, n}$ defined as follows:

\begin{dfn}
    \label{dfn4-3-3}
\upshape{
Let $\Gamma_{\operatorname{par}, i}$ be the tropical curve with parallel rays defined by the pair $(\Gamma_i, \sim_i)$ of a tropical curve $\Gamma_i$ and an equivalence relation $\sim_i$ on rays of $\Gamma_i$.
The \textit{direct sum} $\Gamma_{\operatorname{par}, 1} \sqcup \cdots \sqcup \Gamma_{\operatorname{par}, n}$ of $\Gamma_{\operatorname{par}, 1}, \ldots, \Gamma_{\operatorname{par}, n}$ is the disconnected tropical curve with parallel rays defined by the pair $(\Gamma_1 \sqcup \cdots \sqcup \Gamma_n, \sim)$ such that for rays $L_1, L_2 \subset \Gamma_1 \sqcup \cdots \sqcup \Gamma_n$, 
\begin{align*}
L_1 \sim L_2 \quad \Longleftrightarrow \quad \exists i : L_1, L_2 \subset \Gamma_{\operatorname{par}, i} ,  L_1 \sim_i L_2.
\end{align*}
}
\end{dfn}

In the setting of Definition~\ref{dfn4-3-3}, the natural inclusion $\iota_i : \Gamma_{\operatorname{par}, i} \hookrightarrow \Gamma_{\operatorname{par}, 1} \sqcup \cdots \sqcup \Gamma_{\operatorname{par}, n}$ is a local isometry and it clearly induces a $\boldsymbol{T}$-algebra homomorphism $\operatorname{Rat}(\Gamma_{\operatorname{par}, 1} \sqcup \cdots \sqcup \Gamma_{\operatorname{par}, n}) \to \operatorname{Rat}(\Gamma_{\operatorname{par}, i}); f \mapsto f \circ \iota_i$, which is surjective.

\begin{prop}[cf.~{\cite[Lemmas~3.41 and 3.46]{JuAe=Nakajima}}]
    \label{prop4-3-1}
In the above setting, the map $\psi : \operatorname{Rat}(\Gamma_{\operatorname{par}, 1} \sqcup \cdots \sqcup \Gamma_{\operatorname{par}, n}) \to \bowtie_{i = 1}^n \operatorname{Rat}(\Gamma_{\operatorname{par}, i}); f \not= -\infty \mapsto (f \circ \iota_1, \ldots, f \circ \iota_n), -\infty \mapsto -\infty$ is a $\boldsymbol{T}$-algebra isomorphism.
\end{prop}

\begin{proof}
The existence and definition of $\psi$ follows from Proposition~\ref{prop2-8-1}.
For $(f_1, \ldots, f_n) \in (\bowtie_{i = 1}^n \operatorname{Rat}(\Gamma_{\operatorname{par}, i})) \setminus \{ -\infty \}$, let $f(x) := f_i(\iota_i^{-1}(x))$ when $x \in \iota_i(\Gamma_{\operatorname{par}, i})$ for $x \in \Gamma_{\operatorname{par}, 1} \sqcup \cdots \sqcup \Gamma_{\operatorname{par}, n}$.
Then, by the definition of $\sim$, the function $f$ is in $\operatorname{Rat}(\Gamma_{\operatorname{par}, 1} \sqcup \cdots \sqcup \Gamma_{\operatorname{par},n})$ and $\psi(f) = (f_1, \ldots, f_n)$.
Thus $\psi$ is surjective.
If $\psi(f) = \psi(g)$, then $f = g$ holds since $\operatorname{Im}(\iota_1) \cup \cdots \cup \operatorname{Im}(\iota_n) = \Gamma_{\operatorname{par}, 1} \sqcup \cdots \sqcup \Gamma_{\operatorname{par}, n}$ holds.
\end{proof}

\begin{rem}
    \label{rem4-3-1}
\upshape{
Proposition~\ref{prop4-3-1} justifies the definition of rational function semifields of disconnected tropical curves in Subsection~\ref{subsection2.8}.
In other words, for tropical curves $\Gamma_1, \ldots, \Gamma_n$, when we regard these tropical curves with parallel rays $\Gamma_{\operatorname{par}, 1}, \ldots, \Gamma_{\operatorname{par}, n}$ as in Remark~\ref{rem4-1-1}, the rational function semifield $\operatorname{Rat}(\Gamma_{\operatorname{par}, 1} \sqcup \cdots \sqcup \Gamma_{\operatorname{par}, n})$ of their direct sum $\Gamma_{\operatorname{par}, 1} \sqcup \cdots \sqcup \Gamma_{\operatorname{par}, n}$ coincides with that $\operatorname{Rat}(\Gamma_1 \sqcup \cdots \sqcup \Gamma_n) =\, \bowtie_{i = 1}^n \operatorname{Rat}(\Gamma_i)$ of the direct sum $\Gamma_1 \sqcup \cdots \sqcup \Gamma_n$.
}
\end{rem}

Again, for a disconnected tropical curve with parallel rays $\Gamma_{\operatorname{par}}$ with connected components $\Gamma_{\operatorname{par}, 1}, \ldots, \Gamma_{\operatorname{par}, n}$, where each $\Gamma_{\operatorname{par}, i}$ is considered as a tropical curve with parallel rays inheriting the equivalence relation from that of $\Gamma_{\operatorname{par}}$, the map $\psi : \operatorname{Rat}(\Gamma_{\operatorname{par}}) \to \bowtie_{i = 1}^n \operatorname{Rat}(\Gamma_{\operatorname{par}, i}); f \mapsto (f|_{\Gamma_{\operatorname{par}, 1}}, \ldots, f|_{\Gamma_{\operatorname{par}, n}})$ is an injective $\boldsymbol{T}$-algebra homomorphism.
In what follows, we sometimes regard $\operatorname{Rat}(\Gamma_{\operatorname{par}})$ as a subsemifield over $\boldsymbol{T}$ of $\bowtie_{i = 1}^n \operatorname{Rat}(\Gamma_{\operatorname{par}, i})$ via this $\psi$.

We can apply the proof of Proposition~\ref{prop4-1-1} to the following:

\begin{prop}
    \label{prop4-3-2}
For a disconnected tropical with parallel rays $\Gamma_{\operatorname{par}}$, its rational function semifield $\operatorname{Rat}(\Gamma_{\operatorname{par}})$ is finitely generated as a semifield over $\boldsymbol{T}$.
\end{prop}

\begin{proof}
In the proof of Proposition~\ref{prop4-1-1}, first, replace connected $\Gamma_{\operatorname{par}}$ with disconnected $\Gamma_{\operatorname{par}}$.
Then, since $\Gamma^{\prime}$ has no rays, when it has distinct connected components $\Gamma^{\prime}_1, \ldots, \Gamma^{\prime}_n$ such that $\Gamma^{\prime} = \bigsqcup_{i = 1}^n \Gamma^{\prime}_i$, the equality $\operatorname{Rat}(\Gamma^{\prime}) = \, \bowtie_{i = 1}^n \operatorname{Rat}(\Gamma^{\prime}_i)$ holds, and hence $\operatorname{Rat}(\Gamma^{\prime})$ is finitely generated as a semifield over $\boldsymbol{T}$ by \cite[Proposition~3.6]{JuAe3} and \cite[Lemma~3.34]{JuAe=Nakajima}.
Next, for distinct connected components $\Gamma_{\operatorname{par}, 1}, \ldots, \Gamma_{\operatorname{par}, a}$ of $\Gamma_{\operatorname{par}}$ such that $\Gamma_{\operatorname{par}} = \bigcup_{i = 1}^a \Gamma_{\operatorname{par}, i}$ as a set, let $e_i := (0, \ldots, 0, 1, 0, \ldots, 0) \in \operatorname{Rat}(\Gamma_{\operatorname{par}}) \subset \, \bowtie_{i = 1}^a \operatorname{Rat}(\Gamma_{\operatorname{par}, i})$ (the $1$ is in the $i$th component).
For each $L_{ij}$, choosing suitable numbers $k_1, \ldots, k_{b_{ij}} \in \{ 1, \ldots, a \}$, the inequality
\begin{align*}
    G_{ij} \odot h_{ij}^{\odot l_{ij}} \odot (e_{k_1} \odot \cdots \odot e_{k_{b_{ij}}})^{\odot (- l_{ij})} \le G_{i1}, \ldots, G_{im_i}
\end{align*}
holds on $\Gamma_{\operatorname{par}} \setminus L_{ij}$.
Then, replacing $G_{ij} \odot h_{ij}^{\odot l_{ij}}$ in equation~(\ref{prop4-1-1:eq1}) in the proof of Proposition~\ref{prop4-1-1} with the left hand side of this inequality, the same argument holds.
\end{proof}

We do not pursue the study of disconnected tropical curves with parallel rays and their rational function semifields in this paper except for subgraphs discussed in the next subsection.
However, we can do it using similar arguments as we have done so far.

\subsection{Subgraphs}
    \label{subsection4.4}

As in Subsection~\ref{subsection3.2}, we consider the notion of subgraphs of tropical curves with parallel rays.

Based on Subsection~\ref{subsection3.2}, the following notion is natural:

\begin{dfn}
    \label{dfn4-4-1}
\upshape{
Let $\Gamma_{\operatorname{par}}$ be a tropical curve with parallel rays.
A subset $\Gamma^{\prime}_{\operatorname{par}}$ of $\Gamma_{\operatorname{par}}$ is a \textit{subgraph} of $\Gamma_{\operatorname{par}}$ if it is a subgraph of the underlying tropical curve of $\Gamma_{\operatorname{par}}$ inheriting the equivalence relation on rays.
Explicitly, for rays of $\Gamma^{\prime}_{\operatorname{par}}$, these are parallel in $\Gamma^{\prime}_{\operatorname{par}}$ if and only if so are they in $\Gamma_{\operatorname{par}}$.
}
\end{dfn}

Corresponding to Proposition~\ref{prop3-2-1}, we have the following proposition:

\begin{prop}
    \label{prop4-4-1}
Let $\Gamma_{\operatorname{par}}$ be a tropical curve with parallel rays and $\Gamma^{\prime}_{\operatorname{par}}$ a connected subgraph of $\Gamma_{\operatorname{par}}$.
Then, regarding $\Gamma^{\prime}_{\operatorname{par}}$ itself as a tropical curve with parallel rays, the restriction $f|_{\Gamma^{\prime}_{\operatorname{par}}}$ on $\Gamma^{\prime}_{\operatorname{par}}$ for any $f \in \operatorname{Rat}(\Gamma_{\operatorname{par}})$ is in $\operatorname{Rat}(\Gamma^{\prime}_{\operatorname{par}})$ and the restriction map $\psi : \operatorname{Rat}(\Gamma_{\operatorname{par}}) \to \operatorname{Rat}(\Gamma^{\prime}_{\operatorname{par}}); f \mapsto f|_{\Gamma^{\prime}_{\operatorname{par}}}$ is a surjective $\boldsymbol{T}$-algebra homomorphism.
\end{prop}

\begin{proof}
The proof of Proposition~\ref{prop3-2-1} in the first three paragraphs works in this case by replacing $\Gamma$ (resp.~$\Gamma^{\prime}$) with $\Gamma_{\operatorname{par}}$ (resp.~$\Gamma^{\prime}_{\operatorname{par}}$).
When $\Gamma^{\prime}_{\operatorname{par}} \not= \Gamma_{\operatorname{par}}$ and $\Gamma^{\prime}_{\operatorname{par}} \cap \Gamma_{\operatorname{par}, \infty} \not= \varnothing$, for a ray $L^{\prime}$ of $\Gamma^{\prime}_{\operatorname{par}}$, let $L_1 = L^{\prime}, \ldots, L_k$ be distinct parallel rays of $\Gamma_{\operatorname{par}}$ such that $(1)$ if $i \not= j$, then $L_i \cap L_j$ is empty, and $(2)$ if $e$ is a ray of $\Gamma_{\operatorname{par}}$ parallel to $L^{\prime}$, then there exists a unique number $i$ satisfying $e \subset L_i$ or $e \supset L_i$.
Then $\operatorname{CF}\left( \overline{\Gamma_{\operatorname{par}} \setminus \bigcup_{i = 1}^k L_i}, \infty \right) \in \operatorname{Rat}(\Gamma_{\operatorname{par}})$ holds and its image by $\psi$ is in $\operatorname{Rat}(\Gamma^{\prime}_{\operatorname{par}})$, and by the proof of Proposition~\ref{prop4-1-1}, the $\boldsymbol{T}$-algebra homomorphism $\psi$ is surjective.
\end{proof}

The following proposition holds by Theorem~\ref{thm4-1-1} and Corollary~\ref{cor4-1-1} and \cite[Theorem~3.14]{JuAe4}:

\begin{prop}
    \label{prop4-4-2}
Let $\Gamma_{\operatorname{par}}$ and $\Gamma^{\prime}_{\operatorname{par}}$ be tropical curves with parallel rays and $\psi : \operatorname{Rat}(\Gamma_{\operatorname{par}}) \twoheadrightarrow \operatorname{Rat}(\Gamma^{\prime}_{\operatorname{par}})$ a surjective $\boldsymbol{T}$-algebra homomorphism.
Then the corresponding injective morphism $\varphi : \Gamma^{\prime}_{\operatorname{par}} \hookrightarrow \Gamma_{\operatorname{par}}$ satisfying $\psi(f) = f \circ \varphi$ for any $f \in \operatorname{Rat}(\Gamma_{\operatorname{par}})$ is a local isometry and its image $\operatorname{Im}(\varphi)$ is not contained in $\Gamma_{\operatorname{par}, \infty}$.
\end{prop}

\begin{prop}
    \label{prop4-4-3}
Let $\Gamma_{\operatorname{par}}$ be a tropical curve with parallel rays and $\Gamma^{\prime}_{\operatorname{par}}$ a subgraph of $\Gamma_{\operatorname{par}}$ with connected components $\Gamma^{\prime}_{\operatorname{par}, 1}, \ldots, \Gamma^{\prime}_{\operatorname{par}, l}$ such that $\Gamma^{\prime}_{\operatorname{par}, j} \not\subset \Gamma_{\operatorname{par}, \infty}$ for each $j = 1, \ldots, l$.
Then the $l$-tuple $\left(f|_{\Gamma^{\prime}_{\operatorname{par}, 1}}, \ldots, f|_{\Gamma^{\prime}_{\operatorname{par}, l}} \right)$ for any $f \in \operatorname{Rat}(\Gamma_{\operatorname{par}})$ is in $\bowtie_{i = 1}^l \operatorname{Rat}(\Gamma_{\operatorname{par}, i}^{\prime})$ and the restriction map $\psi : \operatorname{Rat}(\Gamma_{\operatorname{par}}) \to \bowtie_{i = 1}^l \operatorname{Rat}(\Gamma_{\operatorname{par}, i}^{\prime}); f \mapsto \left(f|_{\Gamma^{\prime}_{\operatorname{par}, 1}}, \ldots, f|_{\Gamma^{\prime}_{\operatorname{par}, l}} \right)$ is a $\boldsymbol{T}$-algebra homomorphism whose image is $\operatorname{Rat}(\Gamma^{\prime}_{\operatorname{par}})$, where $\Gamma^{\prime}_{\operatorname{par}}$ is regarded as a disconnected tropical curve with parallel rays.
Moreover, for each natural surjective $\boldsymbol{T}$-algebra homomorphism $\pi_j : \bowtie_{i = 1}^l \operatorname{Rat}(\Gamma_{\operatorname{par}, i}^{\prime}) \twoheadrightarrow \operatorname{Rat}(\Gamma_{\operatorname{par}, j}^{\prime}); (f^{\prime}_1, \ldots, f^{\prime}_l) \mapsto f^{\prime}_j$, the composition $\pi_j \circ \psi$ coincides with the restriction map $\operatorname{Rat}(\Gamma_{\operatorname{par}}) \twoheadrightarrow \operatorname{Rat}(\Gamma_{\operatorname{par}, j}^{\prime}); f \mapsto f|_{\Gamma_{\operatorname{par}, j}^{\prime}}$, which is surjective.
\end{prop}

\begin{proof}
The existence and definition of $\psi$ is clear by Propositions~\ref{prop4-4-1} and \ref{prop2-8-1}.
As $\Gamma^{\prime}_{\operatorname{par}}$ inherits the equivalence relation on rays of $\Gamma_{\operatorname{par}}$, the image $\psi(f)$ is in $\operatorname{Rat}(\Gamma^{\prime}_{\operatorname{par}})$.
Conversely, for any $(f^{\prime}_1, \ldots, f^{\prime}_l) \in \operatorname{Rat}(\Gamma^{\prime}_{\operatorname{par}}) \setminus \{ -\infty \} \subset \, \bowtie_{i = 1}^l \operatorname{Rat}(\Gamma^{\prime}_{\operatorname{par}, i})$, let $f$ be a rational function on $\Gamma_{\operatorname{par}}$ such that $(1)$ $f|_{\Gamma^{\prime}_{\operatorname{par}, i}} = f^{\prime}_i$ and $(2)$ $f$ is constant zero on the outside of the $\varepsilon$-neighborhood of $\Gamma^{\prime}_{\operatorname{par}}$ in $\Gamma_{\operatorname{par}}$ other than rays parallel to some rays of $\Gamma^{\prime}_{\operatorname{par}}$ for a sufficiently small positive number $\varepsilon$ and $(3)$ for such a ray $L$ of $\Gamma_{\operatorname{par}}$, $f$ has the same slope at infinity along $L$ to that at infinity along rays of $\Gamma^{\prime}_{\operatorname{par}}$ parallel to $L$.
Note that such $f$ clearly exists.
Then $\psi(f) = (f^{\prime}_1, \ldots, f^{\prime}_l)$ holds.
\end{proof}

\begin{prop}
    \label{prop4-4-4}
Let $\Gamma_{\operatorname{par}}$ (resp.~$\Gamma^{\prime}_{\operatorname{par}}$) be a tropical curve with parallel rays (resp.~a disconnected tropical curve with parallel rays) and $\psi : \operatorname{Rat}(\Gamma_{\operatorname{par}}) \twoheadrightarrow \operatorname{Rat}(\Gamma^{\prime}_{\operatorname{par}})$ a surjective $\boldsymbol{T}$-algebra homomorphism.
Then the corresponding injective map $\varphi : \Gamma^{\prime}_{\operatorname{par}} \hookrightarrow \Gamma_{\operatorname{par}}$ satisfying $\psi(f) = f \circ \varphi$ for any $f \in \operatorname{Rat}(\Gamma_{\operatorname{par}})$ is a local isometry and each connected component of its image $\operatorname{Im}(\varphi)$ is not contained in $\Gamma_{\operatorname{par}, \infty}$.
\end{prop}

\begin{proof}
It is clear by Propositions~\ref{prop4-3-2}, \ref{prop4-4-2} and \cite[Theorem~3.14]{JuAe4}.
\end{proof}

\subsection{Weights}
    \label{subsection4.5}

In this subsection, we define weights on edges of tropical curves with parallel rays.
This notion is described by degrees on edges of morphisms:

\begin{dfn}
    \label{dfn4-5-1}
\upshape{
Let $\Gamma_{\operatorname{par}}$ and $\Gamma^{\prime}_{\operatorname{par}}$ be tropical curves with parallel rays, respectively, and $\varphi : \Gamma_{\operatorname{par}} \to \Gamma^{\prime}_{\operatorname{par}}$ a morphism.
This $\varphi$ is a \textit{weight} of $\Gamma^{\prime}_{\operatorname{par}}$ if it is bijective and for two rays of $\Gamma_{\operatorname{par}}$, these are parallel if and only if so are these images by $\varphi$.
Then we also call the injective $\boldsymbol{T}$-algebra homomorphism $\psi : \operatorname{Rat}(\Gamma^{\prime}_{\operatorname{par}}) \hookrightarrow \operatorname{Rat}(\Gamma_{\operatorname{par}}); f \mapsto f \circ \varphi$ a \textit{weight} of $\Gamma^{\prime}_{\operatorname{par}}$.
When this $\varphi$ (or $\psi$) is a weight of $\Gamma^{\prime}_{\operatorname{par}}$, for an edge $e^{\prime} \subset \Gamma^{\prime}_{\operatorname{par}}$ such that $\varphi$ has the degree $\operatorname{deg}_{\varphi^{-1}(e^{\prime})}(\varphi)$ on $\varphi^{-1}(e^{\prime})$, we call $\operatorname{deg}_{\varphi^{-1}(e^{\prime})}(\varphi)$ the \textit{weight} on $e^{\prime}$.
}
\end{dfn}

\begin{ex}
    \label{ex4-5-1}
\upshape{
Let $\Gamma := [-\infty, \infty]$.
Both of the maps $\varphi_1 : \Gamma \to \Gamma; x \mapsto x$ and $\varphi_2 :\Gamma \to \Gamma; x \mapsto 2x$ are weights of $\Gamma$.
Then every edge of $\Gamma$ has weight one, two, respectively.
For the rational functions $f_1(x) = x$ and $f_2(x) = 2x$ on $\Gamma$, the former $f_1$ generates $\operatorname{Rat}(\Gamma)$, and $f_1 = f_1 \circ \varphi_1$ and $f_2 = f_1 \circ \varphi_2$ hold, and so choosing $f_1$ and $f_2$ in $\operatorname{Rat}(\Gamma)$ corresponds to giving these weights $\varphi_1$ and $\varphi_2$, respectively.
}
\end{ex}

\begin{prop}
    \label{prop4-5-1}
Assume $\psi$ in Definition~\ref{dfn4-5-1} is a weight of $\Gamma^{\prime}_{\operatorname{par}}$.
For any finite generating set $\{f^{\prime}_1, \ldots, f^{\prime}_n\} \subset \operatorname{Rat}(\Gamma^{\prime}_{\operatorname{par}}) \setminus \{-\infty \}$ of $\operatorname{Rat}(\Gamma^{\prime}_{\operatorname{par}})$ as a semifield over $\boldsymbol{T}$ and an edge $e^{\prime}$ of $\Gamma^{\prime}_{\operatorname{par}}$ where all $f^{\prime}_i$s have constant (integer) slopes, the greatest common divisor $m_f$ of the slopes on $\varphi^{-1}(e^{\prime})$ of $\psi(f^{\prime}_1), \ldots, \psi(f^{\prime}_n)$ coincides with the weight on $e^{\prime}$.
\end{prop}

\begin{proof}
Let $x^{\prime}$ be a point of $e^{\prime}$ other than its endpoint(s) and $\varepsilon$ a sufficiently small positive number.
Then the degree $\operatorname{deg}_{\varphi^{-1}(e^{\prime})}(\varphi)$ coincides with the absolute value of the slopes of $\psi(\operatorname{CF}(\overline{U}, \varepsilon))$ on $\varphi^{-1}(\overline{U})$, where $U$ denotes the $\varepsilon$-neighborhood of $x^{\prime}$ and $\overline{U}$ is the closure of $U$ in $\Gamma^{\prime}_{\operatorname{par}}$.
Since $\psi(f^{\prime}_1), \ldots, \psi(f^{\prime}_n)$ generates $\psi(\operatorname{Rat}(\Gamma^{\prime}_{\operatorname{par}}))$ as a semifield over $\boldsymbol{T}$, the value $m_f$ is less than or equal to $\operatorname{deg}_{\varphi^{-1}(e^{\prime})}(\varphi)$.
Conversely, as $\psi = \varphi^{\ast}$, the slopes of all elements of $\operatorname{Rat}(\Gamma^{\prime}_{\operatorname{par}}) \setminus \{ - \infty \}$ on $\varphi^{-1}(e^{\prime})$ are divided by $\operatorname{deg}_{\varphi^{-1}(e^{\prime})}(\varphi)$, and so $m_f \ge \operatorname{deg}_{\varphi^{-1}(e^{\prime})}(\varphi)$.
\end{proof}

Using the notion of localization in the next subsection, we can translate this proposition as follows:

\begin{cor}
    \label{cor4-5-1}
As in the setting of Proposition~\ref{prop4-5-1}, for any point $x$ of $\varphi^{-1}(e^{\prime})$ other than its endpoint(s), let $\pi_x : \operatorname{Rat}(\Gamma_{\operatorname{par}}) \twoheadrightarrow \operatorname{Rat}(\Gamma_{\operatorname{par}})_x \cong ((\boldsymbol{R} \times \boldsymbol{Z}^2) \cup \{ -\infty \}, \boxplus, \boxdot) = R_2$ be the localization of $\operatorname{Rat}(\Gamma^{\prime}_{\operatorname{par}})$ at $x$.
Then for any finite generating set $\{ f_1^{\prime}, \ldots, f_n^{\prime} \} \subset \operatorname{Rat}(\Gamma^{\prime}_{\operatorname{par}}) \setminus \{ -\infty \}$ of $\operatorname{Rat}(\Gamma^{\prime}_{\operatorname{par}})$ as a semifield over $\boldsymbol{T}$, the greatest common divisor of any one of $\boldsymbol{Z}^2$-components of $\pi_x(\psi(f^{\prime}_1)), \ldots, \pi_x(\psi(f^{\prime}_n))$ is equal to the weight on $e^{\prime}$.
\end{cor}

We give a weighted version of Definition~\ref{dfn4-2-2}:

\begin{dfn}
    \label{dfn4-5-2}
\upshape{
Let $\Gamma_{\operatorname{par}}$ and $\Gamma^{\prime}_{\operatorname{par}}$ be tropical curves with parallel rays, respectively, and $\psi : \operatorname{Rat}(\Gamma^{\prime}_{\operatorname{par}}) \hookrightarrow \operatorname{Rat}(\Gamma_{\operatorname{par}})$ a weight of $\Gamma^{\prime}_{\operatorname{par}}$.
For $f^{\prime}_1, \ldots, f^{\prime}_n \in \operatorname{Rat}(\Gamma^{\prime}_{\operatorname{par}}) \setminus \{ -\infty \}$, the image of the map $\theta : \Gamma_{\operatorname{par}} \setminus \Gamma_{\operatorname{par}, \infty} \to \boldsymbol{R}^n; x \mapsto (\psi(f^{\prime}_1)(x), \ldots, \psi(f^{\prime}_n)(x))$ is called the \textit{weighted realization} of $\Gamma^{\prime}_{\operatorname{par}}$ if $f^{\prime}_1, \ldots, f^{\prime}_n$ generate $\operatorname{Rat}(\Gamma^{\prime}_{\operatorname{par}})$ as a semifield over $\boldsymbol{T}$.
}
\end{dfn}

\subsection{Localization}
    \label{subsection4.6}

In this subsection, we consider the notion of localization motivated by the point of view of geometry.
To do so, we first give a semifield over $\boldsymbol{T}$ consisting of tropical Laurent monomials and $-\infty$.

Let $M_n$ be the set of all tropical Laurent monomials in $\boldsymbol{T}[\boldsymbol{X}^{\pm}]_n$ together with $-\infty$.
This becomes a semifield over $\boldsymbol{T}$ by the natural operations $\boxplus$ and $\boxdot$ defined as follows:
for $a, b \in \boldsymbol{R}$ and $\boldsymbol{i}, \boldsymbol{j} \in \boldsymbol{Z}^n$,
\begin{align*}
a \odot \boldsymbol{X}^{\odot \boldsymbol{i}} \boxplus b \odot \boldsymbol{X}^{\odot \boldsymbol{j}} &:= \begin{cases}
a \odot \boldsymbol{X}^{\odot \boldsymbol{i}} \quad \text{if }a > b,\\
a \odot \boldsymbol{X}^{\odot (\boldsymbol{i} \oplus \boldsymbol{j})} \quad \text{if }a = b,\\
b \odot \boldsymbol{X}^{\odot \boldsymbol{j}} \quad \text{if }a < b,
\end{cases}\\
a \odot \boldsymbol{X}^{\odot \boldsymbol{i}} \boxplus (-\infty) &:= a \odot \boldsymbol{X}^{\odot \boldsymbol{i}},\\
- \infty \boxplus a \odot \boldsymbol{X}^{\odot \boldsymbol{i}} &:= a \odot \boldsymbol{X}^{\odot \boldsymbol{i}},\\
- \infty \boxplus (-\infty) &:= -\infty,\\
a \odot \boldsymbol{X}^{\odot \boldsymbol{i}} \boxdot b \odot \boldsymbol{X}^{\odot \boldsymbol{j}} &:= (a \odot b) \odot \boldsymbol{X}^{\odot (\boldsymbol{i} \odot \boldsymbol{j})},\\
a \odot \boldsymbol{X}^{\odot \boldsymbol{i}} \boxdot (- \infty) &:= -\infty,\\
-\infty \boxdot a \odot \boldsymbol{X}^{\odot \boldsymbol{i}} &:= -\infty, \text{ and}\\
-\infty \boxdot (-\infty) &:= -\infty.
\end{align*}

Let $M_n := (M_n, \boxplus, \boxdot)$.

\begin{prop}
    \label{prop4-6-1}
The above $M_n$ is a semifield over $\boldsymbol{T}$.
Moreover, the map
\begin{align*}
\psi :\, \boldsymbol{T}[\boldsymbol{X}^{\pm}]_n \to M_n;
\bigoplus_{\boldsymbol{i} \in \boldsymbol{Z}^n} a_{\boldsymbol{i}} \odot \boldsymbol{X}^{\odot \boldsymbol{i}} \mapsto \boxplus_{\boldsymbol{i} \in \boldsymbol{Z}^n} a_{\boldsymbol{i}} \odot \boldsymbol{X}^{\odot \boldsymbol{i}},
-\infty \mapsto -\infty
\end{align*}
is a surjective $\boldsymbol{T}$-algebra homomorphism.
\end{prop}

\begin{proof}
The first assertion is straightforward.
The identity with respect to $\boxplus$ is the real number zero, and that with respect to $\boxdot$ is $-\infty$.
For $a \odot \boldsymbol{X}^{\odot \boldsymbol{i}} \in M_n \setminus \{ -\infty \}$, its multiplicative inverse is $a^{\odot (-1)} \odot \boldsymbol{X}^{\odot (-\boldsymbol{i})}$.

By definition, $\psi(0) = 0$ and $\psi(-\infty) = -\infty$ hold.
Let $f = \bigoplus_{\boldsymbol{i} \in \boldsymbol{Z}^n} a_{\boldsymbol{i}} \odot \boldsymbol{X}^{\odot \boldsymbol{i}}, g = \bigoplus_{\boldsymbol{i} \in \boldsymbol{Z}^n} b_{\boldsymbol{i}} \odot \boldsymbol{X}^{\odot \boldsymbol{i}} \in \boldsymbol{T}[\boldsymbol{X}^{\pm}]_n \setminus \{ -\infty \}$.
Here, all $a_{\boldsymbol{i}}$ and $b_{\boldsymbol{j}}$ except for a finite number of $\boldsymbol{i}, \boldsymbol{j} \in \boldsymbol{Z}^n$ are $-\infty$.
Since $a_{\boldsymbol{i}} \odot \boldsymbol{X}^{\odot \boldsymbol{i}} \oplus b_{\boldsymbol{i}} \odot \boldsymbol{X}^{\odot \boldsymbol{i}} = a_{\boldsymbol{i}} \odot \boldsymbol{X}^{\odot \boldsymbol{i}} \boxplus b_{\boldsymbol{i}} \odot \boldsymbol{X}^{\odot \boldsymbol{i}}$ holds by definition, 
\begin{align*}
\psi(f \oplus g) &= \psi \left( \bigoplus_{\boldsymbol{i} \in \boldsymbol{Z}^n} (a_{\boldsymbol{i}} \oplus b_{\boldsymbol{i}}) \odot \boldsymbol{X}^{\odot \boldsymbol{i}} \right)\\
&
= \boxplus_{\boldsymbol{i} \in \boldsymbol{Z}^n} (a_{\boldsymbol{i}} \oplus b_{\boldsymbol{i}}) \odot \boldsymbol{X}^{\odot \boldsymbol{i}}\\
&= \boxplus_{\boldsymbol{i} \in \boldsymbol{Z}^n} (a_{\boldsymbol{i}} \odot \boldsymbol{X}^{\odot \boldsymbol{i}} \oplus b_{\boldsymbol{i}} \odot \boldsymbol{X}^{\odot \boldsymbol{i}})\\
&= \boxplus_{\boldsymbol{i} \in \boldsymbol{Z}^n} (a_{\boldsymbol{i}} \odot \boldsymbol{X}^{\odot \boldsymbol{i}} \boxplus b_{\boldsymbol{i}} \odot \boldsymbol{X}^{\odot \boldsymbol{i}})\\
&= \left( \boxplus_{\boldsymbol{i} \in \boldsymbol{Z}^n} a_{\boldsymbol{i}} \odot \boldsymbol{X}^{\odot \boldsymbol{i}} \right) \boxplus \left( \boxplus_{\boldsymbol{i} \in \boldsymbol{Z}^n} b_{\boldsymbol{i}} \odot \boldsymbol{X}^{\odot \boldsymbol{i}} \right)\\
&= \psi(f) \boxplus \psi(g)
\end{align*}
hold.
Also for $c_{\boldsymbol{k}} := \bigoplus_{\boldsymbol{i}, \boldsymbol{j} \in \boldsymbol{Z}^n :\, \boldsymbol{i} \odot \boldsymbol{j} = \boldsymbol{k}} a_{\boldsymbol{i}} \odot b_{\boldsymbol{j}}$,
\begin{align*}
\psi(f \odot g) &= \psi \left( \bigoplus_{\boldsymbol{k} \in \boldsymbol{Z}^n} c_{\boldsymbol{k}} \odot \boldsymbol{X}^{\odot \boldsymbol{k}} \right)\\
&= \boxplus_{\boldsymbol{k} \in \boldsymbol{Z}^n} c_{\boldsymbol{k}} \odot \boldsymbol{X}^{\odot \boldsymbol{k}}\\
\psi(f) \boxdot \psi(g) &= \left( \boxplus_{\boldsymbol{i} \in \boldsymbol{Z}^n} a_{\boldsymbol{i}} \odot \boldsymbol{X}^{\odot \boldsymbol{i}} \right) \boxdot \left( \boxplus_{\boldsymbol{j} \in \boldsymbol{Z}^n} b_{\boldsymbol{j}} \odot \boldsymbol{X}^{\odot \boldsymbol{j}} \right)\\
&= \boxplus_{\boldsymbol{k} \in \boldsymbol{Z}^n} c_{\boldsymbol{k}} \odot \boldsymbol{X}^{\odot \boldsymbol{k}},
\end{align*}
and hence $\psi(f \odot g) = \psi(f) \boxdot \psi(g)$ holds.
Clearly $m \in M_n \subset \boldsymbol{T}[\boldsymbol{X}^{\pm}]_n$ goes to $m$ by $\psi$.
In conclusion, $\psi$ is a surjective $\boldsymbol{T}$-algebra homomorphism.
\end{proof}

\begin{rem}
    \label{rem4-6-1}
\upshape{
The semifield $M_n$ is one of $\boldsymbol{T}$-extension defined in \cite[Definition~2.9]{Ito}.
More precisely, $M_n$ is the $\boldsymbol{T}$-extension of the subsemifield consisting of all Laurent monomials in $\boldsymbol{B}[\boldsymbol{X}^{\pm}]_n$ possessing $-\infty$.
}
\end{rem}

Note that by taking out coefficients and indices of tropical Laurent monomials, the triple $R_n := ((\boldsymbol{R} \times \boldsymbol{Z}^n) \cup \{ -\infty \}, \boxplus, \boxdot)$ defined as follows is a semifield over $\boldsymbol{T}$ isomorphic to $M_n$ as a $\boldsymbol{T}$-algebra:
for $(a, (i_1, \ldots, i_n)), (b, (j_1, \ldots, j_n)) \in \boldsymbol{R} \times \boldsymbol{Z}^n$,
\begin{align*}
(a, (i_1, \ldots, i_n)) \boxplus (b, (j_1, \ldots, j_n)) &:= \begin{cases}
(a, (i_1, \ldots, i_n)) \quad \text{if }a > b,\\
(a, (\operatorname{max}\{ i_1, j_1 \}, \ldots, \operatorname{max}\{ i_n, j_n \})) \quad \text{if }a = b,\\
(b, (j_1, \ldots, j_n)) \quad \text{if }a < b,
\end{cases}\\
(a, (i_1, \ldots, i_n)) \boxplus (-\infty) &:= (a, (i_1, \ldots, i_n)),\\
-\infty \boxplus (a, (i_1, \ldots, i_n)) &:= (a, (i_1, \ldots, i_n)),\\
-\infty \boxplus (-\infty) &:= -\infty,\\
(a, (i_1, \ldots, i_n)) \boxdot (b, (j_1, \ldots, j_n)) &:= (a + b, (i_1 + j_1, \ldots, i_n + j_n)),\\
(a, (i_1, \ldots, i_n)) \boxdot (-\infty) &:= -\infty,\\
-\infty \boxdot (a, (i_1, \ldots, i_n)) &:= -\infty, \text{ and}\\
-\infty \boxdot (-\infty) &:= -\infty.
\end{align*}

Let us put $R_0 := \boldsymbol{T}$.
For $n \ge 1$ and $k = 1, \ldots, n$, by forgetting the $k$th component of $\boldsymbol{Z}^n$, we have the surjective map
\begin{align*}
R_n \twoheadrightarrow R_{n - 1};
(a, (i_1, \ldots, i_n)) \mapsto (a, (i_1, \ldots, i_{k - 1}, i_{k + 1}, \ldots, i_n)), -\infty \mapsto -\infty,
\end{align*}
which is a $\boldsymbol{T}$-algebra homomorphism.
In particular, if $n = 1$ and $k = 1$, then its target is just $R_0 = \boldsymbol{T}$.
Moreover, when $n \ge 2$, for $j < k$, forgetting the $j$th component of $\boldsymbol{Z}^{n - 1}$ of $R_{n - 1}$ after doing the $k$th component of $\boldsymbol{Z}^n$ of $R_n$ coincides with forgetting the $(k - 1)$th component of $\boldsymbol{Z}^{n - 1}$ of $R_{n - 1}$ after doing the $j$th component of $\boldsymbol{Z}^n$ of $R_n$, i.e., these give the same surjective $\boldsymbol{T}$-algebra homomorphism $R_n \twoheadrightarrow R_{n - 2}$.

Next we give our localization for rational function semifields of tropical curves with parallel rays:

\begin{prop}
    \label{prop4-6-2}
Let $\Gamma_{\operatorname{par}}$ be a tropical curve with parallel rays and $\Gamma^{\prime}_{\operatorname{par}} \subset \Gamma_{\operatorname{par}}$ a subgraph of $\Gamma_{\operatorname{par}}$.
For $f, g \in \operatorname{Rat}(\Gamma_{\operatorname{par}})$, let $\sim_{\Gamma^{\prime}_{\operatorname{par}}}$ be
\begin{align*}
&~ f \sim_{\Gamma^{\prime}_{\operatorname{par}}} g\\
\Longleftrightarrow &~ \text{there exists a neighborhood }U \text{ of } \Gamma^{\prime}_{\operatorname{par}} \text{ such that } f|_U = g|_U.
\end{align*}
Then $\{ (f, g) \in \operatorname{Rat}(\Gamma_{\operatorname{par}})^2 \,|\,  f \sim_{\Gamma^{\prime}_{\operatorname{par}}} g \}$ is a congruence on $\operatorname{Rat}(\Gamma_{\operatorname{par}})$.
\end{prop}

\begin{proof}
This is straightforward.
\end{proof}

Let $\Gamma_{\operatorname{par}}$ be a tropical curve with parallel rays.
For $x \in \Gamma_{\operatorname{par}} \setminus \Gamma_{\operatorname{par}, \infty}$, let $\operatorname{Rat}(\Gamma_{\operatorname{par}})_x$ denote the quotient semifield $\operatorname{Rat}(\Gamma_{\operatorname{par}}) / \sim_{ \{x \}}$ over $\boldsymbol{T}$.
Let $n$ be the valence of $x$ and $e_1, \ldots, e_n$ distinct outgoing directions from $x$.

\begin{prop}
    \label{prop4-6-3}
In the above setting, the map
\begin{align*}
\psi :\, &\operatorname{Rat}(\Gamma_{\operatorname{par}})_x \to R_n;\\ &[f] \not= -\infty \mapsto (f(x), (\operatorname{sl}_f(e_1), \ldots, \operatorname{sl}_f(e_n))), -\infty \mapsto -\infty
\end{align*}
is a well-defined $\boldsymbol{T}$-algebra isomorphism, where $[f]$ denotes the equivalence class of $f \in \operatorname{Rat}(\Gamma_{\operatorname{par}})$ under $\sim_{ \{ x \} }$ and $\operatorname{sl}_f(e_i)$ the outgoing slope of $f$ at $x$ in the direction $e_i$.
\end{prop}

\begin{proof}
When $n = 0$, it is clear that $\operatorname{Rat}(\Gamma_{\operatorname{par}})_x = \boldsymbol{T} = R_0$ and $\psi$ is just the identity map of $\boldsymbol{T}$.
Let $n \ge 1$.
For $f \in \operatorname{Rat}(\Gamma_{\operatorname{par}}) \setminus \{ -\infty \}$, the value $f(x)$ is a real number since $x$ is a finite point.
Also, if $f_1, f_2 \in [f]$, then $f_1(x) = f_2(x)$ and $\operatorname{sl}_{f_1}(e_i) = \operatorname{sl}_{f_2}(e_i)$ hold by the definition of $\sim_{\{x \}}$.
Therefore $(f(x), (\operatorname{sl}_f(e_1), \ldots, \operatorname{sl}_f(e_n))) \in \boldsymbol{R} \times \boldsymbol{Z}^n$ and this is independent of the choice of representative.
It is easy to check that $\psi$ is an injective $\boldsymbol{T}$-algebra homomorphism.
For a sufficiently small positive number $\varepsilon$ and $(a, (i_1, \ldots, i_n)) \in \boldsymbol{R} \times \boldsymbol{Z}^n$, let $f$ be a rational function on $\Gamma_{\operatorname{par}}$ such that $(1)$ $f(x) = a$, $(2)$ $\operatorname{sl}_f(e_j) = i_j$, and $(3)$ $f$ is constant (the real number) zero on $\Gamma_{\operatorname{par}}$ outside of the $\varepsilon$-neighborhood of $x$.
Clearly at least one of such $f$ exists, and for such $f$, the image $\psi([f])$ is $(a, (i_1, \ldots, i_n))$.
Thus $\psi$ is surjective.
\end{proof}

The following corollary is clear:

\begin{cor}
    \label{cor4-6-1}
In addition to the setting in Proposition~\ref{prop4-6-3}, let $w : \operatorname{Rat}(\Gamma_{\operatorname{par}}) \hookrightarrow \operatorname{Rat}(\Gamma^{\prime}_{\operatorname{par}})$ be a weight of $\Gamma_{\operatorname{par}}$ and $\varphi : \Gamma^{\prime}_{\operatorname{par}} \to \Gamma_{\operatorname{par}}$ the corresponding morphism.
Let $x^{\prime} := \varphi^{-1}(x), e^{\prime}_i := \varphi^{-1}(e_i)$, $w_i$ the weight on $e_i$ and $\pi^{\prime} : \operatorname{Rat}(\Gamma^{\prime}_{\operatorname{par}}) \twoheadrightarrow \operatorname{Rat}(\Gamma^{\prime}_{\operatorname{par}})_{x^{\prime}}$ the natural surjective $\boldsymbol{T}$-algebra homomorphism and $\psi^{\prime} : \operatorname{Rat}(\Gamma^{\prime}_{\operatorname{par}})_{x^{\prime}} \to R_n$ the $\boldsymbol{T}$-algebra isomorphism defined by the same way for $\psi$ in Proposition~\ref{prop4-6-3}.
Then $\psi^{\prime}(\pi^{\prime}(w(\operatorname{Rat}(\Gamma_{\operatorname{par}}))))$ is $(\boldsymbol{R} \times w_1\boldsymbol{Z} \times \cdots \times w_n\boldsymbol{Z}) \cup \{ -\infty \}$.
\end{cor}

We regard $\operatorname{Rat}(\Gamma_{\operatorname{par}})_x$ as a localization of $\operatorname{Rat}(\Gamma_{\operatorname{par}})$ at the point $x$ when $\Gamma_{\operatorname{par}}$ is unweighted as in Proposition~\ref{prop4-6-3}, and $\pi^{\prime}(w(\operatorname{Rat}(\Gamma_{\operatorname{par}})))$ when $\Gamma_{\operatorname{par}}$ is weighted as in Corollary~\ref{cor4-6-1}.
Moreover, in both cases, if $\Psi : \operatorname{Rat}(\Gamma_{\operatorname{par}}) \twoheadrightarrow \boldsymbol{T}$ and $\Psi^{\prime} : \operatorname{Rat}(\Gamma^{\prime}_{\operatorname{par}}) \twoheadrightarrow \boldsymbol{T}$ are the surjective $\boldsymbol{T}$-algebra homomorphisms corresponding to $x \in \Gamma_{\operatorname{par}} \setminus \Gamma_{\operatorname{par}, \infty}$ and $x^{\prime} \in \Gamma^{\prime}_{\operatorname{par}} \setminus \Gamma^{\prime}_{\operatorname{par}, \infty}$, respectively, then these are passing through $\pi$ and $\pi^{\prime}$, respectively, where $\pi$ denotes the natural surjective $\boldsymbol{T}$-algebra homomorphism $\operatorname{Rat}(\Gamma_{\operatorname{par}}) \twoheadrightarrow \operatorname{Rat}(\Gamma_{\operatorname{par}})_x$.
In other words, if $p : R_n \twoheadrightarrow R_0 = \boldsymbol{T}$ denotes the map forgetting all components of $\boldsymbol{Z}^n$, i.e., the projection, then $\Psi = p \circ \psi \circ \pi$ and $\Psi^{\prime} = p \circ \psi^{\prime} \circ \pi^{\prime}$ hold.

From now on, we characterize the notion of valence of a point of tropical curves (with parallel rays) with an algebraic language:

\begin{prop}
    \label{prop4-6-4}
Let $\Gamma_{\operatorname{par}}$ be a tropical curve with parallel rays and $\psi : \operatorname{Rat}(\Gamma_{\operatorname{par}}) \twoheadrightarrow \boldsymbol{T}$ be a surjective $\boldsymbol{T}$-algebra homomorphism.
If there exists a surjective $\boldsymbol{T}$-algebra homomorphism $\phi : \operatorname{Rat}(\Gamma_{\operatorname{par}}) \twoheadrightarrow R_n$ such that $\psi = p \circ \phi$ with the projection $p : R_n \twoheadrightarrow \boldsymbol{T}$ for some $n$, then there exists a surjective $\boldsymbol{T}$-algebra homomorphism $\theta : \operatorname{Rat}(\Gamma_{\operatorname{par}})_x \twoheadrightarrow R_n$ such that $\phi = \theta \circ \pi$ holds, where $x \in \Gamma_{\operatorname{par}} \setminus \Gamma_{\operatorname{par}, \infty}$ is the point corresponding to $\psi$ and $\pi$ denotes the natural surjective $\boldsymbol{T}$-algebra homomorphism $\operatorname{Rat}(\Gamma_{\operatorname{par}}) \twoheadrightarrow \operatorname{Rat}(\Gamma_{\operatorname{par}})_x$.
Moreover, the valence of $x$ is greater than or equal to $n$.
\end{prop}

\begin{proof}
If $n = 0$, then the assertions are clear.
Let $n \ge 1$.
Assume that there exists $(f, g) \in \operatorname{Ker}(\pi)$ which is not in $\operatorname{Ker}(\phi)$.
By definition, if $f = -\infty$, then so is $g$, and this contradicts $(f, g) \not\in \operatorname{Ker}(\phi)$.
Thus $f, g \not= -\infty$.
Let $h := f \odot g^{\odot (-1)}$.
If $\phi(h \oplus 0) = \phi(h^{\odot (-1)} \oplus 0) = (0, \boldsymbol{0})$, then, as $h = (h \oplus 0) \odot (h^{\odot (-1)} \oplus 0)^{\odot (-1)}$ and $(0, \boldsymbol{0}) \boxdot (0, \boldsymbol{0})^{\boxdot (-1)} = (0, \boldsymbol{0})$ hold, $\phi(h)$ must coincide with $(0, \boldsymbol{0})$.
This means that $\phi(f) = \phi(h \odot g) = \phi(h) \boxdot \phi(g) = (0, \boldsymbol{0}) \boxdot \phi(g) = \phi(g)$.
Therefore at least one of $\phi(h \oplus 0)$ and $\phi(h^{\odot (-1)} \oplus 0)$ is not $(0, \boldsymbol{0})$.
Without loss of generality, it is assumed that the former is so.
Under this assumption, since $\phi(h \oplus 0) \boxplus (0, \boldsymbol{0}) = \phi(h \oplus 0) \boxplus \phi(0) = \phi((h \oplus 0) \oplus 0) = \phi(h \oplus 0)$, all nonzero components of $\phi(h \oplus 0)$ are positive.
Because $(f, g) \in \operatorname{Ker}(\pi)$, there exists a positive number $\varepsilon$ such that $f$ and $g$ coincide on the $\varepsilon$-neighborhood $U$ of $x$, and so $h$ is constant (the real number) zero on $U$.
Therefore, as $h(x) = 0 < \varepsilon = \varepsilon \odot \operatorname{CF}(\{ x \}, \varepsilon)(x)$ and $(h \oplus 0) \oplus \varepsilon \odot \operatorname{CF}(\{ x \}, \varepsilon) = (h \oplus 0) \odot (\varepsilon \odot \operatorname{CF}(\{ x \}, \varepsilon))$ hold,
\begin{align*}
&~ \phi(\varepsilon \odot \operatorname{CF}(\{ x \} ,\varepsilon)) = \phi(h \oplus 0) \boxplus \phi(\varepsilon \odot \operatorname{CF}(\{ x \}, \varepsilon))\\
=&~ \phi((h \oplus 0) \oplus \varepsilon \odot \operatorname{CF}(\{ x \}, \varepsilon)) = \phi((h \oplus 0) \odot (\varepsilon \odot \operatorname{CF}(\{ x \}, \varepsilon)))\\
=&~ \phi(h \oplus 0) \boxdot \phi(\varepsilon \odot \operatorname{CF}(\{ x\}, \varepsilon)),
\end{align*}
which is a contradiction.
Hence, there does not exist such $(f, g)$, in other words, $\operatorname{Ker}(\pi) \subset \operatorname{Ker}(\phi)$ holds.
By the universality of $\operatorname{Rat}(\Gamma_{\operatorname{par}})_x$, the desired $\boldsymbol{T}$-algebra homomorphism $\theta$ is induced.
The last statement follows from Proposition~\ref{prop-appendix3}.
\end{proof}

Therefore, in the setting of Proposition~\ref{prop4-6-4}, we know that the valence of $x$ is the maximum number $n$ of $R_n$ such that there exists a surjective $\boldsymbol{T}$-algebra homomorphism $\phi : \operatorname{Rat}(\Gamma_{\operatorname{par}}) \twoheadrightarrow R_n$ satisfying $\psi = p \circ \phi$.

\subsection{Finitely generated \texorpdfstring{$\boldsymbol{T}$}{T}-modules}
    \label{subsection4.7}

In this subsection, we focus on finitely generated $\boldsymbol{T}$-modules in rational function semifields of tropical curves with parallel rays and define their degrees.

\begin{prop}
    \label{prop4-7-1}
Let $\Gamma_{\operatorname{par}}$ be a tropical curve with parallel rays.
For $f_1, \ldots, f_n \in \operatorname{Rat}(\Gamma_{\operatorname{par}}) \setminus \{ -\infty \}$, let $R$ denote the $\boldsymbol{T}$-module generated by $f_1, \ldots, f_n$.
For any point $x \in \Gamma_{\operatorname{par}}$ and any generators $g_1, \ldots, g_l \in R \setminus \{ -\infty \}$ of $R$ as a $\boldsymbol{T}$-module, the equality
\begin{align*}
&~ \operatorname{min}\{ \operatorname{div}(f_1)(x), \ldots, \operatorname{div}(f_n)(x)\}\\
=&~ \operatorname{min}\{ \operatorname{div}(g_1)(x), \ldots, \operatorname{div}(g_l)(x)\}
\end{align*}
holds.
\end{prop}

\begin{proof}
If $\Gamma_{\operatorname{par}}$ consists of only one point $x$, then clearly $\operatorname{div}(f_i)(x) = \operatorname{div}(g_j)(x) = 0$ for any $i = 1, \ldots, n, j = 1, \ldots, l$.
Thus the assertion holds.

Assume that $\Gamma_{\operatorname{par}}$ does not consist of only one point.
As $g_j \in R \setminus \{ -\infty \}$ and $f_1, \ldots, f_n$ generate $R$ as a $\boldsymbol{T}$-module, there are $a_{j1}, \ldots, a_{jn} \in \boldsymbol{T}$ such that $g_j = a_{j1} \odot f_1 \oplus \cdots \oplus a_{jn} \odot f_n$ and at least one of $a_{j1}, \ldots, a_{jn}$ is not $-\infty$.
Then, in any direction from $x$, the slope of $g_j$ is greater than or equal to that of $f_i$ such that $a_{ji} \not= -\infty$.
This implies that the inequality
\begin{align*}
&~ \operatorname{min}\{ \operatorname{div}(f_1)(x), \ldots, \operatorname{div}(f_n)(x)\}\\
\le&~ \operatorname{min}\{ \operatorname{div}(g_1)(x), \ldots, \operatorname{div}(g_l)(x)\} 
\end{align*}
holds.
As $g_1, \ldots, g_l$ are also generators of $R$ as a $\boldsymbol{T}$-module, by the same argument, the opposite inequality holds.
\end{proof}

Note that for a point of $\Gamma_{\operatorname{par}}$, it is a pole of some $f_i$ if and only if it is a pole of some $g_j$ by this proposition.
This guaranties that the following definition is independent of the choice of generators of $R$ as a $\boldsymbol{T}$-module.

\begin{dfn}
    \label{dfn4-7-1}
\upshape{
As in the setting of Proposition~\ref{prop4-7-1}, let $x_1, \ldots, x_k$ be the poles of $f_1, \ldots, f_n$.
Then we call the nonnegative integer
\begin{align*}
- \sum_{i = 1}^k \operatorname{min} \{ \operatorname{div}(f_1)(x_i), \ldots, \operatorname{div}(f_n)(x_i) \}    
\end{align*}
the \textit{degree} of $R$ and write it as $\operatorname{deg}(R)$.
For $R = \boldsymbol{T}$, let $\operatorname{deg}(\boldsymbol{T}) := 0$ and for $R = \{ -\infty \}$, let $\operatorname{deg}(\{-\infty\}) := - \infty$.
}
\end{dfn}

With this definition, in the context of \cite[Theorem~35]{Haase=Musiker=Yu}, the degree of $R(D)$ coincides with $\operatorname{deg}(D)$, see Proposition~\ref{prop4-8-2}.
Thus, this is a notion that is already partially studied in \cite{Haase=Musiker=Yu}.

In addition, we know that this definition depends on the inclusion $R \hookrightarrow \operatorname{Rat}(\Gamma_{\operatorname{par}})$.
More precisely, for a finitely generated $\boldsymbol{T}$-module $R^{\prime} \subset \operatorname{Rat}(\Gamma_{\operatorname{par}})$, even when $R^{\prime}$ is isomorphic to $R$ as a $\boldsymbol{T}$-module, these may have different degrees.
For example, for $f(x) := x$ with $x \in [-\infty, \infty]$, the $\boldsymbol{T}$-module $R$ generated by $f$ is isomorphic to the $\boldsymbol{T}$-module $R^{\prime}$ generated by $f^{\odot 2}$ as a $\boldsymbol{T}$-module, but $R$ has degree one and $R^{\prime}$ has degree two.

If $\psi : \operatorname{Rat}(\Gamma_{\operatorname{par}}) \to \operatorname{Rat}(\Gamma_{\operatorname{par}})$ is a $\boldsymbol{T}$-algebra isomorphism, then the degree of $R$ equals that of $\psi(R)$, which is a finitely generated $\boldsymbol{T}$-module.

As in the setting of Definition~\ref{dfn4-7-1}, if $x_i$ is a point at infinity, then we can choose a finite point $y_i$ of a ray containing $x_i$ such that the outgoing slopes of $f_1, \ldots, f_n$ at $y_i$ toward $x_i$ coincide with $-\operatorname{div}(f_1)(x_i), \ldots, -\operatorname{div}(f_n)(x_i)$, respectively.
Hence, we can write the degree of $R$ using only finite points, and thus Subsection~\ref{subsection4.6} guaranties us to write it in a trivial algebraic way.

\subsection{Harmonicity and balancing condition}
    \label{subsection4.8}

In this subsection, we give an algebraic description of the balancing condition for tropical curves with parallel rays, as for tropicalized curves in $\boldsymbol{R}^n$, in terms of harmonicity.

We define harmonicity at a point of a rational function on a tropical curve with parallel rays.
This is an analogue of the notion of discrete harmonic functions introduced by Heilbronn in \cite{Heilbronn}.

\begin{dfn}
    \label{dfn4-8-1}
\upshape{
Let $\Gamma_{\operatorname{par}}$ be a tropical curve with parallel rays and $f \in \operatorname{Rat}(\Gamma_{\operatorname{par}}) \setminus \{ -\infty \}$.
For a point $x \in \Gamma_{\operatorname{par}}$, the function $f$ is \textit{harmonic at $x$} if $\operatorname{div}(f)(x) = 0$.
}
\end{dfn}

Note that the harmonicity at a point is a local property, and hence in fact we need not consider the parallelity of rays, and without the parallelity, the harmonicity is already described in \cite[Subsection~2.9]{Amini=Baker=Brugalle=Rabinoff}.

Let $\Gamma_{\operatorname{par}}$ be a tropical curve with parallel rays other than one point and $x \in \Gamma_{\operatorname{par}} \setminus \Gamma_{\operatorname{par}, \infty}$.
For $n := \operatorname{val}(x)$, let $\pi_x : \operatorname{Rat}(\Gamma_{\operatorname{par}}) \twoheadrightarrow \operatorname{Rat}(\Gamma_{\operatorname{par}})_x \cong R_n = ((\boldsymbol{R} \times \boldsymbol{Z}^n) \cup \{ -\infty\}, \boxplus, \boxdot)$ be the natural surjective $\boldsymbol{T}$-algebra homomorphism.
Since $(\boldsymbol{R} \times \boldsymbol{Z}^n, \boxdot = +)$ is a group, the map $\omega : (\boldsymbol{R} \times \boldsymbol{Z}^n, \boxdot = +) \to (\boldsymbol{Z}, +); (a, (i_1, \ldots, i_n)) \mapsto i_1 + \cdots + i_n$ is a group homomorphism.
The following is clear by definition:

\begin{prop}
    \label{prop4-8-1}
In the above setting, for $f \in \operatorname{Rat}(\Gamma_{\operatorname{par}}) \setminus \{ -\infty \}$, the following are equivalent:

$(1)$ $f$ is harmonic at $x$, and 

$(2)$ $\pi_x(f)$ is in the kernel of $\omega$, i.e., $\omega(\pi_x(f)) = 0$.
\end{prop}

Let $\Gamma_{\operatorname{par}}$ be a tropical curve with parallel rays.
For $f_1, \ldots, f_n \in \operatorname{Rat}(\Gamma_{\operatorname{par}}) \setminus \{ -\infty \}$, let $R$ denote the $\boldsymbol{T}$-module generated by $f_1, \ldots, f_n$.

\begin{prop}
    \label{prop4-8-2}
In the above setting, if $f_1, \ldots, f_n$ are harmonic at every point of $\Gamma_{\operatorname{par}} \setminus \Gamma_{\operatorname{par}, \infty}$, then they are extremals of $R$.
Then for $f \in R \setminus \{ -\infty \}$, the following are equivalent:

$(1)$ $f$ is an extremal of $R$,

$(2)$ there exist $i$ and $a \in \boldsymbol{R}$ such that $f = a \odot f_i$, and

$(3)$ $f$ is harmonic at every point of $\Gamma_{\operatorname{par}} \setminus \Gamma_{\operatorname{par}, \infty}$.
\end{prop}

\begin{proof}
Assume that $f_i = g \oplus h$ for $g, h \in R$ and $f_i \not= g$ and $f_i \not= h$.
Then there exists $x \in \Gamma_{\operatorname{par}} \setminus \Gamma_{\operatorname{par}, \infty}$ such that $f_i|_{U(x, \varepsilon)} \not= g|_{U(x, \varepsilon)}, h|_{U(x, \varepsilon)}$ for all $\varepsilon > 0$ and the $\varepsilon$-neighborhood $U(x, \varepsilon)$ of $x$.
In fact, otherwise, for all $y \in \Gamma_{\operatorname{par}} \setminus \Gamma_{\operatorname{par}, \infty}$, there exists $\delta > 0$ such that $f_i|_{U(y, \delta)} = g|_{U(y, \delta)}$ or $f_i|_{U(y, \delta)} = h|_{U(y, \delta)}$ for the $\delta$-neighborhood $U(y, \delta)$ of $y$.
If $f_i|_{U(y, \delta)} = g|_{U(y, \delta)}$ holds for such a $y$, then, taking $\delta$ bigger and bigger, it is verified that the equality $f_i = g \oplus h$ and the assumption above ensure $f_i = g$ on $\Gamma_{\operatorname{par}} \setminus \Gamma_{\operatorname{par}, \infty}$.
Both $f_i$ and $g$ are continuous, this means that $f_i = g$ holds on the whole of $\Gamma_{\operatorname{par}}$.
When $f_i|_{U(y, \delta)} = h|_{U(y, \delta)}$, the same things hold.
Hence $\operatorname{div}(f_i)(x) > \operatorname{div}(g)(x), \operatorname{div}(h)(x)$ hold.
Note that $g = -\infty$ (resp.~$h = -\infty$) implies that $f_i = h$ (resp.~$f_i = g$), and so both $g$ and $h$ cannot be $-\infty$ and define the principal divisors $\operatorname{div}(g)$ and $\operatorname{div}(h)$, respectively.
Since both $g$ and $h$ are in $R$ and $f_1, \ldots, f_n$ are harmonic at $x$, the values $\operatorname{div}(g)(x)$ and $\operatorname{div}(h)(x)$ must be nonnegative.
This contradicts the assumption that $f_i$ is harmonic at $x$.

For the last assertion, $(2) \Longrightarrow (3)$ holds by definition, $(3) \Longrightarrow (1)$ follows from the above argument, and $(1) \Longrightarrow (2)$ is verified by \cite[Proposition~8 and its proof]{Haase=Musiker=Yu}.
\end{proof}

In Proposition~\ref{prop4-8-2} (or slightly weaker, when $f_1, \ldots, f_n$ have no poles in $\Gamma_{\operatorname{par}} \setminus \Gamma_{\operatorname{par}, \infty}$), we know that the degree of $R$ coincides with the sum of degrees of poles of $f_1 \oplus \cdots \oplus f_n$, and also equals the sum of degrees of zeros of $f_1 \oplus \cdots \oplus f_n \oplus M$ with a sufficiently large positive number $M$.

To connect harmonicity to the balancing condition, we consider the map consisting of rational functions on a tropical curve which is harmonic at every point other than points at infinity as in Proposition~\ref{prop4-8-2}.

\begin{prop}
    \label{prop4-8-4}
Let $\Gamma$ be a tropical curve that is not a point and $f_1, \ldots, f_n \in \operatorname{Rat}(\Gamma) \setminus \{ -\infty \}$.
If all $f_1, \ldots, f_n$ are harmonic at every point of $\Gamma \setminus \Gamma_{\infty}$ and the map $\theta : \Gamma \setminus \Gamma_{\infty} \to \boldsymbol{R}^n; x \mapsto (f_1(x), \ldots, f_n(x))$ is injective, then the image $\operatorname{Im} (\theta)$ is the support of the balanced one-dimensional $\boldsymbol{R}$-rational polyhedral complex $\{ \theta(x), \theta(e \setminus \Gamma_{\infty}) \,|\, x \in V(G) \setminus \Gamma_{\infty}, e \in E(G) \}$ in $\boldsymbol{R}^n$ given as follows.
Let $(G, l)$ be the canonical model for $\Gamma$.
The image $\theta(e \setminus \Gamma_{\infty})$ for an edge $e$ of $G$ has the greatest common divisor of the slopes of $f_1, \ldots, f_n$ on $e$ as its multiplicity $m(\theta(e \setminus \Gamma_{\infty}))$.
Moreover, for the surjective $\boldsymbol{T}$-algebra homomorphism $\psi : \overline{\boldsymbol{T}(\boldsymbol{X})}_n \twoheadrightarrow \boldsymbol{T}(f_1, \ldots, f_n)$ given by the correspondence $X_i \mapsto f_i$, the congruence variety $\boldsymbol{V}(\operatorname{Ker}(\psi))$ coincides with $\operatorname{Im}(\theta)$.
\end{prop}

\begin{proof}
Let $x \in V(G) \setminus \Gamma_{\infty}$.
Since each $f_i$ is harmonic at $x$, there exists a sufficiently small positive number $\varepsilon > 0$ such that $\sum_{j = 1}^m (f_i(y_j) - f_i(x)) = 0$ with the points $y_1, \ldots, y_m$ such that $\operatorname{dist}(x, y_j) = \varepsilon$, where $m$ is the valence of $x$.
If $s_{i, j}$ denotes the outgoing slope of $f_i$ at $x$ toward $y_j$, then $f_i(y_j) - f_i(x) = \varepsilon s_{i, j}$, and thus $\sum_{j = 1}^m s_{i, j} = 0$ hold.
Under the assumption that $f_1, \ldots, f_n$ are harmonic at every point of $\Gamma \setminus \Gamma_{\infty}$, by the definitions of $(G, l)$ and $\theta$, the image $\theta(V(G) \setminus \Gamma_{\infty})$ gives $\operatorname{Im}(\theta)$ an $\boldsymbol{R}$-rational polyhedral complex structure in $\boldsymbol{R}^n$.
Since $\theta$ is injective, the correspondence between $y_1, \ldots, y_m$ and $\theta(y_1), \ldots, \theta(y_m)$ is one-to-one.
With the primitive vector $\boldsymbol{v}_j$ from $\theta(x)$ toward $\theta(y_j)$, the coordinates of the point $\theta(y_j)$ coincide with $\theta(x) + (f_1(y_j) - f_1(x), \ldots, f_n(y_j) - f_n(x)) = \theta(x) + \varepsilon (s_{1, j}, \ldots, s_{n, j}) = \theta(x) + \varepsilon m(\theta(e_j \setminus \Gamma_{\infty})) \boldsymbol{v}_j$, where $e_j \in E(G)$ such that $\overline{xy_j} \subset e_j$.
Therefore
\begin{align*}
\sum_{j = 1}^m m(\theta(e_j \setminus \Gamma_{\infty})) \boldsymbol{v}_j = \sum_{j = 1}^m  (s_{1, j}, \ldots, s_{n, j} ) = \boldsymbol{0 }
\end{align*}
hold.
The last assertion follows from Corollary~\ref{cor4-1-2}.
\end{proof}

In Proposition~\ref{prop4-8-4}, the assumption of the injectivity of $\theta$ is not essential, but when we consider the converse as follows, it appears naturally:

\begin{prop}
    \label{prop4-8-5}
Let $\Gamma^{\prime}$ be the support of a balanced one-dimensional $\boldsymbol{R}$-rational polyhedral complex $\Sigma$ in $\boldsymbol{R}^n$ and $\Gamma_{\operatorname{par}}$ the tropical curve with parallel rays whose topology structure coincides with that of the natural compactification $\overline{\Gamma^{\prime}}$ of $\Gamma^{\prime}$ as a tropical curve and whose length of the edge $e$ corresponding to an edge $e^{\prime}$ of $\Sigma$ is that of $e^{\prime}$ multiplied by $\frac{1}{m(e^{\prime})}$, where $m(e^{\prime})$ denotes the multiplicity of $\Sigma$ on $e^{\prime}$, and whose parallelity of rays is given by that of unbounded one-dimensional cells of $\Sigma$ in $\boldsymbol{R}^n$.
Then the composition $f_i := X_i \circ \theta$ of the $i$th coordinate function $X_i : \boldsymbol{R}^n \to \boldsymbol{R}; (x_1, \ldots, x_n) \mapsto x_i$, which is in $\overline{\boldsymbol{T}(\boldsymbol{X})}_n$, and the natural map $\theta : \Gamma_{\operatorname{par}} \setminus \Gamma_{\operatorname{par}, \infty} \to \Gamma^{\prime} \subset \boldsymbol{R}^n$ induces a rational function on $\Gamma_{\operatorname{par}}$ which is harmonic at every point of $\Gamma_{\operatorname{par}} \setminus \Gamma_{\operatorname{par}, \infty}$.
\end{prop}

\begin{proof}
For each edge $e$ of $\Gamma$ and $e^{\prime} := \theta(e \setminus \Gamma_{\operatorname{par}, \infty})$, the slope of $f_i$ on $e$ is that of the $i$th coordinate function $X_i$ of $\boldsymbol{R}^n$ on $e^{\prime}$ multiplied by $m(e^{\prime})$ with the lattice length.
Since $\Sigma$ is $\boldsymbol{R}$-rational, each $f_i$ induces a rational function on $\Gamma_{\operatorname{par}}$ and for any point $x$ of $\Gamma_{\operatorname{par}} \setminus \Gamma_{\operatorname{par}, \infty}$ and the distinct points $y_1, \ldots, y_m$ that have the same sufficiently small distance from $x$,
\begin{align*}
\sum_{j = 1}^m (f_i(y_j) - f_i(x)) &= \sum_{j = 1}^m (X_i(\theta(y_j)) - X_i(\theta(x)))\\
&= \sum_{j = 1}^m m(\theta(e_j \setminus \Gamma_{\operatorname{par}, \infty})) (X_i(z_j) - X_i(\theta(x)) = 0
\end{align*}
hold.
Here $e_j$ is the one-dimensional cell of $\Sigma$ containing $\theta(y_j)$ and $z_j$ is the internal division point that divides the segment $\overline{\theta(y_j) \theta(x)}$ in the ratio $m(e_i \setminus \Gamma_{\operatorname{par}, \infty}) - 1 : 1$.
Note that the last equality follows from the fact that the vectors from $\theta(x)$ to $z_1, \ldots, z_m$ have the same lattice length and $\Sigma$ is balanced.
\end{proof}

\begin{rem}
    \label{rem4-8-1}
\upshape{
Combining Remark~\ref{rem4-2-2} and Propositions~\ref{prop4-5-1}, \ref{prop4-8-4}, \ref{prop4-8-5}, as in the setting of Remark~\ref{rem4-2-2}, if $f_1, \ldots, f_n$ are harmonic at every point other than points at infinity and the map $\theta$ is injective, then the $\boldsymbol{T}$-algebra homomorphism $\operatorname{Rat}(\Gamma^{\prime}) \to \operatorname{Rat}(\Gamma)$ (resp.~$\operatorname{Rat}(\Gamma^{\prime}_{\operatorname{par}}) \to \operatorname{Rat}(\Gamma_{\operatorname{par}})$) is a weight of $\Gamma^{\prime}$ (resp.~$\Gamma^{\prime}_{\operatorname{par}}$) and the weighted realization of $\Gamma^{\prime}_{\operatorname{par}}$ by $f_1, \ldots, f_n$ preserves the weight $\operatorname{Rat}(\Gamma^{\prime}_{\operatorname{par}}) \hookrightarrow \operatorname{Rat}(\Gamma_{\operatorname{par}})$ and the parallelity of rays.
}
\end{rem}

\begin{dfn}
    \label{dfn4-8-3}
\upshape{
Let $\Gamma_{\operatorname{par}}$ and $\Gamma^{\prime}_{\operatorname{par}}$ be tropical curves with parallel rays, respectively.
Let $\psi : \operatorname{Rat}(\Gamma^{\prime}_{\operatorname{par}}) \hookrightarrow \operatorname{Rat}(\Gamma_{\operatorname{par}})$ be a weight of $\Gamma^{\prime}_{\operatorname{par}}$.
Then $\Gamma^{\prime}_{\operatorname{par}}$ (or $\operatorname{Rat}(\Gamma^{\prime}_{\operatorname{par}})$) has a \textit{realization balanced with respect to $\psi$} if there exists a finite generating set $\{ f^{\prime}_1, \ldots, f^{\prime}_n \} \subset \operatorname{Rat}(\Gamma^{\prime}_{\operatorname{par}}) \setminus \{ -\infty \}$ of $\operatorname{Rat}(\Gamma^{\prime}_{\operatorname{par}})$ as a semifield over $\boldsymbol{T}$ such that all $\psi(f^{\prime}_1), \ldots, \psi(f^{\prime}_n)$ are harmonic at every point of $\Gamma_{\operatorname{par}} \setminus \Gamma_{\operatorname{par}, \infty}$.
For such $\{ f^{\prime}_1, \ldots, f^{\prime}_n \}$, the image of $\theta : \Gamma_{\operatorname{par}} \setminus \Gamma_{\operatorname{par}, \infty} \to \boldsymbol{R}^n; x \mapsto (\psi(f^{\prime}_1)(x), \ldots, \psi(f^{\prime}_n)(x))$ is called a \textit{balanced weighted realization of $\Gamma^{\prime}_{\operatorname{par}}$ of degree $\operatorname{deg}(R)$}, where $R$ is the $\boldsymbol{T}$-module generated by $\psi(f^{\prime}_1), \ldots, \psi(f^{\prime}_n)$, which coincides with the image $\psi(R^{\prime})$ of the $\boldsymbol{T}$-module $R^{\prime}$ generated by $f^{\prime}_1, \ldots, f^{\prime}_n$ by $\psi$.
}
\end{dfn}

By Proposition~\ref{prop4-8-2} and \cite[Proposition~3.3.10 and its proof]{Maclagan=Sturmfels}, the following holds:

\begin{thm}
    \label{thm4-8-1}
In the setting of Definition~\ref{dfn4-8-3}, assume that $\theta$ gives a balanced weighted realization of $\Gamma^{\prime}_{\operatorname{par}}$.
If $n = 2$ and $\Gamma^{\prime}_{\operatorname{par}}$ is not a point, then there exists a tropical Laurent polynomial $F \in \boldsymbol{T}[\boldsymbol{X}^{\pm}]_2 \setminus \{ -\infty \}$ whose tropical hypersurface $\boldsymbol{V}(F)$ coincides with $\operatorname{Im}(\theta)$ respecting the weight $\psi$ and the corresponding Newton polygon is unique up to translations.
When $F$ is taken as a tropical polynomial of minimum degree, the degree of $F$ (as a tropical polynomial) coincides with that of $R$.
\end{thm}

\subsection{Intersection numbers}
    \label{subsection4.9}

Let $\Gamma_{\operatorname{par}, i}$ and $\Gamma^{\prime}_{\operatorname{par}, i}$ be tropical curves with parallel rays.
Assume that $\psi_i : \operatorname{Rat}(\Gamma^{\prime}_{\operatorname{par}, i}) \hookrightarrow \operatorname{Rat}(\Gamma_{\operatorname{par}, i})$ is a weight of $\Gamma^{\prime}_{\operatorname{par}, i}$ and $f^{\prime}_1, \ldots, f^{\prime}_n \in \operatorname{Rat}(\Gamma^{\prime}_{\operatorname{par}, 1}) \setminus \{ -\infty \}$ (resp.~$g^{\prime}_1, \ldots, g^{\prime}_m \in \operatorname{Rat}(\Gamma^{\prime}_{\operatorname{par}, 2}) \setminus \{ -\infty \}$) are generators of $\operatorname{Rat}(\Gamma^{\prime}_{\operatorname{par}, 1})$ (resp.~$\operatorname{Rat}(\Gamma^{\prime}_{\operatorname{par}, 2})$) as a semifield over $\boldsymbol{T}$.
When $n = m$, for the weighted realizations $\operatorname{Im}(\theta)$ and $\operatorname{Im}(\phi)$ with $\theta : \Gamma_{\operatorname{par}, 1} \setminus \Gamma_{\operatorname{par}, 1, \infty} \to \boldsymbol{R}^n; x \mapsto (\psi_1(f^{\prime}_1)(x), \ldots, \psi_1(f^{\prime}_n)(x))$ and $\phi : \Gamma_{\operatorname{par}, 2} \setminus \Gamma_{\operatorname{par}, 2, \infty} \to \boldsymbol{R}^n; y \mapsto (\psi_2(g^{\prime}_1)(y), \ldots, \psi_2(g^{\prime}_n)(y))$, we can consider these intersections, i.e., points of $\operatorname{Im}(\theta) \cap \operatorname{Im}(\phi)$.
In this subsection, we consider the case of $n = 2$.

\begin{dfn}[{cf.~\cite[Definitions~3.6.5 and 3.6.11]{Maclagan=Sturmfels}}]
    \label{dfn4-9-1}
\upshape{
In the above setting, for a point $z$ of $\boldsymbol{R}^2$, the weighted realizations $\operatorname{Im}(\theta)$ and $\operatorname{Im}(\phi)$ \textit{intersect transversally at $x$} if

$(1)$ $z$ is a point of the intersection $\operatorname{Im}(\theta) \cap \operatorname{Im}(\phi)$,

$(2)$ both $x := \theta^{-1}(z)$ and $y := \phi^{-1}(z)$ are two-valent points,

$(3)$ $x$ (resp.~$y$) is not a zero or a pole of either $\psi_1(f^{\prime}_1)$ or $\psi_1(f^{\prime}_2)$ (resp.~$\psi_2(g^{\prime}_1)$ or $\psi_2(g^{\prime}_2)$), and

$(4)$ the integer vectors $\boldsymbol{d}_1 := (\operatorname{sl}_{x}(\psi_1(f^{\prime}_1)), \operatorname{sl}_{x}(\psi_1(f^{\prime}_2)))$ and $\boldsymbol{d}_2 := (\operatorname{sl}_{y}(\psi_2(g^{\prime}_1)), \operatorname{sl}_{y}(\psi_2(g^{\prime}_2)))$, and $\boldsymbol{d}_1$ and $-\boldsymbol{d}_2$, are not parallel, respectively, where $\operatorname{sl}_{x}(\psi_1(f^{\prime}_i))$ (resp.~$\operatorname{sl}_{y}(\psi_2(g^{\prime}_i))$) denotes the outgoing slope of $\psi_1(f^{\prime}_i)$ (resp.~$\psi_2(g^{\prime}_i)$) at $x$ (resp.~$y$) in any direction.

When $\operatorname{Im}(\theta)$ and $\operatorname{Im}(\phi)$ intersect transversally at $z$, the absolute value of the determinant of the matrix $(\t \boldsymbol{d}_1 \t \boldsymbol{d}_2)$ is called the \textit{intersection number with multiplicities} of $\operatorname{Im}(\theta)$ and $\operatorname{Im}(\phi)$ at $z$, where $\t \boldsymbol{d}_i$ denotes the transpose of the vector $\boldsymbol{d}_i$.
}
\end{dfn}

Note that in $(4)$ of Definition~\ref{dfn4-9-1}, there are two choices of outgoing directions at $x$ (resp.~$y$) by $(2)$, but by $(3)$, the integer vector $\boldsymbol{d}_1$ (resp.~$\boldsymbol{d}_2$) is unique up to sign, and hence $(4)$ is independent of the choice of these directions.

When both $\operatorname{Im}(\theta)$ and $\operatorname{Im}(\phi)$ are balanced, by Theorem~\ref{thm4-8-1}, there are tropical polynomials $F_1, F_2 \in \boldsymbol{T}[\boldsymbol{X}]_2$ such that $\boldsymbol{V}(F_1) = \operatorname{Im}(\theta)$ and $\boldsymbol{V}(F_2) = \operatorname{Im}(\phi)$ respecting the weights $\psi_1$ and $\psi_2$ and the degrees, respectively.
Then these intersection numbers coincide with the classical intersection numbers for $\boldsymbol{V}(F_1)$ and $\boldsymbol{V}(F_2)$.
That is, for a point $z$ of $\boldsymbol{R}^2$, $\operatorname{Im}(\theta)$ and $\operatorname{Im}(\phi)$ intersect transversally at $z$ if and only if so do $\boldsymbol{V}(F_1)$ and $\boldsymbol{V}(F_2)$, and then the intersection number with multiplicities of $\operatorname{Im}(\theta)$ and $\operatorname{Im}(\phi)$ is equal to that of $\boldsymbol{V}(F_1)$ and $\boldsymbol{V}(F_2)$.

The following is a consequence of the classical intersection theory for tropical varieties:

\begin{prop}
    \label{prop4-9-1}
In the above setting, if $\operatorname{Im}(\theta)$ and $\operatorname{Im}(\phi)$ intersect transversally at every point of $\operatorname{Im}(\theta) \cap \operatorname{Im}(\phi)$, then the sum of all intersection numbers with multiplicities is at most the product of the degrees of the realizations $\operatorname{Im}(\theta)$ and $\operatorname{Im}(\phi)$.
\end{prop}

More precisely, for example, see \cite[Subsections~3.4, 3.6, 4.6, 6.7]{Maclagan=Sturmfels} and \cite{Jensen=Yu}.

\appendix
\section{}
    \label{appendixA}

This appendix is devoted to the study on $R_n = ((\boldsymbol{R} \times \boldsymbol{Z}^n) \cup \{ -\infty \}, \boxplus, \boxdot)$ in Subsection~\ref{subsection4.6}.
Recall here that $R_0$ is defined as $\boldsymbol{T}$.

\begin{prop}
    \label{prop-appendix1}
The semifield $R_n$ over $\boldsymbol{T}$ is generated by one element when $n = 1, 2$ and by two elements when $n \ge 3$.
For each $n$, this is minimum among the numbers of elements of generating sets of $R_n$ as a semifield over $\boldsymbol{T}$.
\end{prop}

\begin{proof}
In the case $n = 1$, it is clear, that is, $(0, 1) \in \boldsymbol{R} \times \boldsymbol{Z}$ is a generator of $R_n$ and it is minimum.

For $n = 2$, the element $(0, (1, -1)) \in \boldsymbol{R} \times \boldsymbol{Z}^2$ is a minimum generator of $R_2$.
In fact, since clearly $(0, (1, 0))$ and $(0, (0, 1))$ generate $R_2$, and $(0, (1, -1)) \boxplus (0, (0, 0)) = (0, (1, 0))$ and $(0, (1, -1))^{\boxdot (-1)} \boxplus (0, (0, 0)) = (0, (-1, 1)) \boxplus (0, (0, 0)) = (0, (0, 1))$ hold, it is true.

Let $n \ge 3$.
In this case, the elements $\boldsymbol{v}_1 := (0, (1, 0, 1, \ldots, 1))$ and $\boldsymbol{v}_2 := (0, (0, 1, 1, 2, \ldots, n - 2))$ are generators of $R_n$.
It follows from the fact that 
\begin{align*}
\boldsymbol{e}_1^{\prime} := (0, (1, 0, \ldots, 0)), \boldsymbol{e}_2^{\prime} := (0, (0, 1, 0, \ldots, 0)), \ldots, \boldsymbol{e}_n^{\prime} := (0, (0, \ldots, 0, 1))    
\end{align*}
generate $R_n$ and 
\begin{align*}
&\boldsymbol{v}_1 \boxdot \boldsymbol{v}_2^{\boxdot (-1)} \boxplus (0, (0, \ldots, 0)) = \boldsymbol{e}_1^{\prime},\\
&\boldsymbol{v}_2 \boxdot \left( \boldsymbol{v}_1 \boxdot (\boldsymbol{e}_1^{\prime})^{\boxdot (-1)} \right)^{\boxdot (-(n - 2))} \boxplus (0, (0, \ldots, 0)) = \boldsymbol{e}_2^{\prime},\\
& \left( \boldsymbol{v}_1 \boxdot (\boldsymbol{e}_1^{\prime})^{\boxdot (-1)} \right)^{\boxdot 2} \boxdot \left( \boldsymbol{v}_2 \boxdot (\boldsymbol{e}_2^{\prime})^{\boxdot (-1)} \right)^{\boxdot (-1)} \boxplus (0, (0, \ldots, 0)) = \boldsymbol{e}_3^{\prime}
\end{align*}
and by repeating this operation, the semifield over $\boldsymbol{T}$ generated by $\boldsymbol{v}_1$ and $\boldsymbol{v}_2$ contains $\boldsymbol{e}_1^{\prime}, \ldots, \boldsymbol{e}_n^{\prime}$.

The remaining is to show that two is the minimum number.
Since there exist surjective $\boldsymbol{T}$-algebra homomorphisms $R_{3 + m} \twoheadrightarrow R_3$ for any positive integer $m$ such as the ignoring maps, it suffices to see the case $n = 3$. 
Assume that $\boldsymbol{w} \in \boldsymbol{R} \times \boldsymbol{Z}^3$ generates $R_3$ as a semifield over $\boldsymbol{T}$.
If one of the $\boldsymbol{Z}^3$-components of $\boldsymbol{w}$ is zero, then, because all elements of the semifield over $\boldsymbol{T}$ generated by $\boldsymbol{w}$ have zero as that component, $\boldsymbol{w}$ cannot generate $R_3$ with operations $\boxplus$ and $\boxdot$.
Hence, all $\boldsymbol{Z}^3$-components of $\boldsymbol{w}$ are nonzero.
More precisely, these must be one or minus one.
So all possibilities are $(1, 1, 1), (1, 1, -1), (1, -1, -1), (-1, -1, -1)$ and these permutations.
But any element of $R_3$ with one of these elements as the $\boldsymbol{Z}^3$-component cannot generate $R_3$.
In conclusion, there exists no such $\boldsymbol{w}$.
\end{proof}

\begin{prop}
    \label{prop-appendix2}
For any $n \in \boldsymbol{Z}_{\ge 1}$, the Krull dimension of $R_n$ is two.
\end{prop}

\begin{proof}
If $\Gamma$ is the tropical curve obtained from $n$-copies of a unit segment $[0, 1]$ by gluing them at all $0$s to be connected, then it has an $n$-valent point $x$.
Thus there exist surjective $\boldsymbol{T}$-algebra homomorphisms $\operatorname{Rat}(\Gamma) \twoheadrightarrow \operatorname{Rat}(\Gamma)_x \cong R_n$ and $R_n \twoheadrightarrow R_1$ and $R_1 \twoheadrightarrow R_0 = \boldsymbol{T}$.
Since $\operatorname{Rat}(\Gamma)$ has Krull dimension two by \cite[Theorem~3.25]{JuAe=Nakajima} and $\boldsymbol{T}$ has Krull dimension one, $R_n$ has Krull dimension one or two by \cite[Proposition~2.13]{Joo=Mincheva}.
As $R_1$ is a cancellative, i.e., for $f, g, h \in R_1$, if $f \boxdot g = f \boxdot h$, then $f = -\infty$ or $g = h$ holds, and totally ordered, i.e., for $f, g \in R_1$, one of $f \boxplus g = f$ or $f \boxplus g = g$ holds, $\boldsymbol{B}$-algebra by definition, \cite[Proposition~2.10]{Joo=Mincheva2} and \cite[Proposition~2.13]{Joo=Mincheva} guarantee that $R_1$ has Krull dimension two.
Thus, for any $n \ge 1$, the semifield $R_n$ over $\boldsymbol{T}$ has Krull dimension two.
\end{proof}

\begin{prop}
    \label{prop-appendix3}
For $n, m \ge 0$, if there exists a surjective $\boldsymbol{T}$-algebra homomorphism $\psi : R_n \twoheadrightarrow R_m$, then $n \ge m$ holds.
In particular, if $\psi$ is a $\boldsymbol{T}$-algebra isomorphism, then $n$ coincides with $m$.
\end{prop}

\begin{proof}
When $n, m \ge 1$, since $\psi$ induces a surjective group homomorphism between free abelian groups $(\boldsymbol{Z}^n, \boxdot = +) \twoheadrightarrow (\boldsymbol{Z}^m, \boxdot = +)$, the statement is clear.
The case of $m = 0$ is also clear.
If $n = 0$ and $m \ge 1$, then Proposition~\ref{prop-appendix2} and \cite[Proposition~2.13]{Joo=Mincheva} give a contradiction.
\end{proof}

\end{document}